%% file: main.tex
\documentclass{amsart}
\input{preamble.tex}
\input{tikzmacros.tex}

\usepackage{multicol}
\usepackage{mathrsfs}
\newcommand{\uWeyl}{\operatorname{\mathbf{W}}}
\newcommand{\aWeyl}{\operatorname{\mathbf W}^\infty}
\newcommand{\Heis}{\mathcal{H}}
\newcommand{\Wact}{\operatorname{\mathbf{2Weyl}}}

\newcommand{\Id}{\operatorname{id}}
\newcommand{\Ind}{\operatorname{Ind}}
\newcommand{\Res}{\operatorname{Res}}
\newcommand{\End}{\operatorname{End}}
\newcommand{\EndCat}{\operatorname{End}_{\mathbf{Cat}}}

\newcommand{\eone}{\End_{\Wact}(\mathbbm{1}_n)}

\newcommand{\boxalg}{\mathcal{B}}
\newcommand{\boxy}[1]{b_{#1}}
\newcommand{\bdesc}[2]{\boxy{[#1;#2]}}
\newcommand{\projalg}{\mathcal{C}}

\newcommand{\TL}{\operatorname{TL}}
\newcommand{\TLn}{\TL_n}

\newcommand{\ptr}{\operatorname{ptr}}
\newcommand{\tleq}[1][\bullet]{\overset{#1}{=}}
\renewcommand{\mod}[3]{{}_{#1}(#2)_{#3}}
\newcommand{\T}{\mathbf{TL}}
\newcommand{\KT}{\mathcal{T\!L}}

\newcommand*\circled[1]{\tikz[baseline=(char.base)]{
            \node[shape=circle,draw,very thick,inner sep=2pt] (char) {${#1}$};}}

\newcommand{\kimage}[1][k]{C_{#1}}

\newcommand{\ckme}{e'}

\newcommand{\subalg}{\mathcal{K}}

\newcommand{\subjectclass}[1]{
\renewcommand{\thefootnote}{} 
\footnote{2020\ \textit{Mathematics Subject Classification:} \textrm{#1}}
\addtocounter{footnote}{-1}
\renewcommand{\thefootnote}{\arabic{footnote}} }

\begin{document}

\title{The Temperley-Lieb tower and the Weyl algebra}

\subjectclass{16D90}

\date{\today}

\author{Matthew Harper}
\address{Department of Mathematics \\
Michigan State University \\
619 Red Cedar Rd.\\
East Lansing, MI 48824\\
USA
}
\email{mrhmath@proton.me}

\author{Peter Samuelson}
\address{Department of Mathematics \\
University of California, Riverside \\
900 University Ave.\\
Riverside, CA 92521\\
USA
}
\email{psamuels@ucr.edu}

\begin{abstract}
We define a monoidal category $\uWeyl$ and a closely related 2-category $\Wact$  using diagrammatic methods. We show that $\Wact$ acts on the category $\T :=\bigoplus_n \TLn\mathrm{-mod}$ of modules over Temperley-Lieb algebras, with its generating 1-morphisms acting by induction and restriction. The Grothendieck groups of $\uWeyl$ and a third category we define $\aWeyl$ are closely related to the Weyl algebra. We formulate a sense in which $K_0(\aWeyl)$ acts asymptotically on $K_0(\T)$.
\end{abstract}

\maketitle

\tableofcontents
\section{Introduction}
A sequence $\{A_n\mid n \in \mathbb{N}\}$ of algebras is called a \emph{tower of algebras} if it is equipped with algebra homomorphisms $\iota_{m,n}: A_m \otimes A_n \to A_{m+n}$ satisfying an obvious associativity condition. Towers of algebras arise quite naturally in representation theory -- the prototypical example is the sequence $\Bbbk S_n$ of group algebras of the symmetric groups. The representation theory of algebras in a tower is most naturally studied simultaneously; more precisely, the category $\mathcal A := \bigoplus_n A_n\mathrm{-mod}$ has natural endofunctors induction ($\Ind$) and restriction ($\Res$), both along the maps $\iota:A_n \to A_{n+1}$. The monoidal subcategory of $\EndCat(\mathcal{A})$ generated by induction and restriction tends to be  interesting and deserves careful study. 

In the example of the tower of symmetric groups, Khovanov \cite{Kho14} used a diagrammatic construction to define the \emph{Heisenberg category} $\mathcal{H}^K$ and proved that there is a monoidal functor $\mathcal{H}^K \to \EndCat(\bigoplus_n S_n\mathrm{-mod})$. An action in this context maps an object $x\in \mathcal{H}^K$ to a functor between categories of symmetric group representations. He proved that the Heisenberg algebra injects into the Grothendieck group of $\mathcal{H}^K$ and conjectured they were isomorphic -- this conjecture was proved recently in \cite{BSW23}. 

Other invariants of the Heisenberg category have also been attracting
recent attention. First, recall the \emph{trace} of a category $\mathcal{C}$ is defined as 
\[
\mathrm{Tr}(\mathcal C) := \frac{\mathbbm{k}\{\oplus_X \mathrm{End}(X) \mid X \in Ob(\mathcal{C})\}}{\{f\circ g - g \circ f\}}\,.
\]
If $\mathcal C$ is monoidal, then its trace is an algebra, and there is a \emph{Chern character} (algebra) map $K_0(\mathcal C) \to \mathrm{Tr}(\mathcal C)$ sending $[X] \mapsto \mathrm{Id}_X$. In \cite{CLLS18}, Cautis, Lauda, Licata, and Sussan showed the trace of the Heisenberg category is the $\mathcal W_{1+\infty}$ algebra. This provides an example of a category where the Chern character map is not an isomorphism; in this example, the trace is a much bigger, more interesting algebra. 

Khovanov's Heisenberg category was deformed to the quantum Heisenberg category $\mathcal{H}_q^{LS}$ using Hecke algebras by Licata and Savage in \cite{LS13} and this category was generalized further to $\mathcal{H}_{q,k}^{BSW}$ by Brundan, Savage, and Webster in \cite{BSW20}. The traces of these categories were computed in \cite{CLLSS18} and \cite{MS22}, and in both cases they are closely related to the \emph{elliptic Hall algebra} of Burban and Schiffmann \cite{BS12}. 
Finally, in \cite{BSW20a} Brundan, Savage, and Webster showed that the Heisenberg category $\mathcal{H}_q^{BSW}$ can be used to construct 2-representations of Kac-Moody 2-categories, which was an a-priori more difficult problem.

This history of interesting results involving Heisenberg categories motivates the goal of the present paper:  study symmetries of the representation category of the tower of Temperley-Lieb algebras $\TL_n$. Some progress towards this goal was made by Quinn in \cite{Quinn}, who gave a partial diagrammatic description of a monoidal category which acts on
$\T := \bigoplus_n \TL_n\mathrm{-mod}$.

It turns out that the construction of a category acting on $\T$ is more subtle than one might expect from the Heisenberg category (for reasons we discuss later). Our construction actually involves three categories: a ``universal'' monoidal category $\uWeyl$, a 2-category $\Wact$, and an ``asymptotic'' monoidal category $\aWeyl$. The universal category is used to construct the other two categories by imposing different sets of relations on endomorphisms of the tensor unit $\mathbbm{1}$. In other words, there are functors from $\uWeyl$ to $\Hom_{\Wact}(n,m)$ and $\aWeyl$ which quotient out certain relations on morphisms. The resulting 2-category acts on $\T$ with its generating 1-morphisms acting by induction and restriction. We expect that $\aWeyl$ acts 
``asymptotically'' on $\T$, but we leave the details of this to future work. In Section \ref{sec:kzero}, we discuss an asymptotic action of this category under the $K_0$ functor.

Let $\uWeyl'$ be the additive $\mathbb{C}(q)$-linear strict monoidal category generated by the objects $Q_-$ and $Q_+$. Morphisms in $\uWeyl'$ are generated by oriented cups and caps, a disoriented cupcap, and endomorphisms of $\mathbbm{1}$ given by integer-labeled boxes. Relations in this category are given in \eqref{eq:AlternatingCupCaps}-\eqref{eq:MoveBubbles}. 

A key structural difference between the Heisenberg and Weyl categories is the existence of idempotent endomorphisms of the monoidal unit of $\uWeyl'$ called ``idempotent bubbles''. The image of such an idempotent is a subobject of $\mathbbm{1}$ in the Karoubi completion -- the universal Weyl category -- denoted $\uWeyl$. These bubbles are expressible in terms of the generating morphisms in $\uWeyl$, but there are no analogous morphisms in the present definition(s) of the Heisenberg category. 

\begin{Thm}[Proposition \ref{prop:WeylIso}]
	There is an isomorphism in $\uWeyl$:
	\begin{align*}
		Q_-\otimes Q_+\cong \left( Q_+\otimes Q_- \right) \oplus C_0\,.
	\end{align*}
\end{Thm}
This relation motivates the name for our category, since it is very closely related to the relation $xy = yx + 1$ in the Weyl algebra.

\begin{rem}
	The class of the object $C_0$ is an idempotent $[C_0]\neq[\mathbbm{1}]$ in the Grothendieck group of $\uWeyl$. However, $[C_0]$ does \emph{not} commute with $[Q_+]$ or $[Q_-]$. For example, see equation \eqref{eq:XC}.
\end{rem}  

\begin{rem}
    Based on the results of the present paper, we expect there exists an extension of the Heisenberg category (or, 2-category) that contains idempotent bubbles that act by projection onto isotypic components -- finding the correct definition of this category seems to be an interesting but challenging question that we hope to return to in the future. If we identify $S_n\mathrm{-mod}$ with $K^{{\bC^\times \times \bC^\times}}(\mathrm{Hilb}_n(\bC^2))$, then we expect these idempotent bubbles should act geometrically as restriction of sheaves to fixed points. {However, it is not clear whether there is a direct relation between the Heisenberg category and any of the Weyl categories introduced here -- Remark \ref{rem:fromheis} shows that the standard surjection from the Hecke algebra to the Temperley-Lieb algebra does not extend to a functor from the Heisenberg category to $\uWeyl$.}
\end{rem}

Another difference between the Heisenberg category and the category in our construction is the necessity to use a 2-categorical approach to define an action on a category of modules. In particular, the relations that naturally appear between morphisms in $\uWeyl$ have constants that depend on $n$ (i.e. they depend on which summand of $\T = \bigoplus_n \TLn\mathrm{-mod}$ the morphisms are acting on). These constants are encoded into the integer-labeled boxes and 
act as rational functions whose denominators vanish for certain values of $n$, which means the action of an $m$-labeled box is only defined when $n$ is large enough.
We therefore construct a 2-category $\Wact$ before considering an action on $\T$. Each morphism category in $\Wact$ is defined as a quotient of $\uWeyl$. Our first main result is the construction of this 2-category $\Wact$ which acts on $\T$ via induction and restriction.

\begin{Thm}[{Theorem \ref{thm:rep}}]	There is a well-defined functor $\rho:\Wact\to\End(\T)$ which maps $Q_-$ to restriction, $Q_+$ to induction, and $C_k$ to the projection onto the  isotypic simple component $W^n_{n-2k}$ in $\TLn\mathrm{-mod}$ for each $n$.
\end{Thm}

Using this functor we describe spanning sets and bases of certain endomorphism spaces of 1-morphisms of $\Wact$. 

We also define a quotient of $\uWeyl$ called the \emph{asymptotic Weyl category} $\aWeyl$. We give a conjectural description of a subalgebra $K_0(\aWeyl)'\subset K_0(\aWeyl)$ generated by the classes of $Q_+, Q_-$, and the $C_k$  by defining an algebra $\subalg$ by generators and relations which surjects onto this subalgebra of $K_0(\aWeyl)'$ (see Proposition \ref{prop:K0surj} for a precise statement). 

The monoidal category $\aWeyl$ is called ``asymptotic'' because we show $\subalg$ ``acts asymptotically'' on $K_0(\T)$. Our notion of an \emph{aysmptotic representation} is defined as an increasingly-filtered algebra acting on a decreasingly-filtered vector space, which defines an action of an algebra element on the $n$-th filtered component as $n\to \infty$, see Definition \ref{defn:asympt}. This action is defined such that it respects composition in the algebra (see Propostion \ref{prop:asymptoticrep}). It is in this way that $\subalg$ ``acts'' on $K_0(\T)$ despite $\uWeyl$ and $\aWeyl$ not admitting actions  
on $\T$ in the usual sense. We expect that there is a categorification of this notion of asymptotic representation which lifts to $\aWeyl$ (or some subquotient category) acting asymptotically on $\T$.

There are a number of avenues we intend to pursue in future work. In particular, we would like to describe the Grothendieck category of the 2-category $\Wact$, which will require stronger basis theorems than we currently provide. We also plan to use skein-theoretic techniques and results to describe the trace of $\Wact$, which we expect to be closely related to Cherednik's \emph{double affine Hecke algebra} for $\mathfrak{sl}_2$. 

We hope to relate $\uWeyl$, $\aWeyl$, and $\Wact$ to other categories appearing in categorical representation theory. Since any object $V$ in a monoidal category generates a tower of algebras as $\End(V^{\otimes n})$, the asymptotic categorical representation theory of these may also be interesting.

An outline of the paper is as follows. We include a table of notation in Section \ref{sec:notation}. In Section \ref{sec:TL}, we provide background results about Temperley-Lieb algebras. We give a graphical description of $\uWeyl$ in Section \ref{sec:weyldef}, and prove some diagrammatic relations. In Section \ref{sec:Wacts} we construct the 2-category $\Wact$, and we prove it acts on $\T$ in Section \ref{sec:Wactsacts}. In Section \ref{sec:bases} we prove some basis theorems for morphisms spaces in $\uWeyl$ and $\Wact$. Finally, in Section \ref{sec:kzero} we give a conjectural description of the Grothendieck group of $\aWeyl$ and show that it acts asymptotically on $K_0(\T)$. The appendix contains a summary of some of Quinn's results from \cite{Quinn}.\\

\noindent \textbf{Acknowledgments:} The authors would like to thank Aaron Lauda, Tony Licata, Alistair Savage, and Ben Webster for conversations that were very helpful in the preparation of this paper. The authors would also like to thank an anonymous referee for helpful comments and suggestions. The second author would like to thank Lauda for a (still unanswered) question which partially motivated this work. The work of the second author is partially supported by a Simons Collaboration Grant.

\section{Table of notation}\label{sec:notation}
Here we include an index of notation, along with the section where the notation is introduced.
\begin{multicols}{2}
\noindent
    $\boxalg$ box algebra, Sec. \ref{ssec:boxalg}\\
    $\projalg$ algebra of orthogonal idempotents, Sec. \ref{ssec:endalg}\\
    $\projalg_{n}$ algebra of $\floor{n/2}$ orthogonal idempotents, Sec. \ref{ssec:endalg}\\
    $\projalg_{n,s}$, Sec. \ref{ssec:endalg}\\    
    $C_k$ image of idempotent bubble $\circled{k}\in\End(\mathbbm{1})$, Sec. \ref{sec:weyldef} \\
    $\cup_i$, $\cap_i$ cup and cap in $\TLn$, Sec. \ref{sec:TL} \\
    $e_{p,r}$ matrix element in $\TL_n\mathrm{-mod}$, Sec. \ref{ssec:RepTLn}\\
    $e_i$ idempotent $\TLn$ generator, Sec. \ref{sec:TL} \\
    $e_i'$ quasi-idempotent $\TLn$ generator, Sec. \ref{sec:TL} \\
    $\tleq[n]$ equality in $\Hom_{\End(\TL)}(n,-)$, Rem. \ref{rem:tleq}\\
    $\tleq$ equality in $\Hom_{\End(\TL)}(n,-)$ for $n\geq 0$, Rem. \ref{rem:tleq}\\
    $f^{(n)}$ Jones-Wenzl idempotent, Sec. \ref{sec:TL}\\
    $\subalg$, Sec. \ref{ssec:K0}\\
    $\nabla$ null object, Sec. \ref{sec:Wactsacts} \\
    $P^n$ paths in $\mathsf{A}_{n+1}$ graph, Sec. \ref{ssec:RepTLn}\\
    $P^n_m$ paths in $\mathsf{A}_{n+1}$ graph ending at $m$, Sec. \ref{ssec:RepTLn}\\
    $Q_+$, $Q_-$ tensor generators, Sec. \ref{sec:weyldef} \\
    $Q_\epsilon$ tensor product object, Sec. \ref{sec:weyldef}
    \\
    $Q_{\oplus}$, $Q_\ominus$ images of idempotent cupcaps, Rem. \ref{rem:DoubledObjects} \\
    $\rho$ representation $\Wact\to\End(\T)$, Sec. \ref{sec:Wactsacts}\\
    $\rho'$ pre-completion of $\rho$, Sec. \ref{sec:Wactsacts}\\
    $\TLn$ Temperley-Lieb algebra, Sec. \ref{sec:TL} \\
    $\T=\bigoplus_n \TLn\mathrm{-mod}$, Sec. \ref{sec:Wactsacts} \\
    $\mathcal{TL}=K_0(\TL)$, Sec. \ref{ssec:asymp} \\  
    $v_p$, $\check{v}_p$ vectors in $\TLn\mathrm{-mod}$ for $p\in P^n$, Appx. \ref{sec:app} \\
    $\uWeyl$ universal Weyl category, Sec. \ref{sec:weyldef} \\
    $\uWeyl'$ pre-completion of $\uWeyl$, Sec. \ref{sec:weyldef} \\
    $\Wact$ Weyl 2-category, Sec. \ref{sec:Wacts}
    \\
    $\Wact'$ pre-completion of $\Wact$, Sec. \ref{sec:Wacts}
    \\
    $\aWeyl$ asymptotic Weyl category, Sec. \ref{sec:kzero} \\
    $W^n_m$ irreducible representation of $\TL_n$, Sec. \ref{ssec:RepTLn}
\end{multicols}

\section{The Temperley-Lieb Algebra}\label{sec:TL}
Fix a formal parameter $q$ and ground field $\bC(q)$. Let $V$ be the 2-dimensional irreducible representation of $U=U_q(\mathfrak{sl}_2)$ and $W=V\wedge V$, which is a direct summand of $V^{\otimes 2}$. Let $\TLn=\operatorname{End}_U(V^{\otimes n})$ be the Temperley-Lieb algebra on $n\geq0$ strands. The algebra is generated by quasi-idempotents $\ckme_i$ for $i=1,\dots, n-1$, where
\begin{align}
     (\ckme_i)^2=(q+q^{-1})\ckme_i=[2]\ckme_i\,, && \ckme_i\ckme_{i\pm1}\ckme_i=\ckme_i\,,
\end{align}
and $[2]=q+q^{-1}$\,. More generally, we write $[n]$ for $\dfrac{q^n-q^{-n}}{q-q^{-1}}$\,. We denote the idempotent form of these generators by $e_i=\ckme_i/[2]$, which satisfy $e_i^2=e_i$. The Temperley-Lieb diagram for the generator $\ckme_i$ is the cupcap over strands $i$ and $i+1$. We use the convention that strand 1 is \emph{on the right} and strand $n$ is \emph{on the left}. For example, the element $e_1\in\TL_4$ is given in the diagram below.
\begin{align}
	\TL_4\ni \ckme_1=~\tikzc{
		\vseg{-2}{0}\vseg{-2}{2}
		\vseg{0}{0}\vseg{0}{2}
		\ncap{2}{0}\ncup{2}{4}
	}
\end{align}
The natural inclusion $\TL_n \hookrightarrow \TL_{n+1}$ is the $m=1$ specialization of the structure maps of the Temperley-Lieb tower of algebras, which are the following algebra maps:
\begin{align}
	\tikzc{
		\vseg{-1}{2}
		\vseg{0}{2}\vseg{1}{2}
		\vseg{-1}{0}
		\vseg{0}{0}\vseg{1}{0}
		\drawsquaretext[\phantom{f}$x$\phantom{f}]{0}{2}
	}~\mapsto~
	\tikzc{
		\vseg{-2}{0}\vseg{-2}{2}
		\vseg{-1}{2}
		\vseg{0}{2}\vseg{1}{2}
		\vseg{-1}{0}
		\vseg{0}{0}\vseg{1}{0}
		\drawsquaretext[\phantom{f}$x$\phantom{f}]{0}{2}
	}
\end{align}
The partial trace operation $\ptr_n:\TLn\to\TL_{n-1}$ is defined as the $\TL_{n-1}$-bimodule map
\begin{align}\label{def:ptr}
	\tikzc{
		\vseg{-1}{2}
		\vseg{0}{2}\vseg{1}{2}
		\vseg{-1}{0}
		\vseg{0}{0}\vseg{1}{0}
		\drawsquaretext[\phantom{f}$x$\phantom{f}]{0}{2}
	}~\mapsto~
	\tikzc{
		\vseg{1}{2}
		\vseg{0}{2}\vseg{-1}{1}
		\vseg{1}{0}
		\vseg{0}{0}
		\drawsquaretext[\phantom{f}$x$\phantom{f}]{0}{2}
		\ncap{-3}{3}\ncup{-3}{1}
		\vseg{-3}{1}
	}
\end{align}
and for $x\in\TLn\hookrightarrow\TL_{n+1}$, we have $e_{n}'xe_n'=e_n'\ptr_n(x)$.

We briefly recall the Jones-Wenzl idempotents. For each $n\geq0$, there exists a unique $f^{(n)}\in \TLn$ which satisfies the following properties
\begin{itemize}
	\item $f^{(n)}\neq 0$\,,
	\item $f^{(n)}f^{(n)}=f^{(n)}$\,,
	\item $e_if^{(n)}=f^{(n)}e_i=0$ for all $i\in\{1,\dots, n-1\}$\,.
\end{itemize}

Let $f^{(0)}$ denote the empty diagram and set $f^{(1)}$ to be a single strand. For $n\geq2$ the Jones-Wenzl idempotents are defined by the recursive relation 
\begin{align}
	f^{(n+1)}=f^{(n)}-\frac{[n]}{[n+1]}f^{(n)}e_nf^{(n)}
\end{align}
where $f^{(n)}\in \TL_{n+1}$ via the inclusion $\TLn\hookrightarrow \TL_{n+1}$. These idempotents also satisfy the relation
\begin{align}
	\ptr_{n}(f^{(n)})=\frac{[n+1]}{[n]}f^{(n-1)}\,.
\end{align} 

In some cases, it will be easier for us to use the notation $\cup_i$ and $\cap_{i}$ to describe Temperley-Lieb diagrams. For instance, we will consider the $\TL_{n+2}$-diagram $\cup_{n}f^{(n+1)}\cap_n$. Each of $\cup_i$ and $\cap_{i}$ indicates a cup or cap over strands $i$ and $i+1$. In this way, the first $n-1$ strands of $f^{(n+1)}$ are to the right of $\cup_{n}$ and the $n+1$-st strand of $f^{(n+1)}$ is to the left of $\cup_{n}$ (in the right to left labeling convention), and similarly for $\cap_n$. The expression $\cup_{n}f^{(n+1)}\cap_n$ is equal to $e_{n+2}e_{n+1}f^{(n+1)}e_{n+1}$ and may be drawn diagrammatically as
\begin{align}
	\cup_{n}f^{(n+1)}\cap_n
	=~
	\tikzc{
		\ncap{-1.5}{-4}\ncup{-1.5}{1.5}
		\vseg{-2}{-.5}\vseg{1}{-.5}\vseg{2}{-.5}
		\vseg{-2}{-2}\vseg{1}{-2}\vseg{2}{-2}
		\vseg{-2}{-4}\vseg{1}{-4}\vseg{2}{-4}
		\drawsquaretext[\phantom{x}$f^{(n+1)}$\phantom{x}]{0}{-1.25}
	}
\end{align}

\section{The Universal Weyl Category}\label{sec:weyldef}
\subsection{Definition} We define an additive $\mathbb{C}(q)$-linear strict monoidal category $\uWeyl'$ as follows. The set of objects is tensor generated by the objects $Q_+$ and $Q_-$. An object in $\uWeyl'$ is a finite direct sum of 
objects $Q_\epsilon=Q_{\epsilon_\ell}\otimes\dots\otimes Q_{\epsilon_1}$, where $\epsilon=\epsilon_\ell\cdots\epsilon_1\in\{+,-\}^\ell$. We denote the unit object by $\mathbbm{1}$, which corresponds to the empty sequence.

The space of morphisms $\Hom_{\uWeyl'}(Q_{\epsilon},Q_{\epsilon'})$ is the $\bC(q)$-module generated by diagrams
\begin{align}
	\tikzc{
		\up{0}{0}
	}~
	&&
	\tikzc{
		\rcup{0}{0}
	}~
	&&
	\tikzc{
		\rcap{0}{0}
	}~
	&&
	\tikzc{
		\lcup{0}{0}
	}~
	&&
	\tikzc{
		\lcap{0}{0}
	}~
\end{align} 
\begin{align}\label{def:CupCapBox}
	\tikzc{
		\uucup{0}{0}\uucap{0}{-4}
	}~
	&& &&
	\tikzc{
		\drawsquare[k]{0}{0}}
\end{align}
for $k\in\bZ$ modulo rel boundary isotopies. Notice that the last diagram $\boxed{k}$ has no boundary components. We call such elementary diagrams boxes. The other diagram in \eqref{def:CupCapBox} is called a (disoriented) cupcap. We refer to the non-solid components of cupcaps as doubled strands. The orientation at the endpoints of a diagram in $\Hom_{\uWeyl}(Q_{\epsilon},Q_{\epsilon'})$ must agree with the signs in the sequences $\epsilon$ and $\epsilon'$, where $+$ corresponds to $\uparrow$ and $-$ corresponds to $\downarrow$\,. Composition of morphisms is given by stacking diagrams. Diagrams without boundary components are endomorphisms of $\mathbbm{1}$.

Before stating the relations in $\uWeyl'$, we make the following notational conveniences for $k\geq 0$. 
\begin{align}
	\label{eq:DefZeroBubble}
    {\tikzc{\up[top>]{0}{0}\vseg{0}{-2}\drawP{0}{0}{0}}}~&:=~{\tikzc{\up[top>]{0}{0}\vseg{0}{-2}}~-[2]~\tikzc{\vseg[top>]{0}{4}\ncap{2}{4}
    \vseg[bottom<]{4}{2}
    \vseg{4}{0}
    \cupcap{0}{0}\ncup{2}{0}
    \vseg[top>]{0}{-2}
    }}
    &
    {\tikzc{
    \drawP{0}{0}{0}
    }}
    ~&:=~
   { \tikzc{
    \widercup{0}{0}{2}
    \widecap{0}{0}{2}
    \ncup{1}{0}
    \rcap{1}{0}
    \drawP{1}{0}{0}
    }}
    \\[1em]
    {\tikzc{
			\Ncup{0}{0}
			\Rcap{0}{0}
		}}
		~&:=~
		{\tikzc{ 
			\widecap{0}{0}{3}\ncap{2}{0}
			\nncap{0}{-4}\uucup{0}{0}
			\widecup{0}{-4}{3}\ncup{2}{-4}
			\vseg{4}{-4}\vseg{6}{-4}
			\down{4}{-2}\down{6}{-2}
		} }
   &
    {\tikzc{
			\Ncap{0}{0}
			\Rcup{0}{0}			
		}}
		~&:=~
		{\tikzc{ 
			\widecap{0}{0}{3}\ncap{2}{0}
			\nncap{4}{-4}\uucup{4}{0}
			\widecup{0}{-4}{3}\ncup{2}{-4}
			\vseg{0}{-4}\vseg{2}{-4}
			\down{0}{-2}\down{2}{-2}
		} }
   \\[1em]
    {\tikzc{
    \drawP{2}{0}{k+1}}
    }~&:=~
    {\tikzc{
    \Ncup{0}{0}
    \Rcap{0}{0}
    \drawP{1}{0}{k}
    }}
    &
    {\tikzc{\up[top>]{0}{0}\vseg{0}{-2}\drawP{0}{0}{k}}}~&:=~
    {\tikzc{\up[top>]{0}{0}\vseg{0}{-2}\drawP{-2}{0}{k}}}
    \label{eq:BubbleDefs}
    \\[1em]
    {\tikzc{
    	\lefte{0}{0}
    }}~&:=~
    {\tikzc{
    	\ncap{-2}{4}\up{2}{4}\down[lower<]{4}{4}
    	\cupcap{0}{0}
    	\ncup{2}{0}\up[lower>]{0}{-2}
    	\vseg{-2}{2}\vseg{4}{2}
    	\vseg{-2}{0}
    	\down{-2}{-2}\vseg{4}{0}
    }}
    &
     {\tikzc{
    		\righte{0}{0}
    }}~&:=~
    {\tikzc{
    		\up{0}{4}\rcap{2}{4}
    		\cupcap{0}{0}
    		\rcup{-2}{0}\up[lower>]{2}{-2}
    		\vseg{-2}{2}\vseg{4}{2}
    		\vseg{-2}{0}
    		\vseg{4}{0}
    		\down[upper<]{-2}{4}
    		\down{4}{-2}
    }}
   \end{align}
   \vspace{-\baselineskip}
   \begin{align}
   	\left(~
   	\tikzc{\drawsquare[k+1]{0}{0}}
   	~\right)^{\dagger}
   	:=~	
   	\tikzc{\rcup{0}{0}\ncap{0}{0}}
   	~-~
   	\tikzc{\drawsquare[k]{0}{0}}
   \end{align}

\begin{rem}\label{rem:ansatz} The local relations we impose are significantly more complicated than in the Heisenberg category. We don't have a satisfying explanation for this complexity, but our ansatz for imposing relations is ``relations that hold in the action of the category should hold in the category itself.'' The relations imposed will hold for both categories we construct from $\uWeyl'$ in their actions on Temperley-Lieb modules.
\end{rem} 

We impose the following local relations between morphisms in $\uWeyl'$:

\begin{align} 
	\label{eq:AlternatingCupCaps}
    {[2]^3~
    \tikzc{
    \cupcap{0}{0}
    \vseg{4}{0}
    \vseg{4}{2}
    \cupcap{2}{4}
    \vseg{0}{4}
    \vseg{0}{6}
   \uucup{0}{12}\nncap{0}{8}
    \vseg{4}{8}
    \up[top>]{4}{10}
    }}
    ~&=~
    {[2]~\tikzc{
    \uucup{0}{4}\nncap{0}{0}
    \vseg{4}{0}
    \up[top>]{4}{2}
    }}
    &
    {[2]^2~\tikzc{
    \cupcap{0}{0}
    \vseg[bottom<]{4}{0}
    \vseg{4}{2}
    \ncap{2}{4}
    \lcup{2}{8}
    \vseg{0}{4}
    \vseg{0}{6}
    \uucup{0}{12}\nncap{0}{8}
    \vseg{4}{8}
    \vseg{4}{10}
    }}
    ~&{=~
    [2]~\tikzc{
    \uucup{0}{4}\nncap{0}{0}
    \vseg[bottom<]{4}{0}
    \vseg{4}{2}
    }}
\end{align}
and reflections of the relations in \eqref{eq:AlternatingCupCaps} across a vertical axis, 
\begin{align}
	\label{eq:TraceE}
    [2]~{\tikzc{ 
        \ncap{-2}{6}\up[top>]{2}{6}
        \down[bottom<]{-2}{4}
        \vseg{-2}{2}\cupcap{0}{2}
        \ncup{-2}{2}\up{2}{0}
    }}
    ~=~
    {\tikzc{
        \vseg{0}{-2}\vseg{0}{0}\vseg[top>]{0}{2}
    }}
    &&
    [2]~\tikzc{\vseg[top>]{0}{4}\ncap{2}{4}
    \vseg[bottom<]{4}{2}
    \vseg{4}{0}
    \cupcap{0}{0}\ncup{2}{0}
    \vseg[top>]{0}{-2}
    }
    ~=\left(1-{\tikzc{
    \drawP{0}{0}{0}
    }}\right)~{\tikzc{\up{0}{0}\vseg{0}{-2}}}
\end{align}
\begin{align}
	\label{eq:CCWCircles}
		{\tikzc{
		\rcup{0}{0}
		\ncap{0}{0}
	}}
	~&=~
	[2]
	&
	{\tikzc{
		\rcup{0}{0}
		\ncap{0}{0}
		\drawP{1}{0}{k}
	}}
	~&=~
	{\tikzc{
		\drawsquare[-2k+1]{0}{0}
		\drawP{4.5}{0}{k-1}
	}
	~
	+
	\left(~
	\tikzc{\drawsquare[-2k]{0}{0}}
	~
	\right)^{\dagger}
	\tikzc{
		\drawP{1}{0}{k}
	}}\\
	{\tikzc{
		\ncup{0}{0}
		\rcap{0}{0}
	}}
	~&=~
	[2]-{\tikzc{\drawP{1}{0}{0}}
	\left(
	~
	\tikzc{\drawsquare[0]{0}{0}}
	~\right)^{\dagger} }
	&	{\tikzc{
			\ncup{0}{0}
			\rcap{0}{0}
			\drawP{1}{0}{k}
	}}
	~&=~
	{\tikzc{
			\drawsquare[-2k-1]{0}{0}
			\drawP{3.5}{0}{k}
	}
	~
	+
	\left(
	~
	\tikzc{\drawsquare[-2k-2]{0}{0}}
		~\right)^{\dagger}
		\tikzc{
			\drawP{1}{0}{k+1}
	}}
	\label{eq:CWCircles}
\end{align}
\begin{align}
	\label{eq:PThroughUp}
	\tikzc{
		\up[top>]{0}{0}\vseg{0}{-2}
		\drawP{-1.5}{0}{k}
	}~=~
	\tikzc{
		\up[top>]{0}{0}\vseg{0}{-2}
		\drawP{-1.5}{0}{k}
	}
	\left(
	\tikzc{
		\drawP{0}{0}{k}
	}
	~+~
	\tikzc{
		\drawP{0}{0}{k-1}
	}
	\right)
	&&
	\tikzc{
		\up[top>]{0}{0}\vseg{0}{-2}
		\drawP{1.5}{0}{k}
	}~=
	\left(
	\tikzc{
		\drawP{0}{0}{k}
	}
	~+~
	\tikzc{
		\drawP{0}{0}{k+1}
	}
	\right)
	\tikzc{
		\up[top>]{0}{0}\vseg{0}{-2}
		\drawP{1.5}{0}{k}
	} 
\end{align}
\begin{align}
	\label{eq:DoubledRelations}
		\tikzc{
			\ddcup{0}{9}\uucup{4}{9}
			\Widercup{1}{7}{2}
			\ddcap{0}{0}\uucap{4}{0}
			\Widelcap{1}{2}{2}}
		~=~
		\tikzc{
			\ddcap{0}{0}\ddcup{0}{4}
			\uucap{4}{0}\uucup{4}{4}
		}
	&&
	&&
	\tikzc{\uucup{0}{4}\nncap{0}{0}
		\drawP{-1}{2}{0}}~=0 
\end{align}
\begin{align}
	\label{eq:BubbleBoxRelations}
	\tikzc{
		\drawP{0}{0}{k}
		\drawP{2}{0}{j}
	}
	=\delta_{kj}~
	\tikzc{
		\drawP{2}{0}{k}
	}
	&&
	\tikzc{
		\drawsquare[k]{-2}{0}
		\up[top>]{0}{0}\vseg{0}{-2}}
	~=~
	\tikzc{
		\drawsquare[k+1]{3}{0}
		\up[top>]{0}{0}\vseg{0}{-2}}
\end{align}
\begin{align}\label{eq:BoxInverse}
	\left(~\tikzc{
		\drawsquare[k]{0}{0}
	}~\right)^{\dagger}=
	\left(~\tikzc{
		\drawsquare[k]{0}{0}
	}~\right)^{-1}\mbox{for $k\geq0$}
	&&
	\tikzc{
	\drawP{0}{0}{k}}~=0~\mbox{for $k<0$}
\end{align}
\begin{align}
	\label{eq:downKup}
	\tikzc{\vseg{-2}{2}\up[top>]{0}{2}\down[bottom<]{-2}{0}\vseg{0}{0}\drawP{-1}{2}{k}\drawP{1}{2}{k}\drawP{-3}{2}{k}}
	~=~
	\tikzc{
		\vseg{0}{5}\up[top>]{2}{5}    
		\ncup{0}{5}
		\ncap{0}{2}
		\vseg{2}{0}\down{0}{0}
		\drawP{1}{5}{k}
		\drawP{1}{2}{k}
		\drawP{4}{5}{k}
		\drawsquare[-2k]{4}{2}
	}
	&&
	\tikzc{\vseg{-2}{2}\up[top>]{0}{2}\down[bottom<]{-2}{0}\vseg{0}{0}\drawP{-1}{2}{k}\drawP{2}{2}{k-1}\drawP{-4}{2}{k-1}}
	~=~
	\tikzc{
		\vseg{0}{5}\up[top>]{2}{5}    
		\ncup{0}{5}
		\ncap{0}{2}
		\vseg{2}{0}\down{0}{0}
		\drawP{1}{5}{k}
		\drawP{1}{2}{k}
		\drawP{3.5}{3.5}{k-1}
	}\left(~
	\tikzc{\drawsquare[-2k+1]{0}{0}}
	~\right)^{\dagger}
\end{align} 
as well as
\begin{equation}\label{eq:MoveBubbles}
 \begin{aligned}
 	{[2]~
 	\tikzc{
 		\up[top>]{0}{4}\up[top>]{2}{4}
 		\cupcap{0}{0}
 		\vseg{0}{-2}\vseg{2}{-2}
 		\drawP{1}{0}{k}
 	}
 	~=}&{~
 	[2]~
 	\tikzc{
 		\up[top>]{0}{5.25}\up[top>]{3.5}{5.25}
 		\widecupcap{0}{0}{1.75}
 		\vseg{0}{-2}\vseg{3.5}{-2}
 		\drawP{1.75}{5.5}{k+1}
 		\drawP{3}{2.75}{k}
 	}}
 	~{+~
 	[2]~
 	\tikzc{
 		\up[top>]{0}{5.25}\up[top>]{3.5}{5.25}
 		\widecupcap{0}{0}{1.75}
 		\vseg{0}{-2}\vseg{3.5}{-2}
 		\drawP{1.75}{5.5}{k-1}
 		\drawP{0}{2.75}{k}
 	}~
 	+~
 	\tikzc{
 		\up[top>]{0}{0}\up{2}{0}
 		\vseg{0}{-2}\vseg{2}{-2}
 		\drawP{-1.75}{0}{k+1}
 		\drawP{1}{0}{k}
 		\drawP{3}{0}{k}
 	}
 	~\left(~\tikzc{\drawsquare[-2k]{8}{0}}~\right)^{\dagger}
    }
 \\
 	&{+~
 	\tikzc{
 		\up[top>]{0}{0}\up{2}{0}
 		\vseg{0}{-2}\vseg{2}{-2}
 		\drawP{-1}{0}{k}
 		\drawP{1}{0}{k}
 		\drawP{3.75}{0}{k-1}
 	}
 	~\tikzc{\drawsquare[-2k+1]{8}{0}}
    ~-~
 	\tikzc{
 		\up[top>]{0}{0}\up{3.5}{0}
 		\vseg{0}{-2}\vseg{3.5}{-2}
 		\drawP{-1.75}{0}{k+1}
 		\drawP{1.75}{0}{k+1}
 		\drawP{4.5}{0}{k}
 	}
 	~\tikzc{\drawsquare[-2k-1]{8}{0}}
 	~
 	}
    \\&
    {-~
 	\tikzc{
 		\up[top>]{0}{0}\up{3.5}{0}
 		\vseg{0}{-2}\vseg{3.5}{-2}
 		\drawP{-1}{0}{k}
 		\drawP{1.75}{0}{k-1}
 		\drawP{5.25}{0}{k-1}
 	}
 	\left(~\tikzc{\drawsquare[-2k+2]{8}{0}}~\right)^{\dagger}}
 \end{aligned}
 \end{equation}

\begin{rem}\label{rem:DownKUp}
	We obtain a slightly different, but equivalent, set of relations from (\ref{eq:downKup}) by rotating the diagrams using cups and caps.
	\begin{align*}
		{\tikzc{
			\vseg{2}{5}\up[top>]{0}{5}    
			\ncup{0}{5}
			\ncap{0}{2}
			\vseg{0}{0}\down{2}{0}
			\drawP{1}{5}{k}
			\drawP{1}{2}{k}
			\drawP{2.5}{3.5}{k}
		}
		~=~
		\tikzc{
			\vseg{0}{2}\up[top>]{-2}{2}
			\down[bottom<]{0}{0}\vseg{-2}{0}
			\drawP{1}{2}{k}
			\drawsquare[-2k-1]{4.5}{2}
			\drawP{-1}{2}{k}\drawP{-3}{2}{k}
			}}
		&&
		{\tikzc{
			\vseg{2.5}{7}\up[top>]{-1}{7}    
			\widecup{-1}{7}{1.75}
			\widecap{-1}{2}{1.75}
			\vseg{-1}{0}\down{2.5}{0}
			\drawP{.75}{7}{k-1}
			\drawP{.75}{2}{k-1}
			\drawP{3}{4.5}{k}
		}
		~=~
		\tikzc{
			\vseg{1.5}{2}\up[top>]{-2}{2}
			\down[bottom<]{1.5}{0}\vseg{-2}{0}
			\drawP{2.5}{2}{k}
			\drawP{-.25}{2}{k-1}\drawP{-3}{2}{k}
		}
		\left(~\tikzc{
			\drawsquare[-2k]{2}{1}}~\right)^{\dagger}}
	\end{align*} 
	The relations in each of (\ref{eq:downKup}) and those directly above, together with (\ref{eq:PThroughUp}) imply
	\begin{align*}
		\tikzc{\vseg{-2}{2}\up[top>]{0}{2}\down[bottom<]{-2}{0}\vseg{0}{0}\drawP{-1}{2}{k}}
		~=~
		\tikzc{
			\vseg{0}{5}\up[top>]{2}{5}    
			\ncup{0}{5}
			\ncap{0}{2}
			\vseg{2}{0}\down{0}{0}
			\drawP{1}{5}{k}
			\drawP{1}{2}{k}
			\drawP{4}{5}{k}
			\drawsquare[-2k]{4}{2}
		}
		~+~
		\tikzc{
			\vseg{0}{5}\up[top>]{2}{5}    
			\ncup{0}{5}
			\ncap{0}{2}
			\vseg{2}{0}\down{0}{0}
			\drawP{1}{5}{k}
			\drawP{1}{2}{k}
			\drawP{3.5}{3.5}{k-1}
		}\left(~
		\tikzc{\drawsquare[-2k+1]{0}{0}}
		~\right)^{\dagger}
		+~
		\tikzc{\vseg{-2}{2}\up[top>]{0}{2}\down[bottom<]{-2}{0}\vseg{0}{0}
			\drawP{1.75}{2}{k-1}\drawP{-3}{2}{k}\drawP{-1}{2}{k}}
		~+~
		\tikzc{\vseg{-2}{2}\up[top>]{0}{2}\down[bottom<]{-2}{0}\vseg{0}{0}
			\drawP{1}{2}{k}\drawP{-1}{2}{k}\drawP{-3.75}{2}{k-1}}
	\end{align*}
	\begin{align*}
		\tikzc{\vseg{0}{2}\up[top>]{-2}{2}\down[bottom<]{0}{0}\vseg{-2}{0}\drawP{-1}{2}{k}}
		~=~
		\tikzc{
			\vseg{2}{5}\up[top>]{0}{5}    
			\ncup{0}{5}
			\ncap{0}{2}
			\vseg{0}{0}\down{2}{0}
			\drawP{1}{5}{k}
			\drawP{1}{2}{k}
			\drawP{3}{3.5}{k}
		}
		~\left(
		~\tikzc{\drawsquare[-2k-1]{0}{0}}~
		\right)^{\dagger}
		~+~
		\tikzc{
			\vseg{2}{5}\up[top>]{0}{5}    
			\ncup{0}{5}
			\ncap{0}{2}
			\vseg{0}{0}\down{2}{0}
			\drawP{1}{5}{k}
			\drawP{1}{2}{k}
			\drawP{4}{5}{k+1}
			\drawsquare[-2k-2]{5}{2}
		}
		+~
		\tikzc{\vseg{0}{2}\up[top>]{-2}{2}\down[bottom<]{0}{0}\vseg{-2}{0}
			\drawP{1.75}{2}{k+1}\drawP{-3}{2}{k}\drawP{-1}{2}{k}}
		~+~
		\tikzc{\vseg{0}{2}\up[top>]{-2}{2}\down[bottom<]{0}{0}\vseg{-2}{0}
			\drawP{1}{2}{k}\drawP{-1}{2}{k}\drawP{-3.75}{2}{k+1}}
	\end{align*}
\end{rem}

\begin{rem}
	Relation \eqref{eq:MoveBubbles} is somewhat complicated but we make use of it in the proof of Proposition \ref{prop:IsosTable}.
\end{rem}

Let $\uWeyl$ be the Karoubi envelope of $\uWeyl'$. Thus, the objects of $\uWeyl$ are pairs $(Q_\epsilon, e)$, where $e:Q_\epsilon\to Q_\epsilon$ is an idempotent. Morphisms in $\uWeyl$ are triples $(f,e,e')$ where $f:Q_\epsilon\to Q_{\epsilon'}$ is a morphism in $\uWeyl'$ such that $f=e'\circ f=f\circ e$. For each $k\geq0$, let $\kimage[k]$ be the image of the idempotent $\circled{k}$ in $\End(\mathbbm{1})$ as an object in $\uWeyl$. 

\begin{rem}\label{rem:DoubledObjects}
	The second relation in \eqref{eq:AlternatingCupCaps} together with the first relation in \eqref{eq:CCWCircles} imply that 
	\begin{align*}
		\tikzc{ 
			\uucup{0}{4}\nncap{0}{0}\uucup{0}{8}\nncap{0}{4}
		}
		~=~
		\tikzc{ 
			\uucup{0}{4}\nncap{0}{0}
		}
	\end{align*}
	by taking the right trace of the former and resolving the counterclockwise circle. This shows that the cupcap is an idempotent in our category, and similarly for its dual. Let $Q_{\oplus}$ denote the image of this projection on $Q_+^{\otimes 2}$ as an object in $\uWeyl$. The image of the dual cupcap is denoted $Q_{\ominus}$. Diagrammatically, we write $Q_{\oplus}$ as the source and target of an isolated ``doubled strand''. It is then natural write the projection and inclusion maps between $Q_+^{\otimes 2}$ and $Q_{\oplus}$ as factors of the cupcap:
    \begin{align*}
		\tikzc{ 
			\uucap{0}{0}
		}
		~:Q_+\otimes Q_+\to Q_{\oplus}
		&&
        \tikzc{ 
			\uucup{0}{4}
		}
		~:Q_{\oplus}\to Q_+\otimes Q_+
	\end{align*}
    Relations which only involve doubled strands can be expressed as relations on the full subcategory of $\uWeyl$ generated by $Q_\oplus$ and $Q_\ominus$. In particular, the first relation in \eqref{eq:DoubledRelations} and the relation above  can be formulated as
	\begin{align}\label{eq:DoubleRels}
		\tikzc{\Lcap{0}{0}\Rcup{0}{4}}
		~=~
		\tikzc{\Vseg{0}{2}\Down[bottom<]{0}{0}\Vseg{2}{0}\Up[top>]{2}{2}}
		&&\mbox{and}&&
		\tikzc{ 
			\uucup{0}{4}\nncap{0}{4}
		}
		~=~
		\tikzc{ 
			\Up[top>]{0}{2}\Vseg{0}{0}
		}
	\end{align}
 
	Left and right duality between $Q_\oplus$ and $Q_\ominus$ follow from left and right duality between $Q_+$ and $Q_-$. For example,
	\begin{align*}
		\tikzc{
			\Up[top>]{0}{0}
			\Ncap{2}{0}\Ncup{0}{0}
			\Up{4}{-2}
		}~=~
		\tikzc{
			\Up[top>]{1}{4}\Vseg{1}{2}
			\ddcup{4}{4}\ncap{6}{4}\uucup{8}{4}
			\widecap{4}{4}{3}
			\uucap{0}{0}\ddcap{4}{0}\Vseg{9}{0}
			\Up[top>]{9}{-2}	\Vseg{9}{-4}
			\widecup{0}{0}{3}\ncup{2}{0}
		}
		~=~
		\tikzc{
			\nncap{0}{0}\uucup{0}{0}
			\nncap{0}{4}\uucup{0}{4}
		}
		~=~
		\tikzc{
			\Up[top>]{0}{2}
			\Vseg{0}{0}
		}
	\end{align*}
\end{rem}

\begin{rem}\label{rem:fromheis}
    The standard surjection from the Hecke algebra $H_n$ to $\TLn$ defined on standard generators $\sigma_i\mapsto q^{1/2}-q^{-1/2}\ckme_i$ does not determine a functor between the Heisenberg and Weyl categories $\Heis_q^{LS}\to\uWeyl$. Indeed, the relation 
    \[0~=~\tikzc{\draw[very thick] (0,0) to (0,2em) to [out = 90, in = -90] (-2em,4em) to [out = 90, in = 0] (-3em,5em) to [out = 180, in = 90] (-4em,4em) to (-4em, 2em) to [out = -90, in= 180]  (-3em, 1em) to [out = 0, in= -90] (-2em,2em);
        \draw[line width=4pt,white] (-2em,2em) [out = 90, in= -90] to (0em, 4em);
        \draw[very thick,->,-stealth] (-2em,2em) [out = 90, in= -90] to (0em, 4em) to (0em,6em); }\]
    in $\Heis_q^{LS}$ mapping to 
    \[
    0=q^{1/2}~
    \tikzc{
        \rcup{0}{0}
		\ncap{0}{0} 
        \vseg{3}{-2}
        \up{3}{0}
        }
        ~-q^{-1/2}[2]~
    \tikzc{
        \ncap{-2}{4}
        \vseg{-2}{0}\vseg{-2}{2}\cupcap{0}{0}
        \ncup{-2}{0}
        \vseg{2}{-2}
        \up{2}{4}
        }
        ~=~
        q^{3/2}~\tikzc{\vseg{3}{-2}
        \up{3}{0}}
    \]
    implies $\mathrm{id}_{Q_+}=0$.
\end{rem}

\subsection{Basic Isomorphisms in $\uWeyl$}\label{ssec:WeylIsos}

There are several isomorphisms in $\uWeyl$ which follow from single relations together with idempotency of bubbles. We will suppress the tensor product from the notation to improve readability, i.e. $A\otimes B$ is written as $AB$. For example,
\begin{align}\label{eq:IsoCThroughQ}
    C_k Q_+&\cong C_k Q_+ C_k\oplus C_k Q_+ C_{k-1}
    \\
    Q_+ C_k&\cong C_k Q_+ C_k\oplus C_{k+1} Q_+ C_{k}
\end{align}
and similar identities involving $Q_-$ follow from \eqref{eq:PThroughUp} and its rotation, and \eqref{eq:BubbleBoxRelations}. In addition, invertiblity of $\boxed{0}$ from \eqref{eq:BoxInverse}, and \eqref{eq:CWCircles} and \eqref{eq:downKup} imply
\begin{align}\label{eq:CQCQC1}
    C_0 Q_- C_0 Q_+ C_0&\cong C_0\,.
\end{align}

The remaining isomorphism, which lends this category its name, is stated as a proposition.

\begin{prop}\label{prop:WeylIso}
	The following isomorphism holds in $\uWeyl$:
	\begin{equation*}
	Q_{-+}\cong Q_{+-}\oplus \kimage[0]\,.
	\end{equation*}
\end{prop} 
\begin{proof}
	Taking the second relation in (\ref{eq:AlternatingCupCaps}), adjoining (tracing) the left strands, and reflecting across a vertical axis, we observe that
	\begin{align*}
		[2]^2~\tikzc{
			\lefte{0}{0}\righte{0}{2}
		}
		~=
		[2]~
		\tikzc{
			\down{-2}{0}\vseg{-2}{2}
			\vseg[top>]{0}{4}\ncap{2}{4}
			\vseg[bottom<]{4}{2}
			\vseg{4}{0}
			\cupcap{0}{0}\ncup{2}{0}
			\vseg[top>]{0}{-2}
		}
		~=
		~
		\tikzc{
			\down{0}{0}\vseg{0}{2}
			\up{2}{2}\vseg{2}{0}
		}
		~-~
		\tikzc{
			\down{0}{0}\vseg{0}{2}
			\drawP{1}{2}{0}
			\up{2}{2}\vseg{2}{0}
		}
	\end{align*}
	with the second equality coming from the second relation in \eqref{eq:TraceE}. We isolate the $\Id_{Q_{-+}}$ term then apply \eqref{eq:PThroughUp} and \eqref{eq:downKup} at $k=0$ to obtain the identity
	\begin{align*}
		\label{eq:DownUpRel}
		\tikzc{
			\down{0}{0}\vseg{0}{2}
			\up{2}{2}\vseg{2}{0}
		}
		~=~
		[2]^2~\tikzc{
			\lefte{0}{0}\righte{0}{2}
		}
		~+~
		\tikzc{
			\vseg{0}{5}\up[top>]{2}{5}    
			\ncup{0}{5}
			\ncap{0}{2}
			\vseg{2}{0}\down{0}{0}
			\drawP{1}{5}{0}
			\drawP{1}{2}{0}
			\drawP{4}{5}{0}
			\drawsquare[0]{4}{2}
		}
	\end{align*} 
	Notice that only one term from \eqref{eq:PThroughUp} contributes. The morphisms on the right side of the relation factor through the objects $Q_{+-}$ and $\kimage[0]$ as presented in the diagram below. \begin{equation*}
		\begin{tikzpicture}[baseline=0pt]
			\draw (0,3) node {$Q_{-+}$};
			\draw (0,-3) node {$Q_{-+}$};
			\draw (-3,0) node {$Q_{+-}$};
			\draw (3,0) node {$\kimage[0]$};
			\draw[->] (-.5,-2.5) to (-2.5,-.5);
			\draw (-1.2,-1.2) node {$\rho_1$};
			\draw[->] (.5,-2.5) to (2.5,-.5);
			\draw (1.2,-1.2) node {$\rho_2$};
			\draw[->] (-2.5,.5) to (-.5,2.5);
			\draw (-1.2,1.2) node {$\iota_1$};
			\draw[->] (2.5,.5) to (.5,2.5);
			\draw (1.2,1.2) node {$\iota_2$};
			\draw (-9em,-6em) node {$[2]$}; 
			\lefte{-8}{-7};
			\draw (-9em,7em) node {$[2]$}; 
			\righte{-8}{6}
			{
				\down{6}{-7}\vseg{8}{-7}  
				\ncap{6}{-5}
				\drawP{7}{-5}{0}
				\drawP{9}{-4}{0}
				\drawsquare[0]{10}{-6}
			}
			
			{
				\vseg{6}{5}\up[top>]{8}{5}    
				\ncup{6}{5}
				\drawP{7}{5}{0}
				\drawP{9}{4}{0}
			}
		\end{tikzpicture}
	\end{equation*}
	Thus the equality above can be formulated as $\iota_1\rho_1+\iota_2\rho_2=\Id_{Q_{-+}}$. We have $\rho_2\iota_1=0$ and $\rho_1\iota_2=0$ as a consequence of \eqref{eq:DoubledRelations}. One can check $\rho_1\iota_1=\Id_{Q_{+-}}$ using \eqref{eq:AlternatingCupCaps} as well as $\rho_2\iota_2=\Id_{\kimage[0]}$ since $\circled{-1}=0$ and $\boxed{0}\,\cdot\,\boxed{0}\,^\dagger=1$. This proves the desired isomorphism.
\end{proof}

\begin{rem}
    In the Grothendieck group $K_0(\uWeyl)$ the isomorphism in Proposition \ref{prop:WeylIso} becomes 
\begin{equation}\label{eq:KWeyliso}
[Q_-][Q_+]-[Q_+][Q_-]=[\kimage[0]]
\end{equation}
a loosening of the defining relation of the Weyl algebra where $1$ is replaced by a \emph{noncentral} idempotent. 

This isomorphism is analogous to one in the quantum Heisenberg categories $\mathcal{H}_{q,k}^{BSW}$ with central charge $k$ studied by \cite{BSW20} (which specialize to the quantum Heisenberg category of \cite{LS13} by setting $k=-1$). In the Heisenberg categories with nonpositive central charge, there is an isomorphism
\begin{align}
	Q_{-+}^H
	\cong
	Q_{+-}^H \oplus\mathbbm{1}^{(-k)}
\end{align} 
In the isomorphism of Proposition \ref{prop:WeylIso}, the object $\kimage[0]$ on the right hand side is a subobject of $\mathbbm{1}$. In a rough sense, this behavior is ``in between'' the Heisenberg categories with central charges $k=0$ and $k=-1$. 
\end{rem}

\subsection{Relations Involving Bubbles}
In this section we state relations that follow from the defining relations of the universal Weyl category that primarily involve bubbles or diagrams without boundary components. 

\begin{lem}
	\label{lem:PThroughE}
	For $k\geq0$, we have the equality
	\begin{align*}
		\tikzc{\uucup{0}{4}\nncap{0}{0}
			\drawP{-1}{2}{k+1}}
		~=~
		\tikzc{\uucup{0}{4}\nncap{0}{0}
			\drawP{-1}{2}{k+1}\drawP{3}{2}{k}}
		~=~
		\tikzc{\uucup{0}{4}\nncap{0}{0}
			\drawP{3}{2}{k}}
	\end{align*}
\end{lem}
\begin{proof}
	By definition of the $k+1$ labeled bubble 
	\begin{align*}
		{\tikzc{
				\drawP{2}{0}{k+1}}
		}~&=~
		{\tikzc{
				\drawP{-2}{0}{k+1}
				\Ncup{0}{0}
				\Rcap{0}{0}
				\drawP{1}{0}{k}
		}}=~
		{\tikzc{
				\Ncup{0}{0}
				\Rcap{0}{0}
				\drawP{1}{0}{k}
		}}
	\end{align*}
	Thus, we obtain the desired result by multiplying the above equalities on the right by the cupcap and applying the left relation in \eqref{eq:DoubledRelations}.
\end{proof}

 The identity $\circled{-1}=0$ is consistent with formally applying Lemma \ref{lem:PThroughE} at $k=-1$ and noting the second relation in (\ref{eq:DoubledRelations}).

\begin{lem}\label{lem:DoubledCircles}
	We have the following identities:
	\begin{align*}
{\tikzc{
	\Ncup{0}{0}
	\Rcap{0}{0}
	}}
	~={1-~\tikzc{
	\drawP{1}{0}{0}
}}&&\mbox{and}&&
{\tikzc{
	\Ncap{0}{0}
	\Rcup{0}{0}
	}}
	~&=1\,.
	\end{align*}
\end{lem}
\begin{proof}
	Since $[2]$ is a unit in $\bC(q)$, each of the desired relations are verified by direct computation:
	\begin{align*}
		[2]~\tikzc{
			\Ncup{0}{0}
			\Rcap{0}{0}
		}
		~=~
		[2]~\tikzc{ 
			\widecap{0}{0}{3}\ncap{2}{0}
			\nncap{0}{-4}\uucup{0}{0}
			\widecup{0}{-4}{3}\ncup{2}{-4}
			\vseg{4}{-4}\vseg{6}{-4}
			\down{4}{-2}\down{6}{-2}
		} 
		~=~
		\left(1-{\tikzc{
				\drawP{0}{0}{0}
		}}\right)~{\tikzc{\ncup{0}{0}
		\rcap{0}{0}}}
	~=~
\left(1-{\tikzc{
		\drawP{0}{0}{0}
}}\right)\left(	[2]-{\tikzc{\drawP{1}{0}{0}}
		\left(
		~
		\tikzc{\drawsquare[0]{0}{0}}
		~\right)^{\dagger} }\right)
	~=~
	[2]\left(1-{\tikzc{
			\drawP{0}{0}{0}
	}}\right)
	\end{align*}
	\begin{align*}
		[2]~\tikzc{
			\Ncap{0}{0}
			\Rcup{0}{0}			
		}
		~=~
		[2]~\tikzc{ 
			\widecap{0}{0}{3}\ncap{2}{0}
			\nncap{4}{-4}\uucup{4}{0}
			\widecup{0}{-4}{3}\ncup{2}{-4}
			\vseg{0}{-4}\vseg{2}{-4}
			\down{0}{-2}\down{2}{-2}
		} 
		~=~
		{\tikzc{\ncup{0}{0}
				\rcap{0}{0}}}
		~=~[2]
	\end{align*}
\end{proof}
\begin{lem}\label{lem:CCWK}
	The following identities hold in $\uWeyl$:
	\begin{align*}
		{\tikzc{
				\Ncap{0}{0}
				\Rcup{0}{0}
				\drawP{1}{0}{0}
		}}
		~=0
		&& \mbox{and}&&
		{\tikzc{
				\Ncap{0}{0}
				\Rcup{0}{0}
				\drawP{1}{0}{k}
		}}
		~=~\tikzc{\drawP{1}{0}{k-1}} ~\mbox{for $k>0$\,.}
	\end{align*}
\end{lem}
\begin{proof}
	In the case $k=0$, after expanding the doubled cup and cap, we apply the second relation in \eqref{eq:DoubledRelations} to get zero. For $k>0$, we have
	\begin{align*}
		\tikzc{
			\Ncap{0}{0}
			\Rcup{0}{0}
			\drawP{1}{0}{k}
		}
		~=~
		\tikzc{
			\drawP{1}{0}{k-1}
			\Widecup{-1}{0}{2}
			\Widecap{-2}{0}{3}
			\Widercup{-2}{0}{3}
			\Widercap{-1}{0}{2}
		}
		~=~
		\tikzc{
			\drawP{1}{0}{k-1}
			\Ncap{-3}{0}
			\Rcup{-3}{0}
		}
		~=~
		\tikzc{
			\drawP{1}{0}{k-1}
		}
	\end{align*}
	The second equality is a consequence of the first relation in \eqref{eq:DoubledRelations} and the last equality is due to Lemma \ref{lem:DoubledCircles}.
\end{proof}

\begin{rem}\label{rem:infsum}
	Relations dependent on $\circled{k}$ appearing in  \eqref{eq:CWCircles}, \eqref{eq:CCWCircles}, and Lemma \ref{lem:CCWK}  are all consistent with the respective $\circled{k}$ independent relations in \eqref{eq:CWCircles}, \eqref{eq:CCWCircles}, and Lemma \ref{lem:DoubledCircles} via the formal infinite sum
	\begin{align*}
		\sum_{k=0}^\infty 
		\tikzc{
			\drawP{2}{0}{k}
		}=1
	\end{align*}
	as shown in the equalities below. 
	For the relations in \eqref{eq:CWCircles} and \eqref{eq:CCWCircles}, we have
	\begin{align*}
		\tikzc{
			\ncup{0}{0}
			\rcap{0}{0}
		}
		~=~
		\sum_{k=0}^\infty 
		\tikzc{
			\ncup{0}{0}
			\rcap{0}{0}
			\drawP{1}{0}{k}
		}
		~=~
		\sum_{k=0}^\infty
		\tikzc{
			\drawsquare[-2k-1]{0}{0}
			\drawP{3.5}{0}{k}
		}
		~
		+
		\left(
		~
		\tikzc{\drawsquare[-2k-2]{0}{0}}
		~\right)^{\dagger}
		\tikzc{
			\drawP{1}{0}{k+1}
		}
		~=~
		[2]-\tikzc{\drawP{1}{0}{0}}
		\left(
		~
		\tikzc{\drawsquare[0]{0}{0}}
		~\right)^{\dagger}
	\end{align*}
	\begin{align*}
		\tikzc{
			\rcup{0}{0}
			\ncap{0}{0}
		}
		~=~
		\sum_{k=0}^\infty
		\tikzc{
			\rcup{0}{0}
			\ncap{0}{0}
			\drawP{1}{0}{k}
		}
		~=~
		\sum_{k=0}^\infty~
		\tikzc{
			\drawsquare[-2k+1]{0}{0}
			\drawP{4.5}{0}{k-1}
		}
		~
		+
		\left(~
		\tikzc{\drawsquare[-2k]{0}{0}}
		~
		\right)^{\dagger}
		\tikzc{
			\drawP{1}{0}{k}
		}
		~=~
		[2]
	\end{align*}
	Between Lemmas \ref{lem:DoubledCircles} and \ref{lem:CCWK},
	\begin{align*}
		{\tikzc{
				\Ncup{0}{0}
				\Rcap{0}{0}
		}}
		~&=~
		{\sum_{k=0}^\infty 
			\tikzc{
				\Ncup{0}{0}
				\Rcap{0}{0}
				\drawP{1}{0}{k}
			}
			~=~
			\sum_{k=0}^\infty 
			\tikzc{
				\drawP{1}{0}{k+1}
			}
			~=1-~\tikzc{
				\drawP{1}{0}{0}
			}
		}
	\end{align*}
	\begin{align*}
		{\tikzc{
				\Ncap{0}{0}
				\Rcup{0}{0}
		}}
		~=~
		{\sum_{k=0}^\infty 
			\tikzc{
				\Ncap{0}{0}
				\Rcup{0}{0}
				\drawP{1}{0}{k}
			}
			~=~
			\sum_{k=0}^\infty 
			\tikzc{
				\drawP{1}{0}{k-1}
			}
			~=~
			\sum_{k=0}^\infty 
			\tikzc{
				\drawP{1}{0}{k}
			}
			~=1\,.
		}
	\end{align*}
	Although the infinite sum identity is not a relation in $\uWeyl$, it appears naturally in the categorical representation of the Weyl category on the direct sum over $n$ of $\TLn$-modules. On any module, $C_k$ the image of $\circled{k}$ in $\mathbbm{1}$, acts as a projection onto certain isotypic simple module components and the sum over all such projections is the identity on any module. Since there are only finitely many isomorphism classes of simple modules for a given $n$, all but finitely many terms in this infinite sum vanish.
\end{rem}

\begin{rem}
	In light of the relations
	\begin{align*}
		\tikzc{\Lcap{0}{0}\Rcup{0}{4}}
		~=~
		\tikzc{\Vseg{0}{2}\Down[bottom<]{0}{0}\Vseg{2}{0}\Up[top>]{2}{2}}
		&&\mbox{and}&&
		{\tikzc{
				\Ncap{0}{0}
				\Rcup{0}{0}
		}}
		~&=1
	\end{align*}
	from Remark \ref{rem:DoubledObjects} and Lemma \ref{lem:DoubledCircles}, we have the isomorphism $Q_{\ominus\oplus}\cong\mathbbm{1}$ in $\uWeyl$.
\end{rem}

\section{The Weyl 2-Category}\label{sec:Wacts}
As it is currently defined, the monoidal category $\uWeyl$ does not act on $\T$, the sum over $n$ of the categories of Temperley-Lieb modules. This is because the constants needed to define the action depend on $n$ in an essential way. This means that to define the action we want, we need to use the 2-categorical point of view.  In this section, we define the Weyl 2-category and compare isomorphisms which hold in the 2-category but not in the monoidal category. The action on $\bigoplus_n \TLn\mathrm{-mod}$ is described precisely in the next section. 

\subsection{The Weyl 2-Category}
The 2-category $\Wact$ is the Karoubi completion\footnote{By the Karoubi completion of a 2-category $\mathcal C$, we mean the 2-category with the same objects as $\mathcal C$, but whose morphism categories are the Karoubi completions of those of $\mathcal C$. } of a 2-category $\Wact'$ that we define now.

\begin{itemize}
    \item The set of objects of $\Wact'$ is $\nabla \sqcup \bZ_{\geq 0}$ ($\nabla$ should be thought of as the null object), 
    \item $\Hom_{\Wact'}(n,m)$ is trivial\footnote{The trivial additive 1-category has one object and one morphism.} if $m=\nabla$ or $n = \nabla$,
    \item For $m,n\in\bZ_{\geq0}$, $\Hom_{\Wact'}(n,m)$ is the category defined diagrammatically below.
\end{itemize}

\begin{defn}
    For $m,n\in\bZ_{\geq 0}$, the additive $\mathbb{C}(q)$-linear strict monoidal 1-category $\Hom_{\Wact'}(n,m)$ is defined as follows:
    \begin{itemize}
        \item Objects are $Q_\epsilon$, where $\epsilon = (\epsilon_\ell,\cdots, \epsilon_1)$ is a finite sequence of $+$ and $-$, with $m = n + \sum_{i = 1}^\ell \epsilon_i$ and  $n + \sum_{i=1}^j \epsilon_i \geq 0$ for each $1 \leq j \leq \ell$\footnote{In these sums, $+$ and $-$ are considered as $1$ and $-1$ respectively.}. Let $\mathbbm{1}_n$ denote the tensor unit in $\Hom_{\Wact'}(n,n)$.
        \item Morphisms in $\Hom_{\Wact'}(n,m)$ are generated by
        \begin{align*}
	\tikzc{
		\up{0}{0}
		\vseg{0}{-2}
		\draw (1em, 0em) node[] {$n$};
		\draw (-2em, 0em) node[] {$n+1$};
	}
 &&
 \tikzc{
		\down{0}{-2}
		\vseg{0}{0}
		\draw (1em, 0em) node[] {$n$};
		\draw (-2em, 0em) node[] {$n+1$};
	}~
 &&
 \tikzc{
		\uucup{0}{0}\uucap{0}{-4}
		\draw (2.5em, -2em) node[] {$n$};
		\draw (-1em, -2em) node[] {$n+2$};
	}~
      \end{align*}
      \begin{align*}  
	\tikzc{
		\widercup{0}{0}{2}
		\draw (-1em, 0em) node[] {$n$};
		\draw (2em, 0em) node[] {$n+1$};
		\draw (5em, 0em) node[] {$n$};
	}~
	&&
	\tikzc{
		\widercap{0}{0}{2}
		\draw (-1em, 0em) node[] {$n$};
		\draw (2em, 0em) node[] {$n-1$};
		\draw (5em, 0em) node[] {$n$};
	}~
	&&
	\tikzc{
		\widelcup{0}{0}{2}
		\draw (-5em, 0em) node[] {$n$};
		\draw (-2em, 0em) node[] {$n-1$};
		\draw (1em, 0em) node[] {$n$};
	},
	&&
	\tikzc{
		\widelcap{0}{0}{2}
		\draw (-5em, 0em) node[] {$n$};
		\draw (-2em, 0em) node[] {$n-1$};
		\draw (1em, 0em) node[] {$n$};
	}
	\end{align*}
    The rightmost label indicates the source 2-object\footnote{By 2-object, we mean an object of the 2-category $\Wact'$ (which belongs to $\nabla \sqcup \bZ_{\geq 0}$), not an object in one of its morphism 1-categories (which is a sequence $\epsilon$)
.} and the leftmost region indicates the target 2-object. Region labels for diagrams in $\Wact'$ extend uniquely after the choice of the rightmost label, and we may use the notation $\Hom_{\Wact'}(n,-)$ when the left region label is not specified.  Labeled boxes also appear as morphisms in $\Hom_{\Wact'}(n,n)$ via the relations imposed from $\uWeyl'$, but are not considered generators because of relation \eqref{eq:newrelBox}.
    \item We impose the local relations \eqref{eq:AlternatingCupCaps}--\eqref{eq:MoveBubbles} on morphisms; these are the same local relations imposed in $\uWeyl'$. As in $\uWeyl'$, morphisms are considered up to planar isotopy. We impose the \emph{additional} relations on morphisms in $2\Hom_{\Wact'}(\mathbbm{1}_n,\mathbbm{1}_n)$:
\begin{align}\label{eq:newrel}
	{\tikzc{
	\drawP{0}{0}{k}
	}} &= 0  ~\mbox{if $2k>n$},
 \\
 \sum_{k=0}^{\floor{n/2}}{~\tikzc{
	\drawP{0}{0}{k}
	}} &= 1, \label{eq:SumBubbles}
 \\
 {\tikzc{
		\drawsquare[k]{0}{0}
		}} &= \begin{cases}
		    \frac{[n+k+1]}{[n+k+2]} &\mbox{{if $n+k\geq 0$}} 
      \\
      0&\mbox{{if $n+k+1\leq 0$}.}
		\end{cases}\label{eq:newrelBox}
\end{align}
    \end{itemize}
\end{defn}

\begin{rem}
    The motivation for \eqref{eq:newrelBox} follows our ansatz, Remark \ref{rem:ansatz}, that in the representation on $\T$, a box $\boxed{k}$ should act as multiplication by the scalar $[n+k+1]/[n+k+2]$  if $n+k\geq 0$. We assign the value zero to all other boxes so that $\Hom_{\Wact'}(n,-)$ is defined as a quotient of $\uWeyl'$. In this way, no invertible box in $\uWeyl'$ is mapped to zero. This assignment of boxes to constants also eases the description of bases of endomorphism categories. 
    
On the other hand, the bubble $\circled{k}$ in region $n$ should act on $\T$ by projecting onto the $W^n_{n-2k}$ isotypic component (see Appendix \ref{sec:app} for a discussion of these modules). If $n-2k < 0$ this module is 0. We therefore set such a morphism to 0 as well. Whereas the sum over these nonzero components is the identity on any module, the sum of nonzero bubbles is $1$. See also Remark \ref{rem:infsum}.
\end{rem}

We use $\tleq[n]$ to indicate that an equality holds in a morphism category $\Hom_{\Wact}(n,m)$ (or $\Hom_{\Wact}(n,-)$ if $m$ is not specified) and we may omit any explicit region labels in the corresponding diagrams.
The example \eqref{eq:examplepslabel} below shows an equality of 2-morphisms between objects in $\Hom_{\Wact}(n,n+1)$ as region labeled diagrams and equivalently using $\tleq[n]$ to specify the rightmost region label. The equality itself is determined by \eqref{eq:TraceE} in the definition of $\uWeyl'$:
\begin{align}\label{eq:examplepslabel}
	[2]~{\tikzc{ 
			\ncap{-2}{6}\up[top>]{2}{6}
			\down[bottom<]{-2}{4}
			\vseg{-2}{2}\cupcap{0}{2}
			\ncup{-2}{2}\up{2}{0}
			\draw (3em, 0.5em) node[] {$n$};
			\draw (-0.5em, 4.5em) node[] {$n+2$};
			\draw (0em, 0.5em) node[] {$n+1$};
	}}
	~=~
	{\tikzc{
			\vseg{0}{-2}\vseg{0}{0}\vseg[top>]{0}{2}
			\draw (1em, 1em) node[] {$n$};
			\draw (-2em, 1em) node[] {$n+1$};
	}}
&&
\mbox{or equivalently}
&&
[2]~{\tikzc{ 
		\ncap{-2}{6}\up[top>]{2}{6}
		\down[bottom<]{-2}{4}
		\vseg{-2}{2}\cupcap{0}{2}
		\ncup{-2}{2}\up{2}{0}
}}
~\tleq[n]~
{\tikzc{
		\vseg{0}{-2}\vseg{0}{0}\vseg[top>]{0}{2}
}}
\end{align}

For each $n\geq0$, let $\mathbbm{1}_n\in \Hom_{\Wact}(n,n)$ denote the tensor unit of $\Hom_{\Wact}(n,n)$ and $\eone$ its algebra of endomorphisms in $2\Hom_{\Wact}(\mathbbm{1}_n,\mathbbm{1}_n)$. As a consequence of the construction of $\Wact$, $\End_{\Wact}(\mathbbm{1}_n)$ is described as the quotient of $\End_{\uWeyl}(\mathbbm{1})$ by the relations \eqref{eq:newrel}-\eqref{eq:newrelBox}.

\begin{rem}\label{rem:1Weyl}
    For $n\geq0$, let $G_n:\End_{\uWeyl}(\mathbbm{1})\to \eone$ be the quotient map determined by the relations \eqref{eq:newrel}-\eqref{eq:newrelBox}. By definition $\eone=\operatorname{im}(G_n)$. 
\end{rem}

\begin{rem}\label{rem:Exceptional}
    Many terms in relation \eqref{eq:MoveBubbles} vanish if $n-2k+2=0$, leaving the relation
\begin{align*}
 	\tikzc{
 		\up[top>]{0}{5.25}\up[top>]{3.5}{5.25}
 		\widecupcap{0}{0}{1.75}
 		\vseg{0}{-2}\vseg{3.5}{-2}
 		\drawP{1.75}{5.5}{k-1}
 		\drawP{0}{2.75}{k}
   \drawP{4}{2.75}{k-1}
 	}
  ~\tleq[n]~\tikzc{
 		\up[top>]{0}{0}\up{3.5}{0}
 		\vseg{0}{-2}\vseg{3.5}{-2}
 		\drawP{-1}{0}{k}
 		\drawP{1.75}{0}{k-1}
 		\drawP{5.25}{0}{k-1}
 	} && \mbox{for $n-2(k-1)=0$.}
 \end{align*}
\end{rem}

\subsection{Isomorphisms in $\Hom_{\Wact}(n,m)$}\label{sec:2WeylIsos}
We discuss the isomorphisms which are introduced in the hom-categories $\Hom_{\Wact}(n,m)$ in comparison to those of $\uWeyl$ stated in Section \ref{ssec:WeylIsos} and how some isomorphisms in $\uWeyl$ become trivialized. The assignment of $\boxed{k}\tleq[n]\dfrac{[n+k+1]}{[n+k+2]}$ for $n+k\geq0$ means that boxes appearing in relations are invertible, and can be used to construct isomorphisms which are not present in $\uWeyl$. Relation \eqref{eq:newrel} implies a number of diagrams which are
present in the defining relations of $\uWeyl$ but 
vanish in certain hom-categories of $\Wact$. Many of these follow from the basic relations 
\begin{align*}
    \tikzc{
    \drawP{0}{0}{k+1}\drawP{3.5}{0}{k}\vseg{2}{-2}\up[top>]{2}{0}
    }~\tleq[2k]0,
    &&
    \tikzc{
    \drawP{0.5}{0}{k}\drawP{3.5}{0}{k}\vseg{2}{0}\down{2}{-2}
    }~\tleq[2k]0.
\end{align*}
Or more simply, $\circled{k+1}\tleq[n+1]0$ and $\circled{k}\tleq[n-1]0$ when $n=2k$. In such instances the corresponding objects of the Karoubi completion $C_k$ are also zero.\\

For example, relations in Remark \ref{rem:DownKUp}, \eqref{eq:CWCircles}, and \eqref{eq:newrelBox} imply the isomorphism $C_k\cong C_kQ_+C_kQ_-C_k$ in $\Hom_{\Wact}(n,n)$ for $n>2k$. However if $n=2k$, then $C_kQ_+C_kQ_-C_k\cong 0$ as $\circled{k}\tleq[n-1]0$, but $C_k\not\cong 0$ in $\Hom_{\Wact}(n,n)$. At this critical value, boxes $\boxed{-2k-1}$ which appear in the relevant relations are also mapped to zero and so the isomorphism fails in this case. If $n<2k$, then both $C_k$ and $C_kQ_+C_kQ_-C_k$ are zero in $\Hom_{\Wact}(n,n)$ and the isomorphism becomes trivial. The various isomorphisms dependent on the value of $n$ relative to $k$ are summarized in Table \ref{tab:isos}.

\begin{prop}\label{prop:IsosTable}
    The isomorphisms in Table \ref{tab:isos} hold in $\Hom_{\Wact}(n,-)$.
\end{prop}
\begin{table}[h!]
\[\begin{array}{c|c|c}
   \mbox{Above critical values}
   &  \mbox{At critical values} & \mbox{Critical values}
   \\\hline
   C_kQ_+\cong C_kQ_+C_{k}\oplus C_kQ_+C_{k-1}  & C_kQ_+\cong C_kQ_+C_{k-1}& n=2k-1\, 
   \\
   &0\cong 0&n<2k-1
   \\
   \hdashline
   Q_+C_k\cong C_kQ_+C_k\oplus C_{k+1}Q_+C_k & Q_+C_k\cong C_kQ_+C_k & n=2k
   \\
   & 0\cong 0 &n<2k
   \\
   \hdashline
   C_kQ_-\cong C_kQ_-C_{k}\oplus C_kQ_-C_{k+1}  & C_kQ_-\cong C_kQ_-C_{k}& n=2k+1\, 
   \\
   &0\cong 0&n<2k+1
   \\
    \hdashline
   Q_-C_k\cong C_kQ_-C_k\oplus C_{k-1}Q_-C_k & Q_-C_k\cong C_kQ_-C_k & n=2k
   \\
   & 0\cong 0 &n<2k
   \\
   \hdashline
    C_k Q_- C_k Q_+ C_k\cong C_k & 0\cong 0 &n<2k
    \\
    \hdashline
    C_{k-1} Q_- C_k Q_+ C_{k-1}\cong C_{k-1}
   & \text{none\footnotemark} & n=2(k-1)
   \\
   & 0\cong 0 & n<2(k-1)
   \\
   \hdashline
   C_k\cong C_kQ_+C_kQ_-C_k &\text{none}&n=2k
   \\
   & 0\cong0 &n<2k
   \\
   \hdashline
   C_k\cong C_kQ_+C_{k-1}Q_-C_k & 0\cong 0 & n<2k
   \\
   \hdashline
   \kimage[k+1] Q_{+} \kimage[k] Q_+ \kimage[k]
		\cong \kimage[k+1] Q_{+} \kimage[k+1] Q_+ \kimage[k]
  & \text{none} & n=2k
  \\ & 0\cong 0 & n<2k
  \\
  \hdashline
  \kimage[k] Q_{-} \kimage[k+1] Q_- \kimage[k+1]
		\cong \kimage[k] Q_{-} \kimage[k] Q_- \kimage[k+1] & \text{none} & n=2(k+1)
  \\
  & 0\cong 0 & n<2(k+1)
\end{array}\]
\caption{``Generic'' isomorphisms in $\Hom_{\Wact}(n,-)$ and their modification at critical values.}
    \label{tab:isos}
\end{table}
\footnotetext{The relation satisfied by the morphism in \eqref{eq:downKup}, which witnesses the isomorphism generically, becomes $0=0$ when $n=2(k-1)$. Moreover, $C_{k-1} Q_- C_k Q_+ C_{k-1}\cong 0$ and $C_{k-1}\not\cong 0$ and are therefore not isomorphic.}
\begin{proof}
    The most interesting relation in the table is $\kimage[k+1] Q_{+} \kimage[k] Q_+ \kimage[k]
		\cong \kimage[k+1] Q_{+} \kimage[k+1] Q_+ \kimage[k]$, which we will prove directly for $n>2k$. The isomorphism $\kimage[k+1] Q_{-} \kimage[k] Q_-\kimage[k]
		\cong \kimage[k+1] Q_{-} \kimage[k+1] Q_- \kimage[k]$ is a consequence of rotating the diagrams.
    The proof of other relations in the table are straightforward to verify from Section \ref{ssec:WeylIsos} or the relations from \eqref{eq:CCWCircles}, \eqref{eq:CWCircles}, \eqref{eq:downKup}, and Remark \ref{rem:DownKUp} with \eqref{eq:newrelBox}. For the critical values of $k$ relative to $n$, certain morphisms will vanish because of \eqref{eq:newrel} and \eqref{eq:newrelBox}.

    To prove $\kimage[k+1] Q_{+} \kimage[k] Q_+ \kimage[k]
		\cong \kimage[k+1] Q_{+} \kimage[k+1] Q_+ \kimage[k]$ in $\Hom_{\Wact}(n,n+2)$ for $n>2k$, we show that the maps
	\begin{align*}
		[2]~
		\tikzc{
			\up[top>]{0}{5.25}\up[top>]{3.5}{5.25}
			\widecupcap{0}{0}{1.75}
			\vseg{0}{-2}\vseg{3.5}{-2}
			\drawP{1.75}{5.5}{k+1}
			\drawP{-1}{2.75}{k+1}
			\drawP{3.5}{2.75}{k}
			\drawP{1.75}{0}{k}
		}~
		&&\mbox{and}&&
		[2]~
		\tikzc{
			\up[top>]{0}{5.25}\up[top>]{3.5}{5.25}
			\widecupcap{0}{0}{1.75}
			\vseg{0}{-2}\vseg{3.5}{-2}
			\drawP{1.75}{5.5}{k}
			\drawP{-1}{2.75}{k+1}
			\drawP{3.5}{2.75}{k}
			\drawP{1.75}{0}{k+1}
		}
	\end{align*}
	witness the desired isomorphism up to a normalization by boxes. By \eqref{eq:AlternatingCupCaps}, \eqref{eq:PThroughUp}, and \eqref{eq:MoveBubbles} their compositions are 
	\begin{align*}
		[2]^2~
		\tikzc{
			\widecupcap{0}{5.5}{1.75}
			\up[top>]{0}{11}\up[top>]{3.5}{11}
			\widecupcap{0}{0}{1.75}
			\vseg{0}{-2}\vseg{3.5}{-2}
			\drawP{1.75}{11}{k}
			\drawP{-1}{8.5}{k+1}
			\drawP{3.5}{8.5}{k}
			\drawP{1.75}{5.5}{k+1}
			\drawP{-1}{2.75}{k+1}
			\drawP{3.5}{2.75}{k}
			\drawP{1.75}{0}{k}
		}	
		~\tleq[]
		[2]~\tikzc{
			\up[top>]{0}{5.5}\up[top>]{3.5}{5.5}
			\widecupcap{0}{0}{1.75}
			\vseg{0}{-2}\vseg{3.5}{-2}
			\drawP{1.75}{5.5}{k}
			\drawP{-1}{2.75}{k+1}
			\drawP{3.5}{2.75}{k}
			\drawP{1.75}{0}{k}
			\widercup{5}{2.75}{2}
			\widecap{5}{2.75}{2}
			\drawP{7}{2.75}{k+1}
		}
		~=~
		\tikzc{
			\up[top>]{0}{0}\up{2}{0}
			\vseg{0}{-2}\vseg{2}{-2}
			\drawP{-1.75}{0}{k+1}
			\drawP{1}{0}{k}
			\drawP{3}{0}{k}
			\drawsquare[-2k-1]{7}{0}
		}
		~\left(~\tikzc{\drawsquare[-2k]{9}{0}}~\right)^{\dagger}~
	\end{align*}
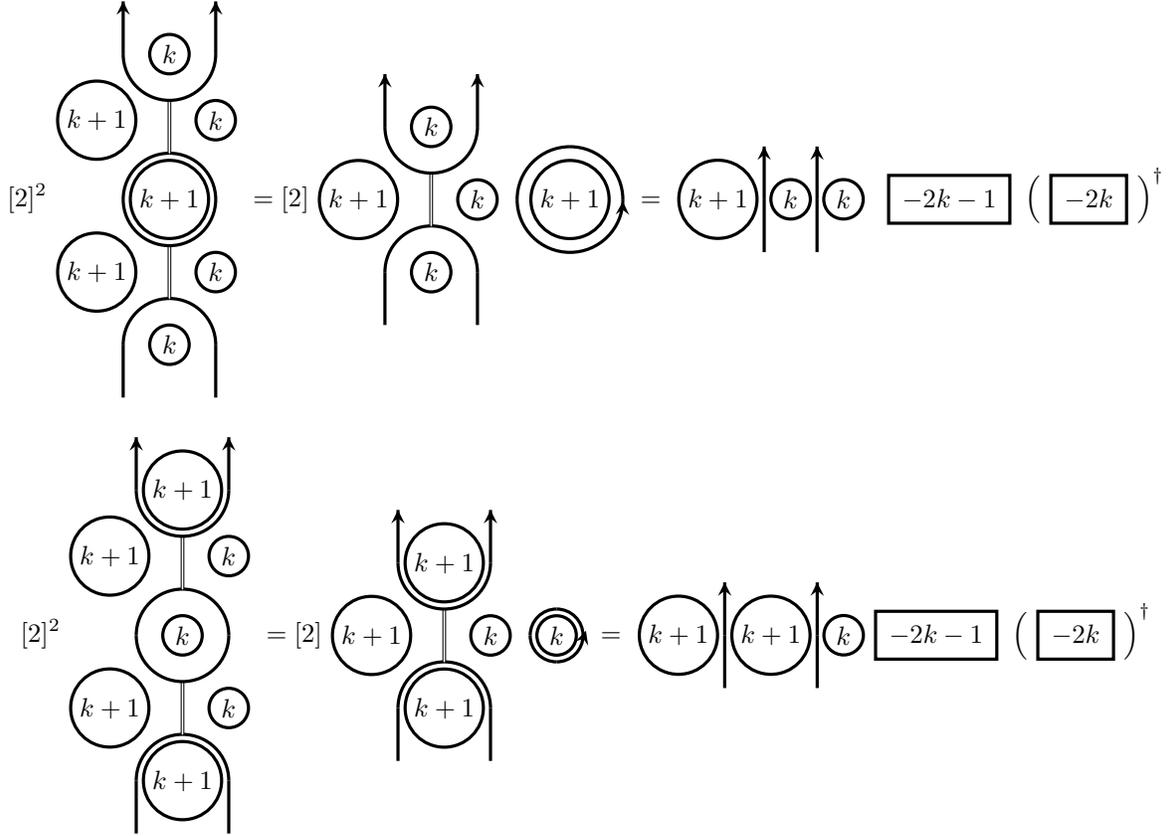
	\begin{align*}
		[2]^2~
		\tikzc{
			\widecupcap{0}{5.5}{1.75}
			\up[top>]{0}{11}\up[top>]{3.5}{11}
			\widecupcap{0}{0}{1.75}
			\vseg{0}{-2}\vseg{3.5}{-2}
			\drawP{1.75}{11}{k+1}
			\drawP{-1}{8.5}{k+1}
			\drawP{3.5}{8.5}{k}
			\drawP{1.75}{5.5}{k}
			\drawP{-1}{2.75}{k+1}
			\drawP{3.5}{2.75}{k}
			\drawP{1.75}{0}{k+1}
		}
		~=
		[2]~\tikzc{
			\up[top>]{0}{5.5}\up[top>]{3.5}{5.5}
			\widecupcap{0}{0}{1.75}
			\vseg{0}{-2}\vseg{3.5}{-2}
			\drawP{1.75}{5.5}{k+1}
			\drawP{-1}{2.75}{k+1}
			\drawP{3.5}{2.75}{k}
			\drawP{1.75}{0}{k+1}
			\rcup{5}{2.75}
			\ncap{5}{2.75}
			\drawP{6}{2.75}{k}
		}
		~=~
		\tikzc{
			\up[top>]{0}{0}\up{3.5}{0}
			\vseg{0}{-2}\vseg{3.5}{-2}
			\drawP{-1.75}{0}{k+1}
			\drawP{1.75}{0}{k+1}
			\drawP{4.5}{0}{k}
			\drawsquare[-2k-1]{8}{0}
		}
		~\left(~\tikzc{\drawsquare[-2k]{9}{0}}~\right)^{\dagger}~
	\end{align*}
	The boxes appearing in the simplified expressions are invertible for $n>2k$. Thus, the above maps can be normalized, which proves the isomorphism. If $n=2k$, then $\kimage[k+1] Q_{+} \kimage[k+1] Q_+ \kimage[k]\cong 0$ and $\kimage[k+1] Q_{+} \kimage[k] Q_+ \kimage[k]\not\cong 0$. Both objects are isomorphic to $0$ if $n<2k$.
\end{proof}

\begin{prop}\label{prop:WeylIson}
    The isomorphism $Q_{-+}\cong Q_{+-}\oplus C_0$ holds in $\Hom_{\Wact}(n,n)$ for any $n\geq 0$.
\end{prop} 
\begin{proof}
    The proof of Proposition \ref{prop:WeylIso} is unchanged for each $n$.
\end{proof}

\begin{rem}\label{rem:asymp}
     In each isomorphism of Table \ref{tab:isos}, if $n$ is sufficiently large $n>2K$, where $K$ is the largest index of $C_k$ in the isomorphism, then the ``generic form'' of the isomorphism holds in $\Hom_{\Wact}(n,-)$. 
\end{rem}

\section{Action of $\Wact$ on $\T$}\label{sec:Wactsacts}

We recall the notion of a 2-categorical representation in the sense of \cite{LS13}. Let $\mathcal{X}= \bigoplus_{m\in M} \mathcal{X}_m$ be an $M$-graded additive category, where $M$ is a monoid. The collection of endofunctors on $\mathcal{X}$ naturally forms a 2-category whose objects are the elements of $M$, whose 1-morphisms from $n$ to $m$, for $n,m\in M$, are the functors from $\mathcal{X}_n$ to $\mathcal{X}_m$, and whose 2-morphisms are natural transformations. Denote such an endofunctor category by $\End(\mathcal{X})$. If $\mathcal{C}$ is a category, then a functor $\rho:\mathcal{C}\to\End(\mathcal{X})$ is called a \emph{representation of $\mathcal{C}$ on $\mathcal{X}$} and we say that $\mathcal{C}$ acts on $\mathcal{X}$. In this section, we construct a representation of $\Wact$ on the category of Temperley-Lieb modules $\T := \bigoplus_{n\geq-1} \mathrm{TL}_n\mathrm{-mod}$, with the convention that $\TL_{-1}\mathrm{-mod}$ is a null object denoted $\nabla$.

Every $(\TL_m,\TLn)$-bimodule $M$ determines a functor
\begin{align}
	M\otimes-:\TL_n\mathrm{-mod}\to \TL_m\mathrm{-mod} &&
	V\mapsto M\otimes V
\end{align}
and any homomorphism $M_1\to M_2$ of $(\TL_m,\TLn)$-bimodules gives  a natural
transformation between the corresponding functors. In this way, a representation of $\Wact$ is an assignment of bimodule maps to diagrams.

For $k,l\leq n$, let $\mod{k}{n}{l}$ denote $\TLn$ as a $(\TL_k,\TL_l)$-bimodule. Tensoring the bimodules $\mod{n+1}{n+1}{n}$ and $\mod{n-1}{n}{n}$ correspond to induction and restriction functors on $\T$, respectively. 

We now propose a representation $\rho':\Wact'\to\End(\T)$ which makes the following assignments.
\begin{itemize}
	\item On objects, $\rho'(n)=n$ for $n\geq0$ and $\rho'(\nabla)=\nabla$. 
	
	\item The functor $\rho'$ maps the object 0 in each hom-category to 0. Each 1-morphism $Q_{\epsilon}\in \Hom_{\Wact}(n,n+\|\epsilon\|)$ with $\epsilon=\epsilon_\ell\dots\epsilon_1$ is mapped by $\rho'$ to a sequence of induction and restriction bimodules, with $+$ mapping to induction and $-$ mapping to restriction. 
	
	\item On 2-morphisms, $\rho'$ makes the following assignments\footnote{We use a notational shorthand $[x \mapsto \ast]: X \to Y$ to describe a function $f:X \to Y$ defined by a formula $f(x) = \ast$. } on elementary diagrams
	\begin{align*}
		{\tikzc{\rcap{0}{0}}}~&\tleq[n]
		\left[x\otimes y\mapsto xy\right]:
		\mod{n}{n}{n-1}\mod{}{n}{n}
		\to \mod{n}{n}{n},
		\\
		{\tikzc{\rcup{0}{1}}}~&\tleq[n]
		\left[x\mapsto x\right]:
		\mod{n}{n}{n}\to \mod{n}{n+1}{n+1}\mod{}{n+1}{n},
		\\
		{\tikzc{\lcap{0}{0}}}~&\tleq[n]
		\left[x\mapsto \ptr_{n+1}(x) \right]:
		\mod{n}{n+1}{n+1}\mod{}{n+1}{n}\to \mod{n}{n}{n}, 
		\\
		{\tikzc{\lcup{0}{1}}}~&\tleq[n]
		\left[x\mapsto \sum_{\substack{p,r\in P^n\\p_n=r_n}}
		c_r e_{p,r} \otimes e_{r,p}x\right]:
		\mod{n}{n}{n}\mapsto \mod{n}{n}{n-1}\mod{}{n}{n},
		\\
		\tikzc{
			\uucup{0}{4}
			\nncap{0}{0}
			}
		~&
		\tleq[n]
		\left[x\mapsto xe_{n+1}\right]:
		\mod{n+2}{n+2}{n}\to\mod{n+2}{n+2}{n},
	\end{align*}
\end{itemize}
where the map $x\mapsto \ptr_{n+1}(x)$ is the right partial trace of $x$ as defined in \eqref{def:ptr} and $\sum_{\substack{p,r\in P^n\\p_n=r_n}}
c_r e_{p,r} \otimes e_{r,p}$ is as described in Section \ref{ssec:Qdiagrams}. 

\begin{rem}\label{rem:tleq}
    In a slight abuse of notation, we use $\tleq[n]$ to denote an equality in a hom-category $\Hom_{\End(\T)}(n,-)$, when it is clear that we are considering an equality of morphisms under $\rho'$. Its use here is similar to that in $\Wact$ where $\tleq[n]$ indicates the rightmost region label on diagrams is $n$ and establishes that the equality is between natural transformations of functors whose sources are $\TLn\mathrm{-mod}$. This convention fits with our ansatz in Remark \ref{rem:ansatz}.
    
    The notation $\tleq[n]$ will typically be reserved for giving an explicit description of an action of a diagram on $\TLn\mathrm{-mod}$. In addition, we will use $\tleq$ if an equality $\tleq[n]$ holds for all $n\geq0$. Typically, this will be used for equalities between the images of diagrams under $\rho'$.  
\end{rem}

Since $\End(\T)$ is idempotent complete, the proposed functor $\rho'$ extends to a functor $\rho:\Wact\to\End(\T)$. In Corollary \ref{cor:isotypic}, each object $\kimage[k]$ is shown to map to $P_k$, the functor which projects onto the $W^n_{n-2k}$ isotypic component of any $\TLn$-module. Here $W^n_{n-2k}$ is an irreducible $\TLn$ defined in Appendix \ref{sec:app}.

\begin{thm}\label{thm:rep}
	The functor $\rho$ is well-defined and determines a representation of $\Wact$ on $\T$.
\end{thm}

It is sufficient to prove the claim for $\rho'$ by showing that the defining relations on 2-morphisms in $\Wact$ hold when interpreted as maps of Temperley-Lieb bimodules according to the assignments above. Therefore, the following corollary is a tautological consequence.

\begin{cor}\label{cor:IsoInAction}
    The images of $Q_+$, $Q_-$, and $C_k$ under $\rho$ are consistent with the isomorphsisms described in Table \ref{tab:isos} and Proposition \ref{prop:WeylIson}.
\end{cor}

\begin{rem}\label{rem:IsoInAction}
    Corollary \ref{cor:IsoInAction} can also be proven by explicitly computing the induction, restriction, and projection functors on simple modules in $\TLn\mathrm{-mod}$. For an irreducible $\TLn$ representation $W^n_m$, the functors $\Ind$, $\Res$, and $P_k$ act as 
    \begin{align}
        \Ind(W^n_m)&\cong W^{n+1}_{m+1}\oplus \overline{\delta}_{m,0}W^{n+1}_{m-1},
        \\
        \Res(W^n_m)&\cong \overline{\delta}_{m,n} W^{n-1}_{m+1}\oplus \overline{\delta}_{m,0}W^{n-1}_{m-1},
        \\
        P_k(W^n_m)&\cong \delta_{m,n-2k}W^n_{n-2k}
    \end{align}
    from Appendix \ref{ssec:RepTLn}. We use the convention that $W^n_m=0$ if $m<0$, $m>n$, or $m\not\equiv n \operatorname{ mod }2$. We write $\overline{\delta}_{ij}\coloneq 1-\delta_{ij}$ to emphasize these conventions and to avoid confusion when composing these functors. With this presentation of induction, restriction, and projection it is straightforward to verify the isomorphisms in Table \ref{tab:isos} and Proposition \ref{prop:WeylIson}.\\
    
    It is important to remark that an instance of ``none'' in the table corresponds to an instance where the generic isomorphism implies a relation between a zero and a nonzero object, as is the case in $\Wact$. For example, if $n=2(k-1)$ so that $n+1-2k=-1$, then
    \begin{gather*}
        P_{k-1}(W^n_0)\cong W^n_0\,,\\
        P_{k-1} \circ \Res \circ P_k \circ\Ind \circ P_{k-1}(W^n_0)\cong 
        P_{k-1} \circ \Res \circ P_k (W^{n+1}_1)\cong 0\,.
    \end{gather*}
    Hence, there is no isomorphism between $P_{k-1}$ and $P_{k-1} \circ \Res \circ P_k \circ\Ind \circ P_{k-1}$.
\end{rem}

We begin the proof of Theorem \ref{thm:rep} by noting that some relations were proven in \cite{Quinn}, see Section \ref{ssec:Qdiagrams}. These include the isotopy relation and the first relations in \eqref{eq:TraceE} and \eqref{eq:CCWCircles}. We proceed by first investigating the action of diagrams which are generated only by Temperley-Lieb generators. In the representation, it is enough to show that the relations in $\uWeyl'$ hold for all choices of rightmost region labels $n\geq0$. The proof of the theorem is complete once the image of all relations in $\uWeyl'$ mapped to $\Hom_{\Wact}(n,m)$ and interpreted as equalities of bimodule maps (i.e. 2-morphisms in $\End(\T)$) are verified for all rightmost region labels (i.e. 2-objects in $\End(\T)$). 

\subsection{Action of Temperley-Lieb Generators}
The first relation in \eqref{eq:AlternatingCupCaps} as applied to $\mod{n+3}{n+3}{n}$ follows from the Temperley-Lieb relation $e_{n+2}'e_{n+1}'e_{n+2}'=e_{n+2}'$. Similarly,  $e_{n+1}'e_{n+2}'e_{n+1}'=e_{n+1}'$ implies the reflection of \eqref{eq:AlternatingCupCaps}. The second relation in \eqref{eq:AlternatingCupCaps} and its reflection are proven in the lemma below.

\begin{lem}
	\label{lem:RepAlternating}
	The following relations hold under the functor $\rho'$:
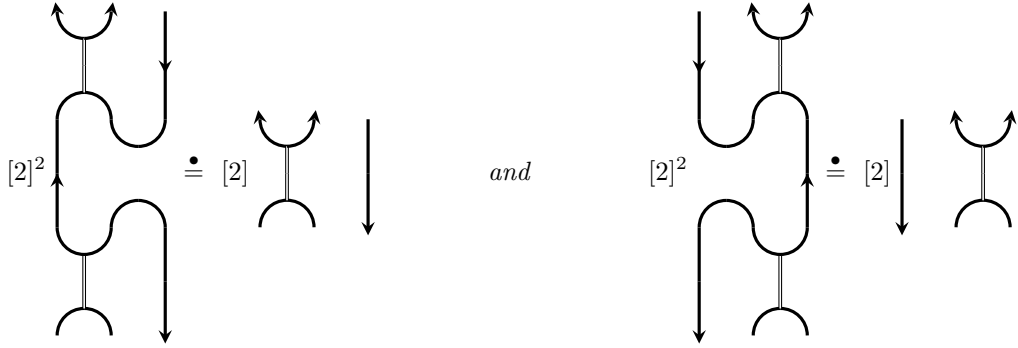
	\begin{align*}
		[2]^2~\tikzc{
			\cupcap{0}{0}
			\vseg[bottom<]{4}{0}
			\vseg{4}{2}
			\ncap{2}{4}
			\ncup{2}{8}
			\up[top>]{0}{4}
			\vseg{0}{6}
			\uucup{0}{12}\nncap{0}{8}
			\vseg{4}{8}
			\vseg[bottom<]{4}{10}
		}
		~\tleq~
		[2]~\tikzc{
			\uucup{0}{4}\nncap{0}{0}
			\vseg[bottom<]{4}{0}
			\vseg{4}{2}
		}
		&&\mbox{and}&&
		[2]^2~\tikzc{
			\cupcap{0}{0}
			\vseg[bottom<]{-2}{0}
			\vseg{-2}{2}
			\ncap{-2}{4}
			\ncup{-2}{8}
			\up[top>]{2}{4}
			\vseg{2}{6}
			\uucup{0}{12}\nncap{0}{8}
			\vseg{-2}{8}
			\vseg[bottom<]{-2}{10}
		}
		~\tleq~
		[2]~\tikzc{
			\uucup{0}{4}\nncap{0}{0}
			\vseg[bottom<]{-2}{0}
			\vseg{-2}{2}
		}
	\end{align*}
\end{lem}
\begin{proof}
	The relations in \eqref{eq:Qzigzag} imply the adjunction between induction and restriction. Therefore, the first relation is equivalent to
	\begin{align*}
		[2]^2~\tikzc{
			\cupcap{0}{0}
			\vseg[bottom<]{4}{0}
			\vseg{4}{2}
			\ncap{2}{4}
			\lcup{2}{8}
			\ncap{4}{8}
			\vseg{0}{4}
			\vseg{0}{6}
			\uucup{0}{12}\nncap{0}{8}
			\vseg{6}{0}
			\vseg{6}{2}
			\up{6}{4}
			\vseg{6}{6}
		}
		~\tleq~
		[2]~\tikzc{
			\uucup{0}{4}\nncap{0}{0}
			\lcap{4}{0}
		}
	\end{align*}
	Fix $n\geq1$. Then the left side of the above maps $x\otimes y\in \mod{n+2}{n+2}{}\mod{n}{n+1}{n}$ to $xe'_{n+1}ye'_{n+1}=x\ptr_{n+1}(y)e'_{n+1}$ which is exactly the image of $x\otimes y$ under the right morphism.
	
	The second relation is proven directly for $n\geq0$. The left side of the relation determines an endomorphism of the bimodule $\mod{n+1}{n+2}{n}$, mapping $x$ to $\ptr_{n+2}(xe'_{n+1})e'_{n+1}$. Using the Jones basis of the Temperley-Lieb algebra \cite{Ridout14}, $x=x_1+x_2e'_{n+1}x_3$ for some $x_1,x_2,x_3\in\TL_{n+1}$. Then as
	\begin{align*}
		\ptr_{n+2}(xe'_{n+1})e'_{n+1}
		&=(x_1\ptr_{n+2}(e'_{n+1})+x_2\ptr_{n+2}(e'_{n+1}x_3e'_{n+1}))e'_{n+1}
		\\
		&=x_1e'_{n+1}+x_2\ptr_{n+2}(\ptr_{n+1}(x_3))e'_{n+1}
		\\
		&=(x_1+x_2\ptr_{n+1}(x_3))e'_{n+1}
		\\
		&=(x_1+x_2e'_{n+1}x_3)e'_{n+1}=xe'_{n+1}
	\end{align*}
	the claim follows.
\end{proof}
\begin{rem}
	The Temperley-Lieb relation $(e'_{n+1})^2=e'_{n+1}$ implies the diagrammatic relation 
	\begin{align*}
			\tikzc{\nncap{0}{0}\uucup{0}{4}\nncap{0}{4}\uucup{0}{8}}
			~\tleq~
			\tikzc{\nncap{0}{0}\uucup{0}{4}}
	\end{align*}
	It is also implied by Lemma \ref{lem:RepAlternating} by closing the right strand of the diagram and resolving the counter-clockwise oriented circle as $[2]$.
\end{rem}

\subsection{Action of Bubbles} Before discussing additional relations involving boxes, namely those in which boxes are introduced i.e. \eqref{eq:CWCircles}, \eqref{eq:CCWCircles}, \eqref{eq:MoveBubbles}, and \eqref{eq:downKup}; we must first discuss the action of bubbles $\circled{k}$ beginning with the $0$ labeled circle. The latter equality in \eqref{eq:TraceE} follows from our definition of $\circled{0}$. Since $\circled{0}$ appears in the definition of both diagrams in \eqref{eq:DefZeroBubble} and are related to each other by \eqref{eq:BubbleDefs}, we compute both of their actions on $\T$ explicitly to show that $\circled{0}$ is well-defined. Quinn proved in \cite[Prop. 5.2.5]{Quinn} that the left diagram in \eqref{eq:DefZeroBubble} acts as an idempotent endomorphism of the induction bimodule \eqref{eq:Qloops}. Indeed, we show that it acts as multiplication by the Jones-Wenzl idempotent $f^{(n+1)}$ on $\mod{n+1}{n+1}{n}$.

\begin{prop}\label{prop:Consistent}
	In the representation $\rho'$, the definitions involving $\circled{0}$ in \eqref{eq:DefZeroBubble} are consistent with the notation introduced in the second equation of \eqref{eq:BubbleDefs}. That is,
	\begin{align*}
		\tikzc{\up[top>]{0}{0}\vseg{0}{-2}\drawP{0}{0}{0}}
		~=~
		{\tikzc{\up[top>]{0}{0}\vseg{0}{-2}}~-[2]~\tikzc{\vseg[top>]{0}{4}\ncap{2}{4}
				\vseg[bottom<]{4}{2}
				\vseg{4}{0}
				\cupcap{0}{0}\ncup{2}{0}
				\vseg[top>]{0}{-2}
		}}
		~\tleq~ 
		{ \tikzc{
				\widercup{0}{0}{2}
				\widecap{0}{0}{2}
				\ncup{1}{0}
				\rcap{1}{0}
				\drawP{1}{0}{0}
				\up[top>]{5}{0}\vseg{5}{-2}
		}}
		~=~
		\tikzc{\up[top>]{0}{0}\vseg{0}{-2}\drawP{-1}{0}{0}}
	\end{align*}
	and their action on $\mod{n+1}{n+1}{n}$ is multiplication by $f^{(n+1)}$.
\end{prop} 
\begin{proof}
	Quinn showed in Prop 5.2.5 that 
	\begin{align*}
		[2]~\tikzc{\vseg[top>]{0}{4}\ncap{2}{4}
			\vseg[bottom<]{4}{2}
			\vseg{4}{0}
			\cupcap{0}{0}\ncup{2}{0}
			\vseg[top>]{0}{-2}
		}~\tleq[n]\left[x\mapsto x\cdot\sum_{\substack{p,r\in P^n\\p_n=r_n}} c_r e_{p,r}e_n'e_{p,r}\right]:\mod{n+1}{n+1}{n}\to\mod{n+1}{n+1}{n}
	\end{align*}
	Expanding the path $r=(r',r_n)$ in the sum, we have
	\begin{gather*}
		\sum_{0\leq m\leq n}\sum_{p,r\in P^n_m} \frac{[r_{n-1}+1]}{[m+1]}\frac{1}{|P^{n-1}_{r_{n-1}}|}e_{p,r}e_n'e_{p,r}
		\\
		=
		\sum_{0<m\leq n}\sum_{\substack{p\in P^n_m\\ r'\in P^n_{m-1}}} \frac{[m]}{[m+1]}\frac{1}{|P^{n-1}_{m-1}|}v_p f^{(m)}e_m'f^{(m)}\check{v}_p
		+
		\sum_{0\leq m< n}\sum_{\substack{p\in P^n_m\\ r'\in P^n_{m+1}}} \frac{1}{|P^{n-1}_{m+1}|}v_p f^{(m+1)}\check{v}_p
		\\
		=
		\sum_{0\leq m\leq n}\sum_{\substack{p\in P^n_m}} 
		v_p
		\left(
		\frac{[m]}{[m+1]} f^{(m)}e_m'f^{(m)} 
		+
		f^{(m+1)}
		\right)\check{v}_p
		-
		\sum_{\substack{p\in P^n_n}} 
		v_pf^{(n+1)}\check{v}_p
		\\
		=\sum_{0\leq m\leq n}\sum_{\substack{p\in P^n_m}} 
		e_{p,p}
		-
		f^{(n+1)}
		\\
		=1-f^{(n+1)}
	\end{gather*}
	This proves
	\begin{align*}
		\tikzc{\up[top>]{0}{0}\vseg{0}{-2}\drawP{0}{0}{0}}~=~{\tikzc{\up[top>]{0}{0}\vseg{0}{-2}}~-[2]~\tikzc{\vseg[top>]{0}{4}\ncap{2}{4}
				\vseg[bottom<]{4}{2}
				\vseg{4}{0}
				\cupcap{0}{0}\ncup{2}{0}
				\vseg[top>]{0}{-2}
		}}~\tleq[n]
		\left[x\mapsto x\cdot f^{(n+1)}\right]
		:
		\mod{n+1}{n+1}{n}\to\mod{n+1}{n+1}{n}
	\end{align*}
	We use this to compute
	\begin{gather*}
		\tikzc{
			\ncup{1}{0}
			\rcap{1}{0}
			\drawP{1}{0}{0}
		}
		\tleq[n]
		\left[x\mapsto x\cdot
		\sum_{p,r\in P^n}c_re_{p,r}f^{(n)}e_{r,p}
		\right]
		=
		\left[x\mapsto x\cdot
		\frac{[n]}{[n+1]} f^{(n)}
		\right]:\mod{n}{n}{n}\to\mod{n}{n}{n}
	\end{gather*}
	The counterclockwise circle now acts as a partial trace and we obtain the desired equality below
	\begin{gather*}
		{ \tikzc{
				\widercup{0}{0}{2}
				\widecap{0}{0}{2}
				\ncup{1}{0}
				\rcap{1}{0}
				\drawP{1}{0}{0}
				\up[top>]{5}{0}\vseg{5}{-2}
		}}~\tleq[n]\left[x\mapsto x\cdot \ptr_{n+2}\left(\frac{[n+2]}{[n+3]}f^{(n+2)}\right)\right]
		=\left[x\mapsto x\cdot f^{(n+1)}\right]
		:
		\mod{n+1}{n+1}{n}\to\mod{n+1}{n+1}{n}
	\end{gather*}
\end{proof}

In the above proof we observed that 
\begin{gather}
	\tikzc{
		\ncup{1}{0}
		\rcap{1}{0}
		\drawP{1}{0}{0}
	}
	\tleq[n]
	\left[x\mapsto x\cdot
	\frac{[n]}{[n+1]} f^{(n)}
	\right]:\mod{n}{n}{n}\to\mod{n}{n}{n}
\end{gather}
The factor $\frac{[n]}{[n+1]}$, which is given by the action of $\boxed{-1}$ on an $n$ labeled region implies
\begin{gather*}
	\tikzc{
		\ncup{1}{0}
		\rcap{1}{0}
		\drawP{1}{0}{0}
	}
	~\tleq~
	\tikzc{
		\drawP{1}{0}{0}
		\drawsquare[-1]{4}{0}
	}
\end{gather*}
which is a consequence of \eqref{eq:CWCircles}. Notice that $\circled{0}$ acts by $f^{(n)}$. We will return to this relation after considering the action of more general bubbles $\circled{k}$.

\begin{lem}
	\label{lem:DoubledRelations}
	In the representation $\rho'$, the following relation holds
	\begin{align*}
		\tikzc{
			\ddcup{0}{9}\uucup{4}{9}
			\Widercup{1}{7}{2}
			\ddcap{0}{0}\uucap{4}{0}
			\Widelcap{1}{2}{2}}
		~\tleq~
		\tikzc{
			\ddcap{0}{0}\ddcup{0}{4}
			\uucap{4}{0}\uucup{4}{4}
		}
	\end{align*}
\end{lem}
\begin{proof}
	Consider $x\in \mod{n}{{n+2}}{n}$. The image of $x$ under the right morphism is $e_{n+1}xe_{n+1}$, which can be written as $e_{n+1}x'$ for some $x'\in\TL_{n}\hookrightarrow \TL_{n+2}$ using the Jones basis \cite[Prop. 2.3]{Ridout14}. Applying the left morphism to $x$ yields
	\begin{align*}
		x\mapsto e_{n+1}xe_{n+1}&\mapsto {\ptr}_{n+1}\circ {\ptr}_{n+2}(e_{n+1}xe_{n+1})
		\mapsto
		e_{n+1} \left({\ptr}_{n+1}\circ {\ptr}_{n+2}(e_{n+1}xe_{n+1})\right)e_{n+1}
		\\
		&=
		e_{n+1} \left({\ptr}_{n+1}\circ {\ptr}_{n+2}(e_{n+1}x')\right)e_{n+1}
		=
		e_{n+1} \left({\ptr}_{n+1}\circ {\ptr}_{n+2}(e_{n+1})\right)x'e_{n+1}
		\\&=
		e_{n+1} x'e_{n+1}
		=
		e_{n+1} x'
		=
		e_{n+1} xe_{n+1}.
	\end{align*}
	Thus proving the identity. 
\end{proof}

\begin{cor}
	\label{cor:isotypic}
	Let $M$ be a $\TLn$-module and recall that $\kimage[k]$ is image of the idempotent $\circled{k}$ in $\Wact$. Then $\rho(\kimage[k])(M)=P_k(M)$ is the projection onto the $W^n_{n-2k}$ isotypic component of $M$.
\end{cor} 
\begin{proof}
	We have already shown that $\circled{0}$ acts by $f^{(n)}$ on $\TL_{n}$-modules, and therefore is the projection onto the trivial isotypic component $W^n_{n}$. Thus we proceed by induction and show that $\circled{k}$ acts as $\sum_{p\in P^n_{n-2k}}e_{p,p}$ on $\TL_{n}\mathrm{-mod}$. By Proposition \ref{prop:Consistent} and Lemma \ref{lem:DoubledRelations},
	\begin{align*}
		\tikzc{\drawP{-2}{0}{k+1}\up{0}{0}\vseg{0}{-2}}
  ~\tleq~
  \tikzc{\vseg[top>]{0}{4}\ncap{2}{4}
			\vseg[bottom<]{4}{2}
			\vseg{4}{0}
			\cupcap{0}{0}\ncup{2}{0}
			\vseg[top>]{0}{-2}
			\drawP{2.5}{2}{k}
		}		
	\end{align*}
	It is therefore enough to compute the action of the latter on $\TLn\mathrm{-mod}$. As an endomorphism of $\mod{n+1}{n+1}{n}$, it acts via multiplication by
	\begin{gather*}
		\sum_{\substack{p,r\in P^n\\s\in P^{n-1}_{n-1-2k}}}
		\frac{[r_{n-1}+1]}{[r_n+1]}\frac{1}{|P^{n-1}_{r_n-1}|}
		e_{p,r}e_{s,s}e_{n}'e_{r,p}
		\\
		=
		\sum_{p\in P^n_{n-2k}}
		\frac{[n-2k]}{[n+1-2k]}
		v_p e_{n}'\check{v}_p
		+
		\sum_{p\in P^n_{n-2-2k}}
		v_p f^{(n-1-2k)}\check{v}_p
		\\
		=\sum_{p\in P^{n+1}_{n-1-2k}} e_{p,p}
		=
		\sum_{p\in P^{n+1}_{n+1-2(k+1)}} e_{p,p}
	\end{gather*}
	The second equality follows from Remark \ref{rem:idemp} and the claim is proven. 
\end{proof}

Consequently, the first relation in \eqref{eq:BubbleBoxRelations}
\begin{align}
	\tikzc{
		\drawP{0}{0}{k}
		\drawP{2}{0}{m}
	}
	\tleq\delta_{k,m}~
	\tikzc{
		\drawP{2}{0}{k}
	}
\end{align}
is satisfied since projections onto different isotypic components are orthogonal.

The relation (\ref{eq:InductProj}) implies the identities
\begin{align}
	\label{eq:PThroughUpN}
	\tikzc{
		\up[top>]{0}{0}\vseg{0}{-2}
		\drawP{-1.5}{0}{k}
	}~\tleq~
	\tikzc{
		\up[top>]{0}{0}\vseg{0}{-2}
		\drawP{-1.5}{0}{k}
	}
	\left(
	\tikzc{
		\drawP{0}{0}{k}
	}
	~+~
	\tikzc{
		\drawP{0}{0}{k-1}
	}
	\right)
	&&\mbox{and}&&
	\tikzc{
		\up[top>]{0}{0}\vseg{0}{-2}
		\drawP{1.5}{0}{k}
	}~\tleq
	\left(
	\tikzc{
		\drawP{0}{0}{k}
	}
	~+~
	\tikzc{
		\drawP{0}{0}{k+1}
	}
	\right)
	\tikzc{
		\up[top>]{0}{0}\vseg{0}{-2}
		\drawP{1.5}{0}{k}
	} 
\end{align}
These can also be inferred directly by considering the above relations as describing isomorphisms between modules:
\begin{align*}
	C_k\circ \Ind(M)\cong C_k\circ \Ind\circ\, C_k(M)\oplus C_k\circ \Ind\circ\, C_{k-1}(M)
	\\
	\Ind\circ\, C_k(M)\cong C_k\circ\Ind\circ\, C_k(M)\oplus C_{k+1}\Ind\circ\, C_k(M)\,.
\end{align*}

\begin{lem}\label{lem:nestedk}
	The following relations hold under the functor $\rho'$:
	\begin{align*}
	{\tikzc{
			\ncup{0}{0}
			\rcap{0}{0}
			\drawP{1}{0}{k}
	}}
	~\tleq~
	{\tikzc{
			\drawsquare[-2k-1]{0}{0}
			\drawP{3.5}{0}{k}
		}
		~
		+
		\left(
		~
		\tikzc{\drawsquare[-2k-2]{0}{0}}
		~\right)^{\dagger}
		\tikzc{
			\drawP{1}{0}{k+1}
	}}
&&
\mbox{and}
&&
	{\tikzc{
			\rcup{0}{0}
			\ncap{0}{0}
			\drawP{1}{0}{k}
	}}
	~\tleq~
	{\tikzc{
			\drawsquare[-2k+1]{0}{0}
			\drawP{4}{0}{k-1}
		}
		~
		+
		\left(~
		\tikzc{\drawsquare[-2k]{0}{0}}
		~
		\right)^{\dagger}
		\tikzc{
			\drawP{1}{0}{k}
	}}
	\end{align*}
\end{lem}
\begin{proof}
	For a given $n\geq0$, the action of the clockwise nested circle is multiplication by
	\begin{align*}
		\sum_{\substack{p,r\in P^n\\p_n=r_n}}\sum_{s\in P^{n-1}_{n-1-2k}} \frac{[p_{n-1}+1]}{[p_n+1]}\frac{1}{|P^{n-1}_{p_{n-1}}|}e_{p,r}e_{s,s}e_{r,p}
		= \sum_{p,r\in P^n_{n-2k}} \frac{[n-2k]}{[n-2k+1]}e_{p,p}
		+
		\sum_{p,r\in P^n_{n-2-2k}} \frac{[n-2k]}{[n-2k-1]}e_{p,p}\,.
	\end{align*}
	This agrees with the proposed equality, including the cases $n-1-2k<0$. 
	For the counterclockwise circle, it acts as multiplication by
	\begin{align*}
		\ptr_{n+1}\left(\sum_{p\in P^{n+1}_{n+1-2k}}e_{p,p} \right)
		&=
		\ptr_{n+1}\left(
		\sum_{p\in P^{n}_{n-2k}}v_pf^{(n+1)}\check{v}_p 
		+
		\sum_{p\in P^{n}_{n+2-2k}}\frac{[n-2k+2]}{[n-2k+3]}v_pe_{n}\check{v}_p
		\right)
		\\
		&=
		\sum_{p\in P^{n}_{n-2k}}\frac{[n-2k+2]}{[n-2k+1]}e_{p,p}
		+
		\sum_{p\in P^{n}_{n+2-2k}}\frac{[n-2k+2]}{[n-2k+3]}e_{p,p}
	\end{align*}
	as desired.
\end{proof}

Equation \eqref{eq:SumBubbles} together with Lemma \ref{lem:nestedk} imply that the relations for the unnested circles in \eqref{eq:CCWCircles} and \eqref{eq:CWCircles} hold. See also Remark \ref{rem:infsum}.

\subsection{Action of Bubbles with Temperley-Lieb Generators} 
We prove that relations involving both Temperley-Lieb generators and bubbles hold in the image of $\Wact'$ under $\rho'$. This will include the last of the defining relations from $\uWeyl'$, thus proving Theorem \ref{thm:rep}. The first of these is the second relation in \eqref{eq:DoubledRelations}, which states
\begin{align*}
	\tikzc{\uucup{0}{4}\nncap{0}{0}
		\drawP{-1}{2}{0}}~\tleq0
\end{align*}
This is indeed the zero map on $\mod{n+2}{n+2}{n}$, it is multiplication by $f^{(n+2)}e_{n+1}=0$.

To prove relation \eqref{eq:MoveBubbles} holds in the action of $\Wact'$ on $\T$, we first prove the following two lemmas.

\begin{lem}\label{lem:kek}
	The following identity holds
	\begin{align*}
		[2]~
		\tikzc{
			\up[top>]{0}{4}\up[top>]{2}{4}
			\cupcap{0}{0}
			\vseg{0}{-2}\vseg{2}{-2}
			\drawP{1}{0}{k}
			\drawP{1}{4}{k}
		}
		~\tleq~
		\tikzc{
			\up[top>]{0}{0}\up{2}{0}
			\vseg{0}{-2}\vseg{2}{-2}
			\drawP{-1.75}{0}{k+1}
			\drawP{1}{0}{k}
			\drawP{3}{0}{k}
		}
		\left(~
		\tikzc{\drawsquare[-2k]{0}{0}}
		~\right)^{-1}
		~+~
		\tikzc{
			\up[top>]{0}{0}\up{2}{0}
			\vseg{0}{-2}\vseg{2}{-2}
			\drawP{-1}{0}{k}
			\drawP{1}{0}{k}
			\drawP{3.75}{0}{k-1}
			\drawsquare[-2k+1]{8}{0}
		}
	\end{align*}
\end{lem}
\begin{proof}
	The left side of the equality acting on $\TLn\mathrm{-mod}$ is the endomorphism of $\mod{n+2}{n+2}{n}$ given via multiplication by
	\begin{gather*}
		\sum_{p,r\in P^{n+1}_{n+1-2k}} e_{p,p}e_{n+1}'e_{r,r}
		\\=
		\sum_{p\in P^{n}_{n-2k}} v_{p}f^{(n+1-2k)}e_{n+1}'f^{(n+1-2k)}\check{v}_p
		+
		\sum_{p,r\in P^{n}_{n+2-2k}} \left(\frac{[n+2-2k]}{[n+3-2k]}\right)^2v_p\cup_{n+2-2k}f^{(n+2-2k)}\cap_{n+2-2k}\check{v}_p\,.
	\end{gather*}
	The resulting expression is equivalent to that obtained from the right side of the equality. 
\end{proof}

\begin{rem}
	Lemma \ref{lem:kek} is a generalization of the Jones-Wenzl recursion. Using \eqref{eq:PThroughUpN} in the case $k=0$, one
	has the relation
	\begin{align*}
		[2]~
		\tikzc{
			\up[top>]{0}{4}\up[top>]{2}{4}
			\cupcap{0}{0}
			\vseg{0}{-2}\vseg{2}{-2}
			\drawP{2}{-.5}{0}
			\drawP{2}{4.5}{0}
		}
		\tikzc{\drawsquare[0]{0}{0}}~
		\tleq~
		\tikzc{
			\up[top>]{0}{0}\up{2}{0}
			\vseg{0}{-2}\vseg{2}{-2}
			\drawP{1}{0}{0}
			\drawP{3}{0}{0}
		}
		~-~
		\tikzc{
			\up[top>]{0}{0}\up{2}{0}
			\vseg{0}{-2}\vseg{2}{-2}
			\drawP{-1}{0}{0}
			\drawP{1}{0}{0}
			\drawP{3}{0}{0}
		}
	\end{align*}
	This implies the equality of multiplicative actions on $\mod{n+2}{n+2}{n}$:
	\begin{align*}
		\frac{[n]}{[n+1]}f^{(n+1)}e'_{n+1}f^{(n+1)}
		=
		f^{(n+1)}f^{(n)}-f^{(n+2)}f^{(n+1)}f^{(n)}=f^{(n+1)}-f^{(n+2)}\,. &\qedhere
	\end{align*}
\end{rem}

\begin{lem}\label{lem:kp1ek}
	The following identities hold in $\End(\T)$
	\begin{align*}
		{[2]~
		\tikzc{
			\up[top>]{0}{5.25}\up[top>]{3.5}{5.25}
			\widecupcap{0}{0}{1.75}
			\vseg{0}{-2}\vseg{3.5}{-2}
			\drawP{1.75}{5.5}{k+1}
			\drawP{1.75}{0}{k}
		}}
		~&\tleq~
	{	[2]~
		\tikzc{
			\up[top>]{0}{5.25}\up[top>]{3.5}{5.25}
			\widecupcap{0}{0}{1.75}
			\vseg{0}{-2}\vseg{3.5}{-2}
			\drawP{1.75}{5.5}{k+1}
			\drawP{3.5}{2.75}{k}
		}
		~-~
		\tikzc{
			\up[top>]{0}{0}\up{3.5}{0}
			\vseg{0}{-2}\vseg{3.5}{-2}
			\drawP{-1.75}{0}{k+1}
			\drawP{1.75}{0}{k+1}
			\drawP{4.5}{0}{k}
			\drawsquare[-2k-1]{8}{0}
		}}
	\\[2em]
		{[2]~
		\tikzc{
			\up[top>]{0}{5.25}\up[top>]{3.5}{5.25}
			\widecupcap{0}{0}{1.75}
			\vseg{0}{-2}\vseg{3.5}{-2}
			\drawP{1.75}{5.5}{k}
			\drawP{1.75}{0}{k+1}
		}
		}
		~&\tleq~
	{	[2]~
	\tikzc{
		\up[top>]{0}{5.25}\up[top>]{3.5}{5.25}
		\widecupcap{0}{0}{1.75}
		\vseg{0}{-2}\vseg{3.5}{-2}
		\drawP{1.75}{0}{k+1}
		\drawP{3.5}{2.75}{k}
	}
		~-~
		\tikzc{
			\up[top>]{0}{0}\up{3.5}{0}
			\vseg{0}{-2}\vseg{3.5}{-2}
			\drawP{-1.75}{0}{k+1}
			\drawP{1.75}{0}{k+1}
			\drawP{4.5}{0}{k}
			\drawsquare[-2k-1]{8}{0}
			}}
	\end{align*}
\end{lem}
\begin{proof}
	Since these relations are mirror images of each other across a horizontal axis, their proofs are also related by reversing the order of multiplication. Therefore, we only write the computation for the first relation. The left side of the equality acting on $\TLn\mathrm{-mod}$ is the endomorphism of $\mod{n+2}{n+2}{n}$  multiply by
	\begin{gather*}
		\sum_{
			\substack{
				p\in P^{n+1}_{n+1-2(k+1)}\\
				r\in P^{n+1}_{n+1-2k}
		}} 
		e_{p,p}e_{n+1}'e_{r,r}
		=
		\sum_{p\in P^n_{n-2k}}\frac{[n-2k]}{[n+1-2k]}
		v_pe'_{n-2k}e'_{n+1-2k}f^{(n+1-2k)}\check{v}_p
		\\=
		\sum_{p\in P^n_{n-2k}}\frac{[n-2k]}{[n+1-2k]}
		v_pe'_{n-2k}\check{v}_pe'_{n+1-2k}
		-
		\sum_{p\in P^n_{n-2k}}\left(\frac{[n-2k]}{[n+1-2k]}\right)^2
		v_p\cup_{n-2k}f^{(n-2k)}\cap_{n-2k}\check{v}_p\,.
	\end{gather*}
	This is the same expression obtained from the right side of the equality. 
\end{proof}

The following can now be determined from the above lemmas and \eqref{eq:PThroughUpN}.
\begin{align}\label{eq:MoveBubblesSetup}
	[2]~
{	\tikzc{
		\up[top>]{0}{4}\up[top>]{2}{4}
		\cupcap{0}{0}
		\vseg{0}{-2}\vseg{2}{-2}
		\drawP{1}{0}{k}
		\drawP{-1}{2}{k}
	}}
	~&\tleq~
	{[2]~\tikzc{
		\up[top>]{0}{4}\up[top>]{2}{4}
		\cupcap{0}{0}
		\vseg{0}{-2}\vseg{2}{-2}
		\drawP{1}{0}{k}
		\drawP{1}{4}{k}
		\drawP{-.5}{2}{k}
	}
	~+[2]~
	\tikzc{
		\up[top>]{0}{5.25}\up[top>]{3.5}{5.25}
		\widecupcap{0}{0}{1.75}
		\vseg{0}{-2}\vseg{3.5}{-2}
		\drawP{1.75}{0}{k}
		\drawP{1.75}{5.5}{k-1}
		\drawP{0}{2.75}{k}
	}
}
	\\
	\notag
	&{\tleq}~
	[2]~
{\tikzc{
	\up[top>]{0}{5.25}\up[top>]{3.5}{5.25}
	\widecupcap{0}{0}{1.75}
	\vseg{0}{-2}\vseg{3.5}{-2}
	\drawP{1.75}{5.5}{k-1}
	\drawP{0}{2.75}{k}
}
	+
	\tikzc{
		\up[top>]{0}{0}\up{2}{0}
		\vseg{0}{-2}\vseg{2}{-2}
		\drawP{-1}{0}{k}
		\drawP{1}{0}{k}
		\drawP{3.75}{0}{k-1}
		\drawsquare[-2k+1]{8}{0}
	}
	-
	\tikzc{
		\up[top>]{0}{0}\up{3.5}{0}
		\vseg{0}{-2}\vseg{3.5}{-2}
		\drawP{-1}{0}{k}
		\drawP{1.75}{0}{k-1}
		\drawP{5.25}{0}{k-1}
	}
	\left(~\tikzc{\drawsquare[-2k+2]{8}{0}}~\right)^{-1}}
\end{align}
\begin{cor}\label{cor:MoveBubblesN}
	The identity
	\begin{align*}
		[2]~
		\tikzc{
			\up[top>]{0}{4}\up[top>]{2}{4}
			\cupcap{0}{0}
			\vseg{0}{-2}\vseg{2}{-2}
			\drawP{1}{0}{k}
		}
		~\tleq~
		[2]~
		\tikzc{
			\up[top>]{0}{5.25}\up[top>]{3.5}{5.25}
			\widecupcap{0}{0}{1.75}
			\vseg{0}{-2}\vseg{3.5}{-2}
			\drawP{1.75}{5.5}{k+1}
			\drawP{3}{2.75}{k}
		}
		~+~
		[2]~
		\tikzc{
			\up[top>]{0}{5.25}\up[top>]{3.5}{5.25}
			\widecupcap{0}{0}{1.75}
			\vseg{0}{-2}\vseg{3.5}{-2}
			\drawP{1.75}{5.5}{k-1}
			\drawP{0}{2.75}{k}
		}~
		+~
		\tikzc{
			\up[top>]{0}{0}\up{2}{0}
			\vseg{0}{-2}\vseg{2}{-2}
			\drawP{-1.75}{0}{k+1}
			\drawP{1}{0}{k}
			\drawP{3}{0}{k}
		}
		\left(~\tikzc{\drawsquare[-2k]{8}{0}}~\right)^{\dagger}\cdots
	\end{align*}\vspace{-1\baselineskip}
	\begin{align*}\notag
		+~
		\tikzc{
			\up[top>]{0}{0}\up{2}{0}
			\vseg{0}{-2}\vseg{2}{-2}
			\drawP{-1}{0}{k}
			\drawP{1}{0}{k}
			\drawP{3.75}{0}{k-1}
		}
		~\tikzc{\drawsquare[-2k+1]{8}{0}}
		~-~
		\tikzc{
			\up[top>]{0}{0}\up{3.5}{0}
			\vseg{0}{-2}\vseg{3.5}{-2}
			\drawP{-1.75}{0}{k+1}
			\drawP{1.75}{0}{k+1}
			\drawP{4.5}{0}{k}
		}
		~\tikzc{\drawsquare[-2k-1]{8}{0}}
		~-~
		\tikzc{
			\up[top>]{0}{0}\up{3.5}{0}
			\vseg{0}{-2}\vseg{3.5}{-2}
			\drawP{-1}{0}{k}
			\drawP{1.75}{0}{k-1}
			\drawP{5.25}{0}{k-1}
		}
		\left(~\tikzc{\drawsquare[-2k+2]{8}{0}}~\right)^{\dagger}
	\end{align*}
	holds in the action of $\Wact'$ on $\T$ given by $\rho'$.
\end{cor}
\begin{proof}
	The relation follows from applying \eqref{eq:PThroughUpN}  to the left side of the proposed equality, followed by relations in Lemmas \ref{lem:kek} and \ref{lem:kp1ek} and equation \eqref{eq:MoveBubblesSetup}.
\end{proof}

Note that applying any of $\circled{k-1}$, $\circled{k}$, or $\circled{k+1}$ in the appropriate region to Corollary \ref{cor:MoveBubblesN} recovers the relations given in Lemmas \ref{lem:kek} and \ref{lem:kp1ek} and \eqref{eq:MoveBubblesSetup}. 

\begin{lem}\label{lem:downKupN}
	The identities hold
	\begin{align*}
		\tikzc{\vseg{-2}{2}\up[top>]{0}{2}\down[bottom<]{-2}{0}\vseg{0}{0}\drawP{-1}{2}{k}\drawP{1}{2}{k}\drawP{-3}{2}{k}}
		~\tleq~
		\tikzc{
			\vseg{0}{5}\up[top>]{2}{5}    
			\ncup{0}{5}
			\ncap{0}{2}
			\vseg{2}{0}\down{0}{0}
			\drawP{1}{5}{k}
			\drawP{1}{2}{k}
			\drawP{4}{5}{k}
			\drawsquare[-2k]{4}{2}
		}
		&&
		\tikzc{\vseg{-2}{2}\up[top>]{0}{2}\down[bottom<]{-2}{0}\vseg{0}{0}\drawP{-1}{2}{k}\drawP{2}{2}{k-1}\drawP{-4}{2}{k-1}}
		~\tleq~
		\tikzc{
			\vseg{0}{5}\up[top>]{2}{5}    
			\ncup{0}{5}
			\ncap{0}{2}
			\vseg{2}{0}\down{0}{0}
			\drawP{1}{5}{k}
			\drawP{1}{2}{k}
			\drawP{3.5}{3.5}{k-1}
		}\left(~
		\tikzc{\drawsquare[-2k+1]{0}{0}}
		~\right)^{-1}
	\end{align*}
\end{lem}
\begin{proof}
Fix $n\geq0$ and let $x\in\mod{n}{n+1}{n}$. If $n+1-2k<0$, then both claims follow from \eqref{eq:newrel} as each diagram is the zero morphism. We prove the first claim, and the other is proven similarly. The left side of the proposed equality maps $x$ to
	\begin{align*}
		\sum_{r\in P^{n}_{n-2k}}e_{r,r}
		\cdot x\cdot
		\sum_{p\in P^{n+1}_{n+1-2k}}e_{p,p}
		\sum_{s\in P^{n}_{n-2k}}e_{s,s}
		=
		\sum_{r\in P^{n}_{n-2k}}e_{r,r}
		\sum_{t,p\in P^{n+1}_{n+1-2k}}x^t_pe_{t,p}
		\sum_{s\in P^{n}_{n-2k}}e_{s,s}
		=
		\sum_{\substack{t,p\in P^{n+1}_{n+1-2k}\\t_{n-1}=p_{n-1}=n-2k}}
		x^t_pe_{t,p}
	\end{align*}
	where coefficients $x_p^t\in\bC(q)$ satisfy $xv_p=\sum_{t\in P^{n+1}_{n+1-2k}}x^t_pv_t$ for all $p\in P^{n+1}_{n+1-2k}$.
	
	Whereas on the right side, $x$ maps to\allowdisplaybreaks
	\begin{gather*}
		\frac{[n-2k+1]}{[n-2k+2]}
		\sum_{r\in P^{n}_{n-2k}}e_{r,r}
		\ptr_{n+1}\left(
		x\cdot
		\sum_{p\in P^{n+1}_{n+1-2k}}e_{p,p}
		\right)
		\sum_{s\in P^{n}_{n+1-2k}}e_{s,s}
		\\=
		\frac{[n-2k+1]}{[n-2k+2]}
		\ptr_{n+1}\left(
		\sum_{r\in P^{n}_{n-2k}}e_{r,r}
		\sum_{t,p\in P^{n+1}_{n+1-2k}}x^t_pe_{t,p}
		\right)
		\sum_{s\in P^{n}_{n+1-2k}}e_{s,s}
		\\=
		\frac{[n-2k+1]}{[n-2k+2]}
		\sum_{\substack{t,p\in P^{n+1}_{n+1-2k}\\t_{n-1}=p_{n-1}=n-2k}}x^t_p
		\frac{[n+2-2k]}{[n+1-2k]}
		e_{t',p'}
		\sum_{s\in P^{n}_{n+1-2k}}e_{s,s}
		\\=
		\sum_{\substack{t,p\in P^{n+1}_{n+1-2k}\\t_{n-1}=p_{n-1}=n-2k}}x^t_p
		e_{t,p}
	\end{gather*}
	Thus, equality holds. 
\end{proof}

\section{Basis Statements}\label{sec:bases}
We describe bases for $\End_{\uWeyl}(\mathbbm{1})$ and certain subalgebras of morphism spaces of $\Hom_{\Wact}(n,-)$. The arguments for proving these results rely on the action of $\Wact$ on $\End(\T)$. 

\subsection{The Box Algebra}\label{ssec:boxalg}
 We first describe the subalgebras generated by boxes. This discussion is relevant for understanding a basis of the subalgebra generated by boxes in $\uWeyl$. It does not apply to $\Wact$ where boxes have been identified with scalars.
{
\begin{defn}
Let $\boxalg$ denote the unital commutative algebra over $\mathbb{C}(q)$ with generators $\boxy{k}$ for $k\geq 0$ and relations  $\boxy{k+1}\boxy{k}=[2]\boxy{k+1}-1$. We call $\boxalg$ the \emph{box algebra}.
\end{defn}
}
 
Introduce the notation
 $
 \bdesc{k}{m}=\boxy{k}\boxy{k-1} \cdots\boxy{k-m}\in\boxalg
 $
 for $k\geq 0$ and $0\leq m\leq k$. Recall that nonnegative quantum integers may be defined recursively by
 \begin{align}
 \label{eq:mRecursion}
 	[m]=[2][m-1]-[m-2]
\end{align}
with $[1]=1$ and $[0]=0$, and where $[2]=q+q^{-1}$.

\begin{lem}\label{lem:BoxDesc}
    Fix $k\geq 0$, and $0\leq m\leq k$. Then the equality  
    \begin{align*}
        \bdesc{k}{m}=[m+1] \boxy{k}-[m]
    \end{align*}
    holds in $\boxalg$.
\end{lem}
 		
\begin{proof}
 We give a proof by induction on $m$ for any given $k\geq 0$. We can easily see that the case $m=0$ holds. Assume that the lemma holds for all pairs $k',m\in\bZ$ such that $0\leq k'\leq k$ and $0\leq m< k'$. Fix $m$ with $0\leq m< k$. Then $\bdesc{k}{m+1}
    =\boxy{k}\bdesc{k-1}{m}$. Since $b_{k}b_{k-1}=[2]b_{k}-1$, the claim now follows by induction:
\begin{align*}
    \bdesc{k}{m+1}
    &=\boxy{k}\bdesc{k-1}{m}
    =\boxy{k}([m+1] \boxy{k-1}-[m])
    =[m+1]([2]\boxy{k}-1)-[m]\boxy{k}
    =[m+2]\boxy{k}-[m+1].\qedhere
\end{align*}
\end{proof}

\begin{lem}\label{lem:BoxProd}
Fix $k\geq 0$, and $0<m\leq k$. Then the equality
\begin{align*}
    \boxy{k} \boxy{k-m}= \dfrac{1}{[m]}\left([m+1]\boxy{k}+[m-1]\boxy{k-m}-[m]\right).
\end{align*} 
 holds in $\boxalg$. 
\end{lem}
\begin{proof}
    By Lemma \ref{lem:BoxDesc},
\[
        [m+1] \boxy{k}-[m]
        =
        \bdesc{k}{m}
        =
        \bdesc{k}{m-1}\boxy{k-m}
        =
        ([m]\boxy{k}-[m-1])\boxy{k-m}\,.
\]
The claim follows by expanding the above and isolating $\boxy{k} \boxy{k-m}$.
 			\end{proof}
 		
\begin{lem}\label{lem:bPowSer}
    Let $m>-1$ and $p\geq0$ be integers. Then 
    \begin{align*}
 			\left(\frac{[m]}{[m+1]} \right)^p=q^p(1-q^{2m})^p\sum_{i=0}^\infty {i+p-1 \choose p-1 } q^{2i(m+1)}\,.
 			\end{align*}
 			\end{lem}
 			\begin{proof}
 			Write 
 		\begin{align*}
 		\frac{[m]}{[m+1]}=\frac{q^m-q^{-m}}{q^{m+1}-q^{-(m+1)}}=\frac{q(q^{2m}-1)}{q^{2(m+1)}-1}=q(1-q^{2m})\sum_{i=0}^\infty q^{2i(m+1)}.
 		\end{align*}
    Recall that the exponentiation of a geometric series is given by the series $\left(\sum_{i=0}^\infty r^i\right)^p=\sum_{i=0}^\infty {i+p-1\choose p-1}r^i$. The desired formula now follows by taking $r=q^{2(m+1)}$.
\end{proof}		
 			
For any integers $n$ and $k$ such that $n+k+2>0$, the identity
\begin{align}
    \frac{[n+k+1]}{[n+k+2]}\cdot\frac{[n+k]}{[n+k+1]}
    =
    \frac{[n+k]}{[n+k+2]}
    =\frac{[2][n+k+1]-[n+k+2]}{[n+k+2]}
    =[2]\frac{[n+k+1]}{[n+k+2]}-1
\end{align}
holds. Thus, for $n\geq -1$, there are algebra homomorphisms 
\begin{equation}\label{eq:boxmap}
\begin{aligned}
    \varrho_{n}:\boxalg&\to \bC(q)
    \\
    \boxy{k}&\mapsto\frac{[n+k+1]}{[n+k+2]}
\end{aligned}
\end{equation} 
Observe that if $n\geq 0$, then no generator is mapped to zero. However, this map is not an injection for any $n$.
 			
\begin{prop}\label{prop:BoxBasis}
    There is a basis of $\boxalg$ over $\mathbb{C}(q)$ spanned by
    \begin{align*}
        \{1\}\cup \left\{\boxy{k}^p:k\in\mathbb{Z}_{\geq 0}, p\in\mathbb{Z}_{>0}\right\}.
 			\end{align*}
 			\end{prop}
 		\begin{proof}
 		We determine that $\{1\}\cup\left\{\boxy{k}^p:k\in\mathbb{Z}_{\geq 0}, p\in\mathbb{Z}_{>0}\right\}$ is a spanning set by Lemma \ref{lem:BoxProd}, which shows the product of any two distinct $\boxy{k}$ and $\boxy{l}$ can be expressed as a $\mathbb{C}(q)$-linear combination of $\boxy{k}$, $\boxy{l}$, and $1$.
 		
 		To prove this spanning set is a basis, suppose there is a linear dependence $\sum_{k,p}a_{k,p}\boxy{k}^p=0$, where all but finitely many $a_{k,p}\in \mathbb{C}(q)$ are nonzero. Assume all $a_{k,p}\in\bC[q]$ by first clearing denominators then multiplying by some power of $q$. Let $a_{0,0}$ denote the coefficient of $1\in \boxalg$. 
 The image of this relation under $\varrho_{n}$ as in \eqref{eq:boxmap} implies $\sum_{k,p}a_{k,p}\left(\frac{[n+k+1]}{[n+k+2]}\right)^p=0$ for all $n>-2$.
 		
By Lemma \ref{lem:bPowSer}, this relation becomes
\begin{align*}
    \sum_{k,p}a_{k,p} q^p(1-q^{2(n+k+1)})^p
 		\sum_{i=0}^\infty {i+p-1 \choose p-1 }q^{2i(n+k+2)}=0\,.
 		\end{align*}
 			Set $a'_{k,p}=a_{k,p}q^p$ and  $x=x(n)=q^{2n}$. Thus,
 		\begin{align*}
 		\sum_{k,p}a'_{k,p} (1-xq^{2(k+1))})^p
 		\sum_{i=0}^\infty {i+p-1 \choose p-1 }q^{2i(k+2)}x^i=0\,.
 			\end{align*}
 		In this way, the coefficient of $x^j$ must be zero for all $j\geq 0$. We compute
 		\begin{gather*}
 		y_{k,p}(x,q):=(1-xq^{2(k+1))})^p
 			\sum_{i=0}^\infty {i+p-1 \choose p-1 }q^{2i(k+2)}x^i
 			\\=
 			\sum_{i=0}^p
 			\left(
 		\sum_{j=0}^i
 		{p\choose j}{i-j+p-1\choose i-j}q^{2(i-j)}
 		\right)q^{2i(k+1)}x^i
 			+
 		\sum_{i=p+1}^\infty
 		\left(
 		\sum_{j=0}^p
 			{p\choose j}{i-j+p-1\choose i-j}q^{2(i-j)}
 		\right)q^{2i(k+1)}x^i\,.
 		\end{gather*}
 		Let $P$ (and $K$) be the maximum value of $p$ (and $k$) among all nonzero $a'_{k,p}$ in the relation. Set $L=\max_{k,p}(\deg_q(a'_{k,p}))$ and $\ell=\max_{p}(\deg_q(a'_{K,p}))$.
 			
 			For a given $k$ and $p$, the coefficient of $x^i$ in the power series $y_{k,p}(x,q)$ is a {$\bC$-linear} combination of $q^{2i(k+2)}, q^{2i(k+2)-2}, \dots, q^{2i(k+2)-2p}$. For each $i$, suppose that these exists $k\neq K$ and $j,p$ such that $\deg_q(a'_{k,p}q^{2i(k+2)-2j})\geq \deg_q(q^{\ell+2i(K+2)})=\ell+2i(K+2)$. In which case $\deg_q(a'_{k,p}q^{-2j-\ell})\geq 2i(K-k)\geq 2i$. Since $\deg_q(a'_{k,p}q^{-2j-\ell})$ is independent of $i$ and is at most $L$, our above supposition fails for $i>L$. Thus, only terms in the expansion of some $a'_{K,p'}y_{K,p'}$ can cancel the $q^{\ell+2i(K+2)}x^i$ term in $a'_{K,p}y_{K,p}$.

            As $
            \deg_q(a'_{K,p})\leq \ell$ for all $p$, let $\hat a_{p}$ be the coefficient of $q^\ell$ in $a'_{K,p}$ for each $p$. Comparing coefficients of $q^{\ell+2i(K+2)} x^i$ in the relation, i.e. $k=K$ and $j=0$, we determine $\sum_{p}\hat a_{p}{i+p-1\choose i}=0$ for all $i>L$.

            Let $\hat P$ be the largest $p$ such that $\hat a_{p}$ is nonzero. Consider $i=L+\hat P$. By the binomial coefficient identity ${i+j\choose i}={i+j-1\choose i}+{i+j-1\choose i-1}$, we have
 			\begin{align*}
 		0&=
 	\sum_{p\leq \hat P}\hat a_{p}{(L+\hat P)+p-1\choose L+\hat P}
 	\\&=
    \sum_{p\leq \hat P}\hat a_{p}{(L+\hat P)+p-2\choose L+\hat P}
    +
 	\cancel{\sum_{p\leq \hat P}\hat a_{p}{(L+\hat P-1)+p-1\choose L+\hat P-1}}
    \end{align*}
    and iterating on this
    \begin{align*}
 		0=
 			\sum_{p\leq \hat P}\hat a_{p}{(L+\hat P)+p-\hat P\choose L+\hat P}
 			=
 			\sum_{p\leq \hat P}\hat a_{p}{L+p\choose L+\hat P}=\hat a_{\hat P}.
 			\end{align*}
 			Thus, we have reached a contradiction as we have assumed $\hat a_{\hat P}\neq0$. Hence there are no relations among the proposed basis elements.
 			\end{proof}
 			
\begin{prop}
    The map $\psi_{0,\boxalg}:\boxalg\to\End_{\uWeyl}(\mathbbm{1})$ which sends $\boxy{k}$ to the box labeled $k$ is an injection.
\end{prop}
\begin{proof}
Suppose that $\psi_{0,\boxalg}(\sum_{k,p}a_{k,p}\boxy{k}^p)=0$, where all but finitely many $a_{k,p}\in \mathbb{C}(q)$ are nonzero.
For each $n\geq0$ recall the homomorphisms $G_{n}:\End_{\uWeyl}(\mathbbm{1})\to2\Hom_{\Wact}(n,n)$ defined in Remark \ref{rem:1Weyl}. Recall the maps $\varrho$ and $\rho$ defined in \eqref{eq:boxmap} and preceding Theorem \ref{thm:rep} respectively. Let $\rho_n$ be the restriction of $\rho$ to $\End_{\uWeyl}(\mathbbm{1})$ with target $\End(\TLn\mathrm{-mod})$. There is a commutative diagram
    \begin{equation*}
        \begin{tikzcd}
\boxalg \arrow[r,"\psi_{0,\boxalg}"] \arrow[d,swap,"\varrho_{n}"] &
\End_{\uWeyl}(\mathbbm{1})
\arrow[d,"\rho_n\circ G_n"]
\\
\bC(q) \arrow[r,"\times"] & \End(\TLn\mathrm{-mod})
\end{tikzcd}
    \end{equation*}
    that maps $\boxy{k}$ to the endomorphism which multiplies by $\frac{[n+k+1]}{[n+k+2]}$. As shown in Proposition \ref{prop:BoxBasis}, the collection of maps $\varrho_n$ distinguishes all elements of $\boxalg$. That is, for all $b\in\boxalg$ such that $b\neq 0$, then there exists $n$ such that $\varrho_n(b)\neq 0$. Thus, all $a_{k,p}$ are zero and therefore $\psi_{0,\boxalg}$ is an injection.
\end{proof}

\subsection{Endomorphism Algebras}\label{ssec:endalg}
We give a description of $\End_{\uWeyl}(\mathbbm{1})$ and $\End_{\Wact}(Q^{\otimes s}_{+,n})$ as algebras.
\begin{defn}
    Let $\projalg$ denote the commutative algebra of orthogonal idempotents $\langle c_0, c_1, \dots  \mid c_kc_l=\delta_{kl}c_{k}\rangle$ over $\mathbb{C}(q)$. Let $\projalg_{n}$ be the subquotient generated by $c_0,c_1,\dots, c_{\floor{n/2}}$ with the additional relation $1=c_0+c_1+\dots+ c_{\floor{n/2}}$.
\end{defn} 

\begin{prop}\label{prop:projalginj}
   The natural map $\psi_{0,\projalg,n}:\projalg_{n}\to \End_{\Wact}(\mathbbm{1}_n)$ is an injection.
\end{prop}
\begin{proof}
    To prove that $\psi_{0,\projalg,n}$ is injective, assume that $\psi_{0,\projalg,n}(\sum_k a_k c_k)=0$ where each $a_k\in\mathbb{C}(q)$. Since $\psi_{0,\projalg,n}$ is a homomorphism to $\Hom_{\Wact}(n,n)$, we consider the composition of $\psi_{0,\projalg,n}$ with $\rho_n$. Then each $c_k$ acts as the projection onto the isotypic component of type $W^n_{n-2k}$ according to Corollary \ref{cor:isotypic}. Since $k\leq\floor{n/2}$, the generators $c_0, c_1, \dots, c_{\floor{n/2}}$ act by nonzero projections onto distinct simple summands. Therefore, they act linearly independently on the direct sum of modules $\bigoplus_{i=0}^{\floor{n/2}} W^n_{n-2i}$. It now follows that the coefficients $a_k$ must all be zero.
\end{proof}

\begin{cor}
     Let $\psi_{0,\projalg}:\projalg\to\End_{\uWeyl}(\mathbbm{1})$ be the map which sends $c_k$ to the bubble labeled $k$. Then $\psi_{0,\projalg}$ is an injection.
\end{cor}
\begin{proof}
    The proof is similar to that of Proposition \ref{prop:projalginj}. By composing $\psi_{0,\projalg}$ with $\rho_K\circ G_K$, with $G_K$ as defined in Remark \ref{rem:1Weyl}, for $K$ sufficiently large, the generators $c_k$ act linearly independently on the $\TL_K$-module $\bigoplus^{\floor{K/2}}_{k=0}W^K_{K-2k}$. 
\end{proof}

\begin{prop}\label{prop:projalgiso}
    The map $\psi_{0,\projalg,n}$ is an isomorphism.
\end{prop}
\begin{proof}
    It remains to show that $\psi_{0,\projalg,n}$ is a surjection. Any endomorphism of $\mathbbm{1}_n$ is a $\bC(q)$-linear combination of products of bubbles $\circled{k}$ and oriented trivalent graphs without boundary points. Each edge of each graph is either solid or doubled, and each vertex joins two solid segments with the same orientation and one doubled segment with the opposite orientation. Here orientation means outward or inward with respect to the vertex. Each region in the complement of a graph contains at most one bubble $\circled{k}$ by \eqref{eq:BubbleBoxRelations}.
    
    Notice that each connected component of a solid (non-doubled strand) in a graph is a closed curve with an even number of vertices (possibly zero). Consider such a component which also contains no solid curves in the region it bounds. Resolve any closed doubled curves on its interior according to Lemma \ref{lem:DoubledCircles} and evaluate any resulting boxes according to \eqref{eq:newrelBox}. If the solid curve has any doubled segments on its interior, consider a region bounded by the solid curve and a doubled edge which has no other doubled edges on its interior. There may be an even number of vertices on the boundary of this region, at each such vertex there is a doubled segment on the exterior of the region. These vertices may be removed in pairs using \eqref{eq:AlternatingCupCaps}. Now apply \eqref{eq:TraceE}, and Lemma \ref{lem:PThroughE} if necessary, to reduce the number of doubled edges in the region bounded by the solid curve. Repeat until there are no such edges remaining. Relation \eqref{eq:AlternatingCupCaps} may be applied to remove any remaining pairs of exterior doubled edges from the solid closed curve. The solid curve has been reduced to a $\bC(q)$ multiple of an oriented circle, which can be further simplified using \eqref{eq:CCWCircles} and \eqref{eq:CWCircles}.
 		
 		Apply the above to remove all solid curves from the graph. The resulting diagram may contain closed doubled curves which can be simplified using \eqref{eq:DoubledRelations}. The resulting diagram is now a linear combination of products of bubbles. Products may be simplified using \eqref{eq:BubbleBoxRelations}, by \eqref{eq:newrel} only bubbles with sufficiently small index are nonzero, and by \eqref{eq:SumBubbles} sums of bubbles may be simplified to $1$. Thus, $\psi_0$ is a surjection.
 	\end{proof}

  \begin{cor}
      The map $\psi_{0,\boxalg}\otimes \psi_{0,\projalg}:\boxalg\otimes\projalg\to\End_{\uWeyl}(\mathbbm{1})$ is an isomorphism of commutative algebras.
  \end{cor}
  \begin{proof}
      The proof is identical to that of Proposition \ref{prop:projalgiso} with the exception that boxes are not evaluated as constants, but are instead moved to the unbounded region of any given diagram using \eqref{eq:BubbleBoxRelations}.
  \end{proof}

\begin{defn}
    For $n>0$, write $\projalg_{n,s}$ to denote the unital commutative algebra generated by $c_{i,j}$ for $0\leq i\leq \floor{(n+j)/2}$ and $0\leq j\leq s$ with relations 
\begin{align}
    c_{i,j}c_{l,j}
    &=\delta_{il}c_{i,j}, 
    &
    \sum_{i=0}^{\floor{(n+j)/2}}c_{i,j}&=1
    \\
    \delta_{j>0}c_{i,j}
    &=c_{i,j}(c_{i,j-1}+c_{i-1,j-1}),
    &
    \delta_{j<s}c_{i,j}
    &=(c_{i,j+1}+c_{i+1,j+1})c_{i,j} \label{eq:ProjAdjRel}
\end{align}
where $c_{i,j}$ is assumed to be zero if $i<0$ or $i>\floor{(n+j)/2}$. 
\end{defn}
For each $0\leq j\leq s$, there is an inclusion of $\projalg_{n+j}\hookrightarrow \projalg_{n,s}$ which maps the generator $c_i$ to $c_{i,j}$. Indeed this map is an inclusion, as
\[
c_{i,j}=\prod_{k=-j}^{s-j}c_{i+k,j+k}+(\text{terms with some factor $c_{i',j'}$ where each $|i'-j'|<|i-j|$})\,.
\]
In this way we may present $\projalg_{n,s}$ as $\projalg_{n+s}\otimes \cdots \otimes \projalg_{n+1}\otimes \projalg_{n}$ together with the relations in \eqref{eq:ProjAdjRel}. 

\begin{lem}\label{lem:CBasis}
    A basis of $\projalg_{n,s}$ given by sequences
\begin{align*}
    \{c_{\boldsymbol{k}}=c_{k_s,s}\cdots c_{k_1,1}c_{k_0,0}:0\leq k_j\leq \floor{(n+j)/2}, k_{j+1}\in\{k_j,k_j+1\}\}.
\end{align*}
\end{lem}
For example, for $0\leq i\leq \floor{n/2}$ we may express $c_{i,0}c_{i,2}\in \projalg_{n,3}$ in this basis as
\begin{align*}
    c_{i,0}c_{i,2}&=c_{i,2}c_{i,0}=\left[(c_{i+1,3}+c_{i,3})c_{i,2}(c_{i,1}+c_{i-1,1})\right]\left[(c_{i+1,1}+c_{i,1})c_{i,0}\right]=
(c_{i+1,3}+c_{i,3})c_{i,2}c_{i,1}c_{i,0}
\\&=c_{i+1,3}c_{i,2}c_{i,1}c_{i,0}+c_{i,3}c_{i,2}c_{i,1}c_{i,0},
\end{align*}
where in the second equality we applied the relations \eqref{eq:ProjAdjRel}.
\begin{defn}\label{defn:TLs+n}
    Let $\TL^+_s(n)$ be a $\mathbb{C}(q)$-algebra with generators $\ckme_j$ and $c_{k,i}$ for $1\leq j\leq s-1$,  $0\leq k\leq \floor{(n+j)/2}$, and $0\leq i\leq s$. Each respective set of generators is subject to relations such that they generate $\TL_s$ and $\projalg_{n,s}$ as subalgebras. The other defining relations are:
\begin{align}
    \ckme_jc_{k,i} 
    &= 
    c_{k,i}\ckme_j,  
    &\mbox{for $i\neq j$}
    \\
    \label{eq:CthroughEAlgebraic}
    c_{k,j-1}\ckme_j
    &=
    c_{k+1,j+1}c_{k,j-1}\ckme_j
    =
    c_{k+1,j+1}\ckme_j, 
\end{align}	
\vspace{-1\baselineskip}
\begin{align}\label{eq:BigRelEC}
    \ckme_i c_{k,i}
    &=
    c_{k+1,i+1}c_{k+1,i}c_{k,i-1}\left(\ckme_i-\frac{[n-2k]}{[n-2k+1]}
    \right)
    +
    c_{k,i+1}c_{k-1,i}c_{k-1,i-1}\left(\ckme_i- \frac{[n-2k+4]}{[n-2k+3]}\right)
    \\&
    \phantom{=}+ 
    c_{k+1,i+1}c_{k,i}c_{k,i-1} \frac{[n-2k+2]}{[n-2k+1]}
    + 
    c_{k,i+1}c_{k,i}c_{k-1,i-1}
    \frac{[n-2k+2]}{[n-2k+3]}
    \notag.
\end{align}
\end{defn}
The relation 
\begin{align}
    \ckme_ic_{k,i}\ckme_i
    &=
    \left(
    c_{k-1,i-1}
    \frac{[n-2k+2]}{[n-2k+3]}
    +
    c_{k,i-1}
    \frac{[n-2k+2]}{[n-2k+1]}
    \right)
    \ckme_i
\end{align}
is readily verified using (\ref{eq:BigRelEC}),
$c_{k+1,i+1}(c_{k+1,i}+c_{k,i})c_{k,i-1}=c_{k+1,i+1}c_{k,i-1}$, 
and $(\ckme_i)^2=[2]\ckme_i$.

The Jones basis of $\TL_s$ for nonidentity elements consists of words
\begin{align}\label{eq:JonesBasis}
    \ckme_{\boldsymbol{jk}}=(\ckme_{j_1}\ckme_{j_1+1}\cdots \ckme_{k_1})(\ckme_{j_2}\ckme_{j_2+1}\cdots \ckme_{k_2})
\cdots
(\ckme_{j_r}\ckme_{j_r+1}\cdots \ckme_{k_r})
\end{align}
where $s>j_1>j_2>\cdots>j_r>0$ and $s>k_1>k_2>\cdots>k_r>0$ \cite[Prop. 2.3]{Ridout14}. We write the identity element as $\ckme_{00}$.
    
\begin{prop}\label{prop:TLBasis}
There is a basis of $\TL_s^+(n)$ consisting of products $c_{\boldsymbol{k}}e'_{\boldsymbol{ij}}$ where $c_{\boldsymbol{k}}$ is a basis vector of $\projalg_{n,s}$ as in Lemma \ref{lem:CBasis} and $e'_{\boldsymbol{ij}}$ is a basis vector of $\TL_s$ as in \eqref{eq:JonesBasis} with the additional constraints:
\begin{itemize}
    \item if $\ckme_\ell\ckme_{\boldsymbol{ij}}=[2]\ckme_{\boldsymbol{ij}}$, then  $\boldsymbol{k}_{\ell+1}=\boldsymbol{k}_{\ell-1}+1$\,,
    \item if $n+j-2i=0$ and $c_{i+1, j+2}c_{i,j+1}c_{i,j}c_{\boldsymbol{k}}=c_{\boldsymbol{k}}$, then $e'_{j+1}e'_{\boldsymbol{ij}}\neq [2]e'_{\boldsymbol{ij}}$.
\end{itemize}
\end{prop}
\begin{proof}
The defining relations in $\TL_s^+$ give a rule to present any element of $\TL_s^+$ in the desired form. The additional constraints are due to \eqref{eq:CthroughEAlgebraic} and the exceptional form of \eqref{eq:BigRelEC}, see also Remark \ref{rem:Exceptional}. It remains to show that any ambiguities in applying these rules can be resolved in the sense of Bergman's diamond lemma \cite{Bergman}. The ambiguities arising from relations which relate one monomial to another are easily resolved. Ambiguities for the relation involving $\ckme_ic_{k,i}$, namely the confluence conditions for $(\ckme_i)^2c_{k,i}$, $\ckme_i\ckme_{i\pm1}\ckme_ic_{k,i}$, $\ckme_ic_{k,i}^2$, and $\ckme_ic_{k,i}(c_{k,i\pm1}+c_{k\pm1,i\pm1})$ are resolved by short computations.
\end{proof}

Let $Q_{+,n}^{\otimes s}$ denote the object $Q_+^{\otimes s}\in \Hom_{\Wact}(n,n+s)$. Write  $\End_{\Wact}(Q_{+,n}^{\otimes s})$ for the algebra of endomorphisms of this object. The assignment of diagrams to elements of $\TL_s^+(n)$ naturally determines an algebra map $\psi_{s,n}:\TL_s^+(n)\to \End_{\Wact}(Q_{+,n}^{\otimes s})$. Recall that strands are labeled so that strand $1$ is on the right and strand $s$ is on the left. For generators indexed by $0<i<s$, $\psi_{s,n}$ makes the following assignments:
\begin{align}
    e_i&\mapsto\mbox{disoriented cupcap over positions $i$ and $i+1$,}
    \\
    c_{k,i} &\mapsto \circled{k} \mbox{ on (or immediately to the left of) the top of strand $i$,}
    \\
    c_{k,0} &\mapsto \circled{k} \mbox{ right of strand $1$.}
\end{align}
Multiplying on the left in $\TL_s^+(n)$ corresponds to stacking upwards in diagrams under $\psi_{s,n}$. 
The existence of this map can be verified by comparing the defining relations of $\TL_s^+(n)$ to the relations on $\End_{\Wact}(Q_{+,n}^{\otimes s})\subset\Hom_{\Wact}(n,n+s)$. 
 			
\begin{prop}
The map $\psi_{s,n}$ is a surjection.
\end{prop}
\begin{proof} 
Any diagram in $\End_{\Wact}(Q_{+,n}^{\otimes s})$ can be reduced to a diagram in a standard form. Any closed subgraph in a diagram may be simplified to a linear combination of bubbles by Proposition \ref{prop:projalgiso}. A bubble that appears directly below a disoriented cupcap may be moved above it and have the cupcap resolve according to \eqref{eq:MoveBubbles}. The result is a linear combination of Temperley-Lieb diagrams with at most one bubble in each region across the top of the complement of the diagram (including its sides). If there are no bubbles in these regions, we introduce a sum of bubbles into the rightmost region (any region is fine) using \eqref{eq:SumBubbles}. We then use \eqref{eq:PThroughUp} to guarantee that each region across the top of the complement of the diagrams has a bubble. It is now easy to determine an element of $\TL_s^+(n)$ which maps to this linear combination of diagrams.
\end{proof}
 	
\begin{conj}
    The map $\psi_{s,n}$ is an injection.
\end{conj}

Proving that $\psi_{s,n}$ is injective is equivalent to showing that its kernel is generated by the relations of $\TL_s^+(n)$ given in Definition \ref{defn:TLs+n}. This proof requires studying the action of $\TL_s$ on $Q_+^{\otimes s}$, which we discuss briefly in the next section.  
 	
\subsection{Action of $\TL_s$ on $Q_{+,n}^{\otimes s}$}\label{sec:ends}

It follows from the first relation in \eqref{eq:AlternatingCupCaps} that the identification of idempotent Temperley-Lieb generators $e_i$ with the disoriented cupcap over strands $i$ and $i+1$ (numbered right to left) defines a homomorphism $\TL_s\to\End_{\uWeyl}(Q_+^{\otimes s})$. Similarly, by rotating diagrams there is a homomorphism $\TL_s\to\End_{\uWeyl}(Q_-^{\otimes s})$. 
 			
 	Now each idempotent in $\TL_s$ defines an object in the idempotent complete category $\uWeyl$. These include the projections onto irreducible representations discussed in Section \ref{ssec:RepTLn}, among which the Jones-Wenzl idempotent $f^{(s)}$ is a special case. The simplest of these idempotents and their complements are $e_1,1-e_1\in\TL_2$. 
 	
 	The action of $Q_{+,n}^{\otimes2}\kimage[k]$ on a $\TLn$-module is to project onto the $W^n_{n-2k}$ component and induct twice. From Section \ref{ssec:IndRes}, $\Ind^2(W^n_{n-2k})\cong W^{n+2}_{n+2-2k}\oplus 2W^{n+2}_{n-2k}\oplus W^{n+2}_{n-2-2k}$. Set $m=n-2k$. Then $v\in W^n_{m}$ is mapped to the following direct sum of diagrams under the isomorphism:
 	\begin{align}
 		\left(~
 		\tikzc{
 			\vseg{1}{-.5}\vseg{0}{-.5}\vseg{-1}{-.5}\vseg{-2}{-.5}\vseg{-3}{-.5}
 			\vseg{1}{-1.5}\vseg{0}{-1.5}\vseg{-1}{-1.5}\vseg{-2}{-1.5}\vseg{-3}{-1.5}
 			\vseg{1}{-4}\vseg{0}{-4}\vseg{-1}{-4}
 			\vseg{-2}{-4}\vseg{-3}{-4}
 			\drawsquaretext[\phantom{x}$v$\phantom{x}]{0}{0}
 			\drawsquaretext[\phantom{x}$f^{(m+2)}$\phantom{x}]{-1}{-2.5}
 		}~,~
 	\sqrt{\frac{[m]}{[m+1]}}~
 	\tikzc{
 		\vseg{1}{-.5}\vseg{0}{-.5}\vseg{-1}{-.5}\vseg{-3}{-.5}\vseg{-4}{-.5}
 		\vseg{1}{-1.5}\vseg{0}{-1.5}\vseg{-1}{-1}\vseg{-3}{-1}\vseg{-4}{-1.5}
 		\vseg{1}{-3}\vseg{0}{-3}\vseg{-4}{-3}
 		\vseg{1}{-5}\vseg{0}{-5}\vseg{-4}{-5}
 		\drawsquaretext[\phantom{x}$v$\phantom{x}]{0}{0}
 		\drawsquaretext[\phantom{xxx}$f^{(m)}$\phantom{xxx}]{-1.5}{-3.5}
 		\ncup{-3}{-1}
 	}
 	~,~
 		\sqrt{\frac{[m+1]}{[m+2]}}~
 	\tikzc{
 		\vseg{1}{-.5}\vseg{0}{-.5}\vseg{-1}{-.5}\vseg{-2}{-.5}\vseg{-4}{-.5}
 		\vseg{1}{-1.5}\vseg{0}{-1.5}\vseg{-1}{-1.5}
 		\vseg{-2}{-1.5}\vseg{-4}{-1.5}
 		\vseg{1}{-3}\vseg{0}{-3}\vseg{-1}{-3}
 		\vseg{-2}{-4}\vseg{-4}{-3}
 		\vseg{1}{-5}\vseg{0}{-5}\vseg{-1}{-5}
 		\drawsquaretext[\phantom{x}$v$\phantom{x}]{0}{0}
 		\drawsquaretext[$f^{(m+1)}$]{-.5}{-2.5}
 		\ncup{-4}{-4}\vseg{-4}{-4}
 	}
 	~,~
 	\sqrt{\frac{[m]}{[m+2]}}~
 	\tikzc{
 		\vseg{1}{-.5}\vseg{0}{-.5}\vseg{-1}{-.5}\vseg{-3}{-.5}\vseg{-4}{-.5}
 		\vseg{1}{-1.5}\vseg{0}{-1}\vseg{-1}{-1}\vseg{-3}{-1}\vseg{-4}{-1}
 		\vseg{1}{-3}
 		\vseg{1}{-5}
 		\drawsquaretext[\phantom{x}$v$\phantom{x}]{0}{0}
 		\ncup{-3}{-1}
 		\widecup{-4}{-1}{2}
 	}
 		~\right)
 	\end{align}
 	
 	The action of $e_1\in\TL_2$ on $Q_{+,n}^{\otimes 2}\kimage[k]$ determines an endomorphism of $\Ind^2(W^n_m)$ under $\rho$, namely multiply by $e_{n+2}$. Under this action, both the first and last summands vanish as capping a Jones-Wenzl idempotent yields zero and $vf^{(m)}=v\in W^n_m$. This observation is consistent with the relation in Lemma \ref{lem:PThroughE}, which implies a projection onto the $W^{n+2}_{n+2-2(k+1)}=W^{n+2}_{n-2k}$ isotypic component. The images of these vectors are both nonzero multiples of 
 	\begin{align}
 		\tikzc{
 			\vseg{1}{-.5}\vseg{0}{-.5}\vseg{-1}{-.5}
 			\vseg{1}{-1.5}\vseg{0}{-1.5}\vseg{-1}{-1.5}
 			\vseg{1}{-4}\vseg{0}{-4}\vseg{-1}{-4}
 			\drawsquaretext[\phantom{x}$v$\phantom{x}]{0}{0}
 			\drawsquaretext[$f^{(m)}$]{0}{-2.5}
 			\ncup{-4}{1.5}
 		}
 	\end{align}
 			
 	Note that the complement $1-e_1$ of this idempotent endomorphism on $Q_{+,n}^{\otimes 2}$ does not act as a projection onto the complementary summands $W^{n+2}_{n+2-2k}\oplus W^{n+2}_{n-2-2k}$. This is because $1-e_{n+2}$ is nonzero on the $W^{n+2}_{n-2k}$ summands of $\Ind^2(W^n_{n-2k})$. Further investigation of the interaction between the images of idempotents in $\TL_s$ under $\rho$ and the idempotents $\kimage[k]$ requires a more careful treatment.

\begin{rem}
	The affine Temperley-Lieb algebra also acts on $Q_+^{\otimes s}$. In future work, we use this action to compute $\operatorname{Tr}(\uWeyl)$. 
\end{rem}

\section{The Asymptotic Weyl Category and its Grothendieck Ring}\label{sec:kzero}
For each $n,m\geq0$, we defined the morphism categories $\Hom_{\Wact}(n,m)$ of $\Wact$, as a quotient of $\uWeyl$ in which bubbles $\circled{k}=0$ for $2k>n$ and $\boxed{k}=0$ for $n+k+1\leq 0$. The asymptotic Weyl category is a monoidal category which should be interpreted as a $n\to\infty$ version of the categories $\Hom_{\Wact}(n,m)$. In this category, no bubbles are assumed to be zero and all boxes are invertible.

\begin{defn}\label{def:aWeyl}
    The \emph{asymptotic Weyl category} $\aWeyl$ is the quotient of the universal Weyl category $\uWeyl$ by the relations
    \begin{align}
\left(~\tikzc{
		\drawsquare[k]{0}{0}
	}~\right)^{\dagger}=
	\left(~\tikzc{
		\drawsquare[k]{0}{0}
	}~\right)^{-1}\mbox{for all $k\in\bZ$.}
\end{align}
\end{defn}

Given that $\aWeyl$ is a quotient of $\uWeyl$, the isomorphisms given in Section \ref{ssec:WeylIsos} hold. Moreover, the ``Above critical values'' (large $n$) isomorphisms in Table \ref{tab:isos} hold in $\aWeyl$. Indeed, no bubbles or boxes vanish among the morphisms witnessing the isomorphisms on objects given in the table. It is in this sense that we may also consider $\aWeyl$ a $n\to\infty$ version of the categories $\Hom_{\Wact}(n,m)$.

\begin{prop}\label{prop:AIsos}
    The isomorphisms in Section \ref{ssec:WeylIsos} and the large $n$ relations in \ref{sec:2WeylIsos} hold in the full subcategory of $\aWeyl$ generated by $Q_+$, $Q_-$, and $C_k$ for $k\geq0$.
\end{prop}

\subsection{Grothendieck Ring}\label{ssec:K0}
Recall that the Grothendieck group of an additive category $\mathcal{C}$ is abelian group $K_0(\mathcal{C})$ with generators $X$ for $X\in\operatorname{Ob}(\mathcal{C})$ and relations $[Z]=[X]+[Y]$ whenever $Z\cong X\oplus Y$. The Grothendieck group is endowed with the structure of a ring with multiplication defined such that $[X][Y]=[X\otimes Y]$.

\begin{defn}\label{defn:subalg}
	Let $\subalg$ be the algebra generated by $x$, $y$ and $c_k$ for $k\geq0$ with relations
	\begin{align}
		yx&=xy+c_0
		&
		c_kc_\ell&=\delta_{k\ell}c_k\\
		c_kx~&=c_kx(c_k+c_{k-1})&
		xc_k~&=(c_k+c_{k+1})xc_k\label{eq:XC}\\
		c_ky~&=c_ky(c_k+c_{k+1})&
		yc_k~&=(c_k+c_{k-1})yc_k\label{eq:YC}\\
		c_{k+1}xc_kxc_k&=c_{k+1}xc_{k+1}xc_k& c_{k}yc_{k+1}yc_{k+1}&=c_{k}yc_kyc_{k+1}\label{eq:XCX}\\
		c_ky c_kx c_k&=c_k &
		c_{k-1}y c_kx c_{k-1}&=c_{k-1}\\
		c_kx c_ky c_k&=c_k &
		c_{k}x c_{k-1}y c_{k}&=c_{k}\\
		y c_kx~&= c_k+c_{k-1}+c_kxy c_{k-1}+c_{k-1}xy c_{k}
		&
		x c_ky~&= c_k+c_{k+1}+c_kxy c_{k+1}+c_{k+1}xy c_{k}\label{eq:YCX}
		\\
		c_{-1}&=0 &y c_{0} x~&=c_0
	\end{align}
	
\end{defn}

The relations of $\subalg$ are determined from the image of the isomorphisms in Proposition \ref{prop:AIsos} under the $K_0$ functor. We expect these to be a complete set of relations between the classes of these objects in $\aWeyl$, see Conjecture \ref{conj:K0iso}. 

We do not claim that $[Q_+]$, $[Q_-]$, and $[C_k]$ form a complete set of generators for $K_0(\aWeyl)$. For example, the Jones-Wenzl idempotents which act on $Q_+^s$, determine objects in $\aWeyl$ and therefore generators of $K_0(\aWeyl)$, see Section \ref{sec:ends}. However, we do not consider them here. 

\begin{prop}\label{prop:K0surj}
    Let ${K_0(\aWeyl)}'$ be the subalgebra of $K_0(\aWeyl)$ generated by $[Q_+]$, $[Q_-]$, and $[C_k]$. The homomorphism $\subalg\to {K_0(\aWeyl)}'$ determined by $x\mapsto[Q_+]$, $y\mapsto[Q_-]$, $c_k\mapsto[C_k]$ is well-defined. 
\end{prop}

\begin{conj}\label{conj:K0iso}
	The map defined in Proposition \ref{prop:K0surj} is an  isomorphism. 
\end{conj}

\begin{rem}
   We expect that there is a categorical version of ``asymptotic representation" which applies to $\aWeyl$ acting asymptotically on $\T$, the notion is defined here for algebras and vector spaces in Definition \ref{defn:asympt}. Such an action would imply $K_0(\aWeyl)$ acts on $K_0(\T)$. Mapping $\subalg$ to ${K_0(\aWeyl)}'$ and applying the argument of Lemma \ref{lem:SubalgBasis} would prove Conjecture \ref{conj:K0iso}.
\end{rem}

\subsection{Asymptotic Action}\label{ssec:asymp} Assign a degree to each of the generators of $\subalg$, $\deg(x)=\deg(y)=1$ and $\deg(c_k)=0$ for all $k\geq0$. For $a,b\in\subalg$, set $\deg(a\cdot b)=\deg(a)+\deg(b)$. We filter $\subalg$ by vector subspaces according to this degree, defining $\subalg_n=\{a\in \subalg\vert \deg(a)\leq n\}$. 

Recall that $\T$ is the direct sum over $n$ of $\TLn\mathrm{-mod}$. Then $\KT\coloneq K_0(\T)$ is the abelian group with generators the classes of simple modules $\{[W^n_m]:0\leq m\leq n, m\equiv n\operatorname{ mod }2 \}$. By Remark \ref{rem:IsoInAction}, the functors $\Ind$, $\Res$, and $P_k$ determine endomorphisms of $\KT$:
\begin{align}
	[\Ind]\cdot[W^n_m]&=[W^{n+1}_{m+1}]+[W^{n+1}_{m-1}]
	\\
	[\Res]\cdot[W^n_m]&=[W^{n-1}_{m+1}]+[W^{n-1}_{m-1}]
	\\
	[P_k]\cdot[W^n_m]&=\delta_{m,n-2k}[W^n_m]
\end{align}
with the convention that $[W^n_m]=0$ whenever $m<0$, $n<m$, or $m\not\equiv n\operatorname{ mod }2$.

Let $\KT_j$ be the subspace generated by $\{[W^n_m]\vert n\geq m\geq j\}$. This determines a filtration of $\KT$ with $\KT_{j+1}\subset \KT_{j}$ and $\KT_{-1}=\KT_0=\KT$. Now
\begin{align}\label{eq:GradingShift}
	[\Ind]\cdot \KT_j\subset \KT_{j-1}, && [\Res]\cdot \KT_j\subset \KT_{j-1},&& [P_k]\cdot \KT_j\subset \KT_{j}\,.
\end{align}

\begin{rem}\label{rem:nBound}
    For $[W^n_m]\in \KT_j$, $[P_k]\cdot [W^n_m]=\delta_{m,n-2k}[W^n_m]$. If $[P_k]\cdot [W^n_m]$ is nonzero, then $n\geq j+2k$. 
\end{rem}
 
\begin{defn}\label{defn:asympt}
    An \emph{asymptotic representation} of a filtered algebra $i_n:A_n \hookrightarrow A_{n+1}$ acting on a filtered vector space $j_{n+1}:V_{n+1} \hookrightarrow V_{n}$ is a collection of linear maps $f_{n,m}:A_n\otimes V_m\to V_{m-n}$ for $m\geq n\geq 0$ satisfying the following axioms:
    \begin{align}
        \label{eq:algcoherence} f_{n+1,m+1}\circ (i_n\otimes \Id_{V_{m+1}})&=j_{m-n+1}\circ f_{n,m+1}:A_n\otimes V_{m+1}\to V_{m-n}\,,
        \\
        \label{eq:vspcoherence} f_{n,m}\circ (\Id_{A_n}\otimes j_{m+1})&=j_{m-n+1}\circ f_{n,m+1}:A_n\otimes V_{m+1}\to V_{m-n}\,,
        \\
        \label{eq:compaxiom} f_{n+l,m}\circ (\mu_{l,n}\otimes \Id_{V_{m}})&=f_{l,m-n}\circ (\Id_{A_l}\otimes f_{n,m}):A_l\otimes A_n\otimes V_{m}\to V_{m-n-l}\,,
    \end{align}
    where $m\geq n+l$ and  $\mu_{l,n}:A_l\otimes A_n\to A_{n+l}$ is the multiplication map. These axioms are referred to as algebra coherence, vector space coherence, and composition respectively.
\end{defn}

\begin{prop}\label{prop:asymptoticrep}
    The assignment $x\mapsto[\Ind]$, $y\mapsto[\Res]$, $c_k\mapsto[P_k]$ defines an asymptotic representation of $\subalg$ on $\KT$. 
\end{prop}
\begin{proof}
    Formally, the collection of maps $f_{n,m}:\subalg_n\otimes \KT_m\to\KT_{m-n}$ are determined from the axioms by the assignments:
    \begin{align*}
        &f_{0,m}: \subalg_0\otimes \KT_m\to\KT_m & &{\begin{cases}
        c_k\otimes [W^n_m]\mapsto\delta_{m,n-2k}[W^n_{n-2k}]
        \end{cases}}
        \\
        &f_{1,m}: \subalg_1\otimes \KT_m\to\KT_{m-1} & &\begin{cases}
            x\otimes [W^n_m]\mapsto [W^{n+1}_{m+1}] + [W^{n+1}_{m-1}]
            \\
            y\otimes [W^n_m]\mapsto \overline{\delta_{n,m}}[W^{n-1}_{m+1}] + [W^{n-1}_{m-1}]
        \end{cases}
    \end{align*}
    The axioms \eqref{eq:algcoherence} and \eqref{eq:vspcoherence} of Definition \ref{defn:asympt} hold trivially. The composition axiom \eqref{eq:compaxiom} determines $f_{n,m}$ by expressing each element of $\subalg_n$ as a word written in the generators of $\subalg$. It remains to check that two words in $\subalg_n$ which are equivalent under the relations of $\subalg$ have the same action under $f_{n,m}$.
    
    The relations \eqref{eq:XC}-\eqref{eq:YCX} are determined by classes of functors in the ``generic'' isomorphisms of Table \ref{tab:isos}. As noted in Remark \ref{rem:asymp}, these isomorphisms hold in $\Hom_{\Wact}(n,-)$ provided that $n>2k$, where $k$ is an index of a functor $C_k$ appearing in a given isomorphism. In particular, we are concerned with avoiding instances where a ``generic'' isomorphism fails, namely those which are labeled ``none'' in the table. The corresponding words in such relations belong to $\subalg_2$. By Remark \ref{rem:nBound}, if $W^n_m\in\KT_2$ and $[P_k]\cdot [W^n_m]\neq 0$, then $n\geq 2k+2>2k$. Thus, the relations are indeed respected by the action. The remaining relations in $\subalg$ are considered generic and are easily verified to hold under the asymptotic action. 
\end{proof}

\subsection{Structure of $\subalg$} Here we discuss additional relations and bases for certain subalgebras of $\subalg$. For later use, observe that there is an anti-automorphism $\phi:\subalg\to\subalg$ which exchanges $x$ and $y$. 

\begin{lem}\label{lem:xBinProd}
	Let $0\leq l_1\leq l_2\leq l_3$ and $n,m\geq0$. Then 
	\begin{align*}
		c_{l_3}x^nc_{l_2}x^mc_{l_1}=
		\dfrac{{n\choose l_3-l_2}{m\choose l_2-l_1}}{{n+m\choose l_3-l_1}}c_{l_3}x^{n+m}c_{l_1}\,.
	\end{align*}
\end{lem}
\begin{proof}
	Let $k_0,k_1,\dots, k_{n+m}$ be any sequence such that $k_{i+1}\in\{k_i, k_i+1\}$, $k_0=l_1$, $k_m=l_2$, and $k_{n+m}=l_3$. It follows from both \eqref{eq:XC} and \eqref{eq:XCX}, that 
	\begin{gather*}
		c_{l_2}x^mc_{l_1}={m\choose l_2-l_1} c_{k_m}\prod_{i=0}^{m-1} xc_{k_i}\,, 
		\quad\quad
		c_{l_3}x^nc_{l_2}={n\choose l_3-l_2} c_{k_{n+m}}\prod_{i=m}^{m+n-1} xc_{k_i}\,,
		\\
		c_{l_3}x^{n+m}c_{l_1}={n+m\choose l_3-l_1} c_{k_{n+m}}\prod_{i=0}^{m+n-1} xc_{k_i}\,.
	\end{gather*}
	Then
	\begin{align*}
		c_{l_3}x^nc_{l_2}x^mc_{l_1}
		&=
		{n\choose l_3-l_2}{m\choose l_2-l_1}
		c_{k_{n+m}}\prod_{i=m}^{m+n-1} xc_{k_i}\cdot
        c_{k_m}\prod_{i=0}^{m-1} xc_{k_i}
        \\&
        =
        {n\choose l_3-l_2}{m\choose l_2-l_1}
        c_{k_{n+m}}\prod_{i=0}^{m+n-1} xc_{k_i}
		=
		\dfrac{{n\choose l_3-l_2}{m\choose l_2-l_1}}{{n+m\choose l_3-l_1}}c_{l_3}x^{n+m}c_{l_1}\,.
	\end{align*}
\end{proof}
 
\begin{rem}
    Lemma \ref{lem:xBinProd} readily implies an integral basis for the subalgebra generated by $\{c_{l_2}x^mc_{l_1}:l_1,l_2,m\geq 0\}$. This basis is spanned by ``divided powers elements'' $x^{(l_2, m, l_1)}\coloneq \dfrac{1}{{m \choose l_1-l_2}} c_{l_2}x^mc_{l_1}$
    with $l_1+m\geq l_2\geq l_1$, which multiply according to the rule $ x^{(l_4, n, l_3)} x^{(l_2, m, l_1)}=\delta_{l_3,l_2} x^{(l_4, n+m, l_1)}$.
\end{rem}

Introduce the notation 
\begin{align*}
	\{x\}^n_l=\sum_{i=1}^n x^ic_{l-i}x^{n-i} && \mbox{and} && [x]^n_l=\sum_{i=1}^n x^ic_{l-1}x^{n-i}\,.
\end{align*}
It is then easy to verify that
\begin{align}\label{eqn:CommxInduct}
	\{x\}^{n+1}_l=x^{n+1}c_{l-{n+1}}+\{x\}^{n}_l\cdot x && \mbox{and} && [x]^{n+1}_l=x^{n+1}c_{l-1}+[x]^n_l\cdot x
\end{align}

\begin{lem}\label{lem:ycxPow}
	For nonnegative integers $l$ and $n$, the following identity holds in $\subalg$:
	\[
	yc_lx^n
	=c_lx^nyc_{l-n}
	+c_{l-1}x^nyc_{l}
	+c_l\{x\}^{n-1}_l
	+c_{l-1}[x]^{n-1}_l\,.
	\]
\end{lem}
\begin{proof}
	We give a proof by induction on $n$. The claim is immediately verified in the case $n=0$. We then compute by the inductive hypothesis
	\begin{align*}
		yc_lx^{n+1}&=\left(c_lx^nyc_{l-n}
		+c_{l-1}x^nyc_{l}
		+c_l\{x\}^{n-1}_l
		+c_{l-1}[x]^{n-1}_l\right)x
		\\
		&=\left(c_lx^nc_{l-n}
		+c_lx^nc_{l-n}xyc_{l-n-1}\right)
		+
		\left(c_{l-1}x^nc_{l-1}
		+c_{l-1}x^nc_{l-1}xyc_{l}\right)
		+c_l\{x\}^{n-1}_l\cdot x
		+c_{l-1}[x]^{n-1}_l\cdot x
		\\
		&=c_lx^{n+1}yc_{l-n-1}
		+c_{l-1}x^{n+1}yc_{l}
		+c_l\{x\}^n_l
		+c_{l-1}[x]^n_l\,.
	\end{align*}
	In the second equality we use that $c_lx^nc_{l-n-1}=c_{l-1}x^nc_{l}=0$. The third equality is due to the relations $c_lx^nc_{l-n}xyc_{l-n-1}=c_lx^{n+1}yc_{l-n-1}$, $c_{l-1}x^nc_{l-1}xyc_{l}=c_{l-1}x^{n+1}yc_{l}$, and those in \eqref{eqn:CommxInduct}. Thus, proving the claim.
\end{proof}

\begin{lem}
	 For each $k\geq l\geq0$, $c_{k}\subalg c_l$ is spanned by 
 \[S_{k,l}=\{c_{k}x^{k-l+m}c_l, c_{k}x^{k-l}y^mc_l\mid m\in\bZ_{\geq 0} \}\,.\]
\end{lem}
\begin{proof}
Fix $k\geq0$. For each word $w\in \subalg$ let $\ell(w)$ be the length of $w$ (not necessarily a reduced word). We give an inductive proof on the length of $w$ and a downward induction on $l$. For a given word $w$, observe that if $k-l>\ell(w)$, then $c_{k+l}wc_k=0$ and if $k-l=\ell(w)$, then $c_{k+l}wc_k$ is nonzero exactly when $w=x^{k-l}$. The base case for a word of length $1$ is immediate.
	
	Assume for any word $w$ of length at most $n$ and any $k-n\geq l\geq 0$ that $c_{k}wc_l\in \operatorname{span}(S_{k,l})$. Fix a word $w$ of length $n+1$. The base case of the downward induction is $l=k-(n+1)$ as described above. Now fix that $c_{k}wc_l\in \operatorname{span}(S_{k,l})$. We now show that $c_{k}wc_{l-1}\in \operatorname{span}(S_{k,l-1})$. We proceed in cases, assuming that $w=w_1w'$ where $w_1\in\{c_{l-1}, x,y\}$ and $\ell(w')=n$.
	
	The case $w_1=c_{l-1}$ is the simplest as $c_{k}wc_{l-1}=c_{k}w'c_{l-1}$ and belongs to $\operatorname{span}(S_{k,l-1})$ by the inductive hypothesis. Note that if $w_1$ is equal to any other $c_s$, then $c_{k}wc_{l-1}=0$.
	
	Next, suppose $w_1=x$. Then
	\begin{align*}
	c_{k}wc_{l-1}=&~c_{k}xw'c_{l-1} =          c_{k}xc_{k}w'c_{l-1}
		+c_{k}xc_{k-1}w'c_{l-1}
		\\
		&\sum_m a_m c_{k}xc_{k}x^{k-l+1+m}c_{l-1}
		+
		\sum_m a'_m c_{k}xc_{k-1}x^{k-l}y^mc_{l-1}
		\\&+
		\sum_m b_m c_{k}xc_{k-1}x^{k-l+m}c_{l-1}
		+
		\sum_m b'_m c_{k}xc_{k}x^{k-l+1}y^mc_{l-1}
	\end{align*}
	for some coefficients, $a_m,a'_m,b_m,b'_m\in\bC(q)$. The resulting (finite) sums may be expressed in the desired basis by first writing $c_{k'}x^{k'-l+1}y^mc_{l-1}=c_{k'}x^{k'-l+1}c_{k'}y^mc_{l-1}$ for the sums involving $y^m$ terms, then applying Lemma \ref{lem:xBinProd} to all four sums. 
	
	Finally, consider $w_1=y$. We have 
	\begin{align*}
		c_{k}wc_{l-1}&~=c_{k}yw'c_{l-1}=c_{k}yc_{k}w'c_{l-1}
		+c_{k}yc_{k+1}w'c_{l-1}
		\\
		&\sum_m a_m c_{k}yc_{k}x^{k-l+1+m}c_{l-1}
		+
		\sum_m a'_m c_{k}yc_{k}x^{k-l+1}y^mc_{l-1}
		\\&+
		\sum_m b_m c_{k}yc_{k+1}x^{k-l+2+m}c_{l-1}
		+
		\sum_m b'_m c_{k}yc_{k+1}x^{k-l+2}y^mc_{l-1}
	\end{align*}
	for some coefficients, $a_m,a'_m,b_m,b'_m\in\bC(q)$. The claim follows from Lemmas \ref{lem:ycxPow} and \ref{lem:xBinProd}.
\end{proof}

\begin{lem}\label{lem:SubalgBasis}
    For each  $k\geq l\geq0$, $S_{k,l}$ is a basis for $c_{k}\subalg c_l$.
\end{lem}
\begin{proof}
	Let $\delta=k-l\geq0$. We prove that $S_{k,l}$ is a linearly independent set. Assume for a contradiction that there is a linear dependence
	\begin{align*}
		d=\sum_m a_m c_{k}x^{\delta+m}c_l + b_m c_{k}x^{\delta}y^mc_l=0\,.
	\end{align*}
	for some coefficients $a_m,b_m\in\bC(q)$. Let $M$ be the largest $m$ such that $a_m$ and $b_m$ are nonzero. Thus, $d\in \subalg_{M+\delta}$ and we may consider its action on $\KT_{M+\delta}$ since $\subalg$ acts asymptotically on $\KT$. Consider $[W^{M+k+l}_{M+\delta}]\in\KT_{M+\delta}$ so that
	\begin{align*}
		c_{k}x^{\delta+m}c_l\cdot[W^{M+k+l}_{M+\delta}]&=
        c_{k}x^{\delta+m}\cdot [W^{M+k+l}_{M+k-l}]={l+m\choose l}[W^{M+2k+m}_{M+m}]
		\\
		c_{k}x^{\delta}y^mc_l\cdot[W^{M+k+l}_{M+\delta}]
        &=c_{k}x^{\delta}c_l y^mc_l\cdot[W^{M+k+l}_{M+k-l}]
        =c_{k}x^{\delta}\cdot [W^{M-m+k+l}_{M-m+k-l}]=[W^{M+2k-m}_{M-m}]
	\end{align*}
	Since each term in $d$ maps $[W^{M+k+l}_{M+\delta}]$ to a unique nonzero vector in $\KT$, the linear dependence must be the trivial one.
\end{proof}

\begin{rem}
	By applying the anti-automorphism $\phi$ which switches $x$ and $y$, there is a basis of $c_k\subalg c_{l}$ for $k\geq l\geq0$ given by $S_{l,k}=\{c_{l}y^{(k-l)+m}c_{k}, c_{l}y^{k-l}x^mc_{k}\mid m\in\bZ_{\geq 0} \}$.
\end{rem}
 
For integers $k,l,m\geq 0$, define
\begin{align*}
    X^m_{k,l}\coloneq
    \begin{cases}
        c_kx^{m+(k-l)}c_l/{m+(k-l)\choose k-l} & k\geq l\\
        c_kx^{m}y^{l-k}c_l & l\geq k
    \end{cases}\,,
    &&
    X^{-m}_{l,k}\coloneq\phi(X^m_{k,l})\,.
\end{align*}
The choice of asymmetric normalization is intentional.

\begin{lem}
	For $a,b\geq 0$ and $n,m\in\bZ$, the multiplication map $c_{k}\subalg c_l \otimes c_l \subalg c_{k}\to c_{k}\subalg c_{k}$ satisfies \[X^n_{k,l}\otimes X^m_{l,k}\mapsto X^{n+m}_{k,k}.\] 
\end{lem}
\begin{proof}
    The proof is a direct computation for choices $k\geq l$ or $l\geq k$, $n\geq 0$ or $n\leq 0$, and $m\geq 0$ or $m\leq 0$. We include one such computation. Assume $k\leq l$ and $m,n\geq 0$. Then
    \begin{align*}
        X^n_{k,l}X^m_{l,k}
        =
        c_k x^n y^{l-k} c_l x^{m+l-k}c_l/{m+l-k\choose l-k}
        =
        c_k x^n c_ky^{l-k} c_l x^{l-k}c_kx^mc_k
        =
    c_k x^n c_kx^mc_k
        = X^{n+m}_{k,k}\,.
    \end{align*}
\end{proof}

\begin{cor}
	For each $k\geq0$, $c_k\subalg c_k$ is isomorphic to a Laurent polynomial algebra in the variable $X_{k,k}$. 
\end{cor}

\appendix
\section{Results from Quinn}\label{sec:app}
We recall key statements and main results from \cite{Quinn} which we use in this paper. Namely, the structure of the representation categories of Temperley-Lieb algebras, the action of induction and restriction functors, and relations in the precursor to the universal Weyl category. The reader is referred there and references therein for additional details. Unless stated otherwise, results given here are proven in \cite{Quinn}. Some notation is changed in this paper. 

\subsection{Representations of Temperley-Lieb Algebras}\label{ssec:RepTLn}
The set of isomorphism classes of irreducible $\TLn$ representations are in bijection with nonnegative integers $m\leq n$ such that $m\equiv n\operatorname{mod}2$. We denote such an irreducible representation $W^n_m$. Vectors in these representations may be presented as $(m,n)$-Temperley-Lieb diagrams $w$ such that $w=wf^{(m)}$. This implies that $W^n_n$ is the trivial representation. We assume that $W^n_m=0$ is $m$ is negative, exceeds $n$, or has different parity from $n$. 

A basis for these representations is given using a path algebra approach. Equivalently, there is a Young tableaux description of basis vectors given in \cite{Moore}, for example. Let $P^n$ denote the set of paths of length $n$ starting from the vertex labeled $1$ on the type $\mathsf{A}_{n+1}$ graph whose vertices are labeled $0,1,2,\dots, n$\,. A path of length $n$ is notated by a sequence $p=(p_1,p_2,\dots, p_n)$ of nonnegative integers such that $p_1=1$ and $p_{i+1}=p_i\pm1$ for $i<n$. A basis for $W^n_m$ is in bijection with $P^n_m$ the set of paths in $P^n$ which end at vertex $m$, i.e, $p_n=m$. For a given path, the corresponding basis element is constructed inductively from a single strand. Suppose that $v_k$ is the $(p_k,k)$-Temperley-Lieb diagram constructed from a path of length $k$ ending at $p_k$. If the path continues to the right (increases), then
\begin{align}
	\tikzc{
		\vseg{1}{-.5}\vseg{0}{-.5}\vseg{-1}{-.5}
		\vseg{1}{-1.5}\vseg{0}{-1.5}\vseg{-1}{-1.5}
		\drawsquaretext[$v_{k+1}$]{0}{0}
	}
	~=~
	\tikzc{
		\vseg{1}{-.5}\vseg{0}{-.5}\vseg{-1}{-.5}
		\vseg{1}{-1.5}\vseg{0}{-1.5}\vseg{-1}{-1.5}
		\vseg{1}{-4}\vseg{0}{-4}\vseg{-1}{-4}
		\drawsquaretext[\phantom{}$v_k$\phantom{}]{.5}{0}
		\drawsquaretext[$f^{(p_k+1)}$]{0}{-2.5}
	}
\end{align}
and if the path continues to the left (decreases), then
\begin{align}
	\tikzc{
		\vseg{1}{-.5}\vseg{0}{-.5}\vseg{-1}{-.5}
		\vseg{1}{-1.5}
		\drawsquaretext[$v_{k+1}$]{0}{0}
	}
	~=\sqrt{\frac{[p_k]}{[p_k+1]}}~
	\tikzc{
		\vseg{1}{-.5}\vseg{0}{-.5}\vseg{-2}{-.5}
		\vseg{1}{-1.5}\vseg{0}{-1.5}\vseg{-2}{-1.5}
		\vseg{1}{-2.5}
		\drawsquaretext[$v_k$]{0.5}{0}
		\ncup{-2}{-1.5}
	}
\end{align}
The normalization factor $\sqrt{\frac{[p_k]}{[p_k+1]}}$ is required for orthonormality in the following sense. If $p\in P^n_m$ corresponds to $v_p\in W^n_m$, then the reflection of $v_p$ over a horizontal axis is the corresponding dual basis vector $\check{v}_p$ in the sense that $\check{v}_p(v_r)=\delta_{p,r}$, where $\delta=1$ if $p=r$ and is zero otherwise. Write $e_{p,r}$ for the matrix element $v_p\check{v}_r$ provided $p,r\in P^n_m$. 

\begin{rem}\label{rem:idemp}
	We observe that $\sum_{p\in P^n_m} e_{p,p}\in \TLn$ is an idempotent corresponding to the projection onto $W^n_m$. Moreover,
	\begin{align*}
		\label{eq:InductProj}
		\sum_{p\in P^n_m} e_{p,p}
		=
		\sum_{r\in P^{n-1}_{m+1}} \frac{[m+1]}{[m+2]}v_re_{n-1}'\check{v}_r
		+
		\sum_{r\in P^{n-1}_{m-1}} v_rf^{(m)}\check{v}_r.&\qedhere
	\end{align*}
\end{rem}

\subsection{Induction and Restriction for  Temperley-Lieb Modules} \label{ssec:IndRes}
For any $\TLn$-module $V$, the induction functor may be defined as the functor
\begin{align}
	\Ind:\TLn\mathrm{-mod}\to\TL_{n+1}\mathrm{-mod}
	&&
	V\mapsto\mod{n+1}{n+1}{n}\otimes V
\end{align}
or informally as adding an additional string to the diagram. Restriction is instead given by the functor
\begin{align}
	\Res:\TLn\mathrm{-mod}\to\TL_{n-1}\mathrm{-mod}
	&&
	V\mapsto\mod{n-1}{n}{n}\otimes V
\end{align} which in some sense forgets the $n$-th string in the diagram. 

Quinn provides a formula for the restriction of (nonzero) simple modules $W^n_m$ in terms of simple modules $\Res(W^n_m)\cong W^{n-1}_{m+1}\oplus W^{n-1}_{m-1}$. The isomorphism is
witnessed by the maps:
\begin{align}
	\Res(W^n_m)\to W^{n-1}_{m+1}\oplus W^{n-1}_{m-1}: &&
	{
		\tikzc{
			\vseg{-1}{-.5}
			\vseg{0}{-.5}
			\vseg{-1}{-1.5}\vseg{0}{-1.5}
			\drawsquaretext[$v$]{-.5}{0}
	}}
	~
	&\mapsto{\left(~
		\tikzc{
			\vseg{1.5}{.5}
			\ncap{-1.5}{1.5}
			\vseg{1.5}{-.5}\vseg{.5}{-.5}\vseg{-1.5}{-.5}
			\vseg{1.5}{-1.5}\vseg{.5}{-1.5}\vseg{-1.5}{-1.5}
			\vseg{1.5}{-4}\vseg{.5}{-4}\vseg{-1.5}{-4}
			\drawsquaretext[$v$]{1}{0}
			\drawsquaretext[$f^{(m+1)}$]{0}{-2.5}
		}~,~
		{\frac{[m]}{[m+1]}}
		\tikzc{
			\vseg{1}{-.5}
			\vseg{0}{-.5}
			\vseg{1}{-1.5}\vseg{0}{-1.5}
			\ncup{-2}{-1.5}
			\ncap{-2}{1.5}
			\vseg{-2}{-1.5}\vseg{-2}{-.5}
			\vseg{1}{.5}\vseg{1}{-2.5}
			\drawsquaretext[$v$]{.5}{0}
		}
		\right)
	}
	\\
	W^{n-1}_{m+1}\oplus W^{n-1}_{m-1}\to\Res(W^n_m):
	&&
	{
		\left(~
		\tikzc{
			\vseg{-1}{-.5}
			\vseg{0}{-.5}
			\vseg{-1}{-1.5}\vseg{0}{-1.5}
			\drawsquaretext[$v$]{-.5}{0}
		}~,~
		\tikzc{
			\vseg{-1}{-.5}
			\vseg{0}{-.5}
			\vseg{-1}{-1.5}\vseg{0}{-1.5}
			\drawsquaretext[$w$]{-.5}{0}
		}~
		\right)
	}
	&
	\mapsto~
	{\tikzc{
			\vseg{1}{-.5}
			\vseg{0}{-.5}
			\vseg{1}{-1.5}\vseg{0}{-1.5}
			\drawsquaretext[$v$]{.5}{0}
			\ncup{-2}{-1.5}\vseg{-2}{-1.5}\vseg{-2}{-.5}
			\vseg{1}{-3}
		}
		~+~
		\tikzc{
			\vseg{-1}{2}
			\vseg{0}{2}\vseg{1}{2}
			\vseg{-1}{-0}\vseg{0}{-0}\vseg{1}{-0}
			\vseg{-1}{4}\vseg{0}{4}\vseg{1}{4}
			\drawsquaretext[$f^{(m)}$]{0}{2}
			\drawsquaretext[$w$]{0.5}{4.5}
		}
	}
\end{align}

Quinn also proves that $\Ind(W^n_m)\cong W^{n+1}_{m+1}\oplus W^{n+1}_{m-1}$. The maps for this isomorphism are not explicitly given, but can be inferred from \cite[Props. 3.3.1, 3.4.1]{Quinn}. Here $x\in\TL_{n+1}$ and $v,w$ are vectors in the indicated simple module:
\begin{align}\label{eq:IndtoW}
	\Ind(W^n_m)\to W^{n+1}_{m+1}\oplus W^{n+1}_{m-1}: &&
	{
		\tikzc{
			\vseg{-1}{1.5}
			\vseg{0}{1.5}\vseg{1}{1.5}
			\vseg{-1}{.5}
			\vseg{0}{.5}\vseg{1}{.5}
			\vseg{-1}{-1.5}\vseg{0}{-1.5}\vseg{1}{-1.5}
			\drawsquaretext[$v$]{.5}{0}
			\drawsquaretext[\phantom{$x$}$x$\phantom{$x$}]{0}{2}
	}}
	~
	&\mapsto{\left(~
		\tikzc{
			\vseg{-1}{1.5}
			\vseg{0}{1.5}\vseg{1}{1.5}
			\vseg{-1}{.5}
			\vseg{0}{.5}\vseg{1}{.5}
			\vseg{-1}{-1.5}\vseg{0}{-1.5}\vseg{1}{-1.5}
			\vseg{-1}{-4}\vseg{0}{-4}\vseg{1}{-4}
			\drawsquaretext[$v$]{.5}{0}
			\drawsquaretext[\phantom{$x$}$x$\phantom{$x$}]{0}{2}
			\drawsquaretext[$f^{(m+1)}$]{0}{-2.5}
		}~,~
		\sqrt{\frac{[m]}{[m+1]}}~
		\tikzc{
			\vseg{1}{1.5}
			\vseg{0}{1.5}\vseg{-2}{1.5}
			\vseg{1}{.5}
			\vseg{0}{.5}
			\vseg{1}{-1.5}\vseg{0}{-1.5}\vseg{-2}{0}
			\vseg{1}{-2.5}
			\drawsquaretext[$v$]{.5}{0}
			\drawsquaretext[\phantom{$xx$}$x$\phantom{$xx$}]{-.5}{2}
			\ncup{-2}{-1.5}\vseg{-2}{-1.5}
		}
		\right)
	}
	\\
	W^{n+1}_{m+1}\oplus W^{n+1}_{m-1}\to\Ind(W^n_m):
	&&
	{
		\left(~
		\tikzc{
			\vseg{-1}{-.5}
			\vseg{0}{-.5}
			\vseg{-1}{-1.5}\vseg{0}{-1.5}
			\drawsquaretext[$v$]{-.5}{0}
		}~,~
		\tikzc{
			\vseg{-1}{-.5}
			\vseg{0}{-.5}
			\vseg{-1}{-1.5}\vseg{0}{-1.5}
			\drawsquaretext[$w$]{-.5}{0}
		}~
		\right)
	}
	&
	\mapsto
	{\tikzc{
			\vseg{-1}{-.5}
			\vseg{0}{-.5}
			\vseg{-1}{-1.5}\vseg{0}{-1.5}
			\drawsquaretext[$v$]{-.5}{0}
		}
		~+~
		\sqrt{\frac{[m]}{[m+1]}}~
		\tikzc{
			\vseg{1}{4}
			\vseg{0}{4}
			\vseg{1}{2}
			\vseg{0}{2}\vseg{-1}{2}\vseg{-3}{2}
			\vseg{-3}{0}	 	
			\vseg{1}{-0}\vseg{0}{-0}\vseg{-1}{-0}
			\drawsquaretext[$f^{(m)}$]{0}{2}
			\drawsquaretext[$w$]{0.5}{4.5}
			\ncap{-3}{4}
		}
	}
\end{align}
That the image of the last homomorphism belongs to $\Ind(W^n_m)$ is a consequence of \cite[Prop. 3.4.1]{Quinn}.

\begin{rem}
	The vectors appearing in the image of \eqref{eq:IndtoW}, together with their dual vectors, agree with the decomposition of $\sum_{p\in P^n_m} e_{p,p}$ given in Remark \ref{rem:idemp}. Our previous observation is therefore an expression of the projection onto $W^n_m$ in terms of projections through components of $\Ind(W^{n-1}_{m+1})$ and $\Ind(W^{n-1}_{m-1})$.
\end{rem}

\subsection{Diagrammatic Calculus}\label{ssec:Qdiagrams} In Chapter 5, Quinn initiates the definition of the Weyl category. An important difference in our notation is that Quinn's crossing is replaced by $[2]$ times the cupcap. It acts by multiplication by $e'$, the pre-idempotent Temperley-Lieb generator, on the two induced strands. Quinn defines an abstract category which acts on $\TLn$-modules in the following way
\begin{align}
	{\tikzc{\rcap{0}{0}}}~&\tleq[n]
	\left[x\otimes y\mapsto xy\right]:
	\mod{n}{n}{n-1}\mod{}{n}{n}
	\to \mod{n}{n}{n},
	\\
	{\tikzc{\rcup{0}{1}}}~&\tleq[n]
	\left[x\mapsto x\right]:
	\mod{n}{n}{n}\to \mod{n}{n+1}{n+1}\mod{}{n+1}{n},
	\\
	{\tikzc{\lcap{0}{0}}}~&\tleq[n]
	\left[x\mapsto \ptr_{n+1}(x) \right]:
	\mod{n}{n+1}{n+1}\mod{}{n+1}{n}\to \mod{n}{n}{n}, 
	\\
	{\tikzc{\lcup{0}{1}}}~&\tleq[n]
	\left[x\mapsto \sum_{\substack{p,r\in P^n\\p_n=r_n}}
	c_r e_{p,r} \otimes e_{r,p}x\right]:
	\mod{n}{n}{n}\to \mod{n}{n}{n-1}\mod{}{n}{n},
	\\
	\tikzc{
		\uucup{0}{4}
		\nncap{0}{0}
	}
	~&
	\tleq[n]
	\left[x\mapsto xe_{n+1}\right]:
	\mod{n+2}{n+2}{n}\to\mod{n+2}{n+2}{n}
\end{align}
where the map $x\mapsto \ptr_{n+1}(x)$ is the right partial trace of $x$ as defined in \eqref{def:ptr}, $e_{p,r}$ and $e_{r,p}$ are the matrix units defined here in Section \ref{ssec:RepTLn}, and 
$c_r=\frac{[r_{n-1}+1]}{[r_n+1]}\frac{1}{|P^n_{r_{n-1}}|}$.

Quinn showed that these generators satisfy the following relations for any $n\geq0$ in the action on $\TLn\mathrm{-mod}$. The proof of \eqref{eq:Qzigzag} is given in \cite[Theorem 4.2.14]{Quinn}. 
\begin{align} \label{eq:Qzigzag}
	{\tikzc{
			\ncap{0}{2}\up[top>]{4}{2}
			\up{0}{0}\ncup{2}{2}
		}
		~\tleq~
		\tikzc{
			\vseg{0}{0}\up[top>]{0}{2}
		}
		~\tleq~
		\tikzc{
			\ncap{2}{2}\up[top>]{0}{2}
			\up{4}{0}\ncup{0}{2}
	}}
	&&
	{\tikzc{
			\ncap{0}{2}\down{4}{2}
			\down{0}{0}\ncup{2}{2}
		}
		~\tleq~
		\tikzc{
			\vseg{0}{0}\down{0}{-2}
		}
		~\tleq~
		\tikzc{
			\ncap{2}{2}\down{0}{2}
			\down{4}{0}\ncup{0}{2}
	}}
\end{align}
\vspace{-\baselineskip}
\begin{align}\label{eq:Qloops}
	{[2]~\tikzc{ 
			\ncap{-2}{6}\up[top>]{2}{6}
			\down[bottom<]{-2}{4}
			\vseg{-2}{2}\cupcap{0}{2}
			\ncup{-2}{2}\up{2}{0}
		}
		~\tleq~
		\tikzc{
			\vseg{0}{-2}\vseg{0}{0}\vseg[top>]{0}{2}
	}}
	&&
	{[2]^2~\tikzc{
			\vseg[top>]{0}{12}\ncap{2}{12}
			\vseg[bottom<]{4}{10}
			\vseg{4}{8}
			\cupcap{0}{8}\ncup{2}{8}
			\vseg{0}{6}
			\vseg[top>]{0}{4}\ncap{2}{4}
			\vseg[bottom<]{4}{2}
			\vseg{4}{0}
			\cupcap{0}{0}\ncup{2}{0}
			\vseg[top>]{0}{-2}
		}
		~\tleq~
		[2]~ \tikzc{
			\vseg[top>]{0}{12}\ncap{2}{12}
			\vseg[bottom<]{4}{10}
			\vseg{4}{8}
			\cupcap{0}{8}\ncup{2}{8}
			\vseg[top>]{0}{6}
	}}
	&&
	{	[2]^2~\tikzc{ 
			\uucup{0}{4}\nncap{0}{0}\uucup{0}{8}\nncap{0}{4}
		}
		~\tleq~
		[2]~\tikzc{ 
			\uucup{0}{4}\nncap{0}{0}
	} }
	\\
	{ 	[2]^2~\tikzc{
			\lefte{0}{2}
			\righte{0}{0}
		}
		~\tleq~
		\tikzc{
			\up{0}{0}\vseg{0}{-2}
			\down{2}{-2}\vseg{2}{0}
	}}
	&&
	{
		[2]^2~\tikzc{
			\lefte{0}{0}
			\righte{0}{2}
		}
		~\tleq~ [2]~ \tikzc{
			\down{-2}{1}\vseg{-2}{5}\vseg{-2}{3}
			\vseg[top>]{0}{6}\ncap{2}{6}
			\vseg[bottom<]{4}{4}
			\vseg{4}{2}
			\cupcap{0}{2}\ncup{2}{2}
			\vseg[top>]{0}{0}
		}
	}
	&&
	{\tikzc{
			\ncap{0}{0}\rcup{0}{0}
		} 
		~\tleq~
		[2]}
\end{align}

 \newcommand{\etalchar}[1]{$^{#1}$}

 \end{document}

%% file: preamble.tex
\usepackage{amsmath,amssymb,amsfonts,amsthm,mathrsfs,mathtools,bbm,cancel,calc,arydshln, tikz-cd}

\usepackage{mathrsfs}

\usepackage{comment}

\usepackage[all]{xy}
\SelectTips{cm}{}

\usepackage[section]{placeins}
\usepackage{leftidx}
\usepackage{stmaryrd} 
\usepackage{microtype}

\theoremstyle{plain}
\newtheorem{prop}{Proposition}[section]
\makeatletter
\@addtoreset{prop}{section}
\makeatother
\newtheorem{conj}[prop]{Conjecture}
\newtheorem{thm}[prop]{Theorem}
\newtheorem{lem}[prop]{Lemma}
\newtheorem*{lem*}{Lemma}
\newtheorem{cor}[prop]{Corollary}
\newtheorem*{cor*}{Corollary}           
\newtheorem*{conv*}{Convention}

\theoremstyle{remark}
\newtheorem{remark}[prop]{Remark}

\newtheorem{defn}[prop]{Definition}

\newenvironment{rem}
  {\pushQED{\qed}\remark}
  {\popQED\endremark}

\theoremstyle{plain}
\newtheorem{Thm}{Theorem}

\newcommand{\floor}[1]{\left\lfloor#1\right\rfloor}
\newcommand{\Hom}{\operatorname{Hom}}

\usepackage{hyperref}
\usepackage{graphicx}

\def\bZ{{\mathbb{Z}}}

\def\bC{{\mathbb{C}}}

\usepackage{color}

\usepackage{xcolor}
\definecolor{dark-red}{rgb}{0.7,0.25,0.25}
\definecolor{dark-blue}{rgb}{0.15,0.15,0.55}
\definecolor{medium-blue}{rgb}{0,0,0.65}
\hypersetup{
    colorlinks, linkcolor={dark-red},
    citecolor={dark-blue}, urlcolor={medium-blue}
}



%% file: tikzmacros.tex
\usepackage{tikz,tikz-cd}
\usetikzlibrary{shapes,knots,hobby}
\usetikzlibrary{backgrounds}
\usetikzlibrary{decorations,decorations.pathreplacing,decorations.markings}
\usetikzlibrary{fit,calc,through}
\usetikzlibrary{external}

\newcommand{\up}[3][top>]{
\draw[very thick, #1] 
(#2em,#3em)--+(0,2em);
}

\newcommand{\down}[3][bottom<]{
\draw[very thick, #1] 
(#2em,#3em)--+(0,2em);
}

\newcommand{\vseg}[3][]{
\draw[very thick, #1] 
(#2em,#3em)--+(0,2em);
}

\newcommand{\Up}[3][upper>]{
\draw[very thick, doubled, #1] 
(#2em,#3em)--+(0,2em);
}

\newcommand{\Down}[3][mid<]{
\draw[very thick, doubled, #1] 
(#2em,#3em)--+(0,2em);
}

\newcommand{\Vseg}[3][]{
\draw[very thick, doubled,#1] 
(#2em,#3em)--+(0,2em);
}

\newcommand{\lcup}[2]{
\draw[very thick, lower<] 
(#1em,#2em) arc (180:270:1em) node[coordinate] (c) {}; 
\draw[very thick] (c) arc (270:360:1em);
}

\newcommand{\rcup}[2]{
\draw[very thick, lower<] 
(#1em,#2em) ++ (2em,0) arc (0:-90:1em) 
node[coordinate] (c) {}; 
\draw[very thick] 
(c) arc (270:180:1em);
}

\newcommand{\ncup}[2]{
\draw[very thick] 
(#1em,#2em) arc (180:270:1em) node[coordinate] (c) {}; 
\draw[very thick] (c) arc (270:360:1em);
}

\newcommand{\Rcup}[2]{
\draw[very thick, doubled, lower<] 
(#1em,#2em) ++ (2em,0) arc (0:-90:1em) 
node[coordinate] (c) {}; 
\draw[very thick, doubled] 
(c) arc (270:180:1em);
}

\newcommand{\Ncup}[2]{
\draw[very thick, doubled] 
(#1em,#2em) arc (180:270:1em) node[coordinate] (c) {}; 
\draw[very thick, doubled] (c) arc (270:360:1em);
}

\newcommand{\widecup}[3]{
\draw[very thick] 
(#1em,#2em) arc (180:360:#3em) node[coordinate] (c) {}; 
}

\newcommand{\widercup}[3]{
\draw[very thick,decoration={markings, mark=at position .99 with {\arrow{stealth}}}, postaction={decorate}] 
(#1em,#2em) arc (180:360:#3em) node[coordinate] (c) {}; 
}

\newcommand{\widelcup}[3]{
	\draw[very thick,decoration={markings, mark=at position .99 with {\arrow{stealth}}}, postaction={decorate}] 
	(#1em,#2em) arc (360:180:#3em) node[coordinate] (c) {}; 
}

\newcommand{\Widecup}[3]{
\draw[very thick, doubled] 
(#1em,#2em) arc (180:270:#3em) node[coordinate] (c) {}; 
\draw[very thick, doubled] (c) arc (270:360:#3em);
}

\newcommand{\lcap}[2]{
\draw[very thick, lower<] 
(#1em,#2em) arc (180:90:1em)
node[coordinate] (c) {};
\draw[very thick]
(c) arc (90:0:1em); 
}

\newcommand{\rcap}[2]{
\draw[very thick, lower<]
(#1em,#2em)++(2em,0em) arc (0:90:1em) 
node[coordinate] (c) {}; 
\draw[very thick] (c) arc (90:180:1em);
}

\newcommand{\ncap}[2]{
\draw[very thick]
(#1em,#2em)++(2em,0em) arc (0:90:1em) 
node[coordinate] (c) {}; 
\draw[very thick] (c) arc (90:180:1em);
}

\newcommand{\Lcap}[2]{
\draw[very thick, doubled, lower<] 
(#1em,#2em) arc (180:90:1em)
node[coordinate] (c) {};
\draw[very thick, doubled]
(c) arc (90:0:1em); 
}

\newcommand{\Rcap}[2]{
\draw[very thick, doubled, lower<]
(#1em,#2em)++(2em,0em) arc (0:90:1em) 
node[coordinate] (c) {}; 
\draw[very thick, doubled] (c) arc (90:180:1em);
}

\newcommand{\Ncap}[2]{
\draw[very thick, doubled]
(#1em,#2em)++(2em,0em) arc (0:90:1em) 
node[coordinate] (c) {}; 
\draw[very thick, doubled] (c) arc (90:180:1em);
}

\newcommand{\widecap}[3]{
\draw[very thick]
(#1em,#2em)++(#3em,0em)++(#3em,0em) arc (0:90:#3em) 
node[coordinate] (c) {}; 
\draw[very thick] (c) arc (90:180:#3em);
}

\newcommand{\widercap}[3]{
\draw[very thick,decoration={markings, mark=at position .99 with {\arrow{stealth}}}, postaction={decorate}] 
(#1em,#2em) arc (180:0:#3em) node[coordinate] (c) {};
}

\newcommand{\widelcap}[3]{
\draw[very thick,decoration={markings, mark=at position .99 with {\arrow{stealth}}}, postaction={decorate}] 
(#1em,#2em) arc (0:180:#3em) node[coordinate] (c) {};
}

\newcommand{\Widecap}[3]{
\draw[very thick, doubled]
(#1em,#2em)++(#3em,0em)++(#3em,0em) arc (0:90:#3em) 
node[coordinate] (c) {}; 
\draw[very thick, doubled] (c) arc (90:180:#3em);
}

\newcommand{\Widercap}[3]{
\draw[very thick, doubled, bottom<]
(#1em,#2em)++(#3em,0em)++(#3em,0em) arc (0:90:#3em) 
node[coordinate] (c) {}; 
\draw[very thick, doubled] (c) arc (90:180:#3em);
}

\newcommand{\Widelcap}[3]{
	\draw[very thick, doubled, bottom<]
	(#1em,#2em) arc (180:90:#3em) 
	node[coordinate] (c) {}; 
	\draw[very thick, doubled] (c) arc (90:0:#3em);
}

\newcommand{\Widercup}[3]{
	\draw[very thick, doubled, bottom<]
	(#1em,#2em)++(#3em,0em)++(#3em,0em) arc (0:-90:#3em) 
	node[coordinate] (c) {}; 
	\draw[very thick, doubled] (c) arc (-90:-180:#3em);
}
\newcommand{\uucap}[2]{
\draw[very thick, lower>] 
(#1em,#2em) arc (180:90:1em) 
node[coordinate] (c) {}; 

\draw[very thick, lower>] 
(#1em,#2em)++(2em,0em) arc (0:90:1em);

\draw[very thick, doubled] 
(c) -- +(0,1em);
}

\newcommand{\uucup}[2]{
\draw[very thick, bottom<] 
(#1em,#2em) arc (-180:-90:1em)
node[coordinate] (c) {}; 
\draw[very thick, doubled] (c) -- +(0,-1em);
\draw[very thick, bottom<] 
(#1em,#2em)++(2em,0em) arc (0:-90:1em);
} 

\newcommand{\ddcap}[2]{
\draw[very thick, lower<] 
(#1em,#2em) arc (180:90:1em) 
node[coordinate] (c) {}; 

\draw[very thick, lower<] 
(#1em,#2em)++(2em,0) arc (0:90:1em);

\draw[very thick, doubled] 
(c) -- +(0,1em);
}

\newcommand{\ddcup}[2]{
\draw[very thick, lower>] 
(#1em,#2em) arc (-180:-90:1em)
node[coordinate] (c) {}; 

\draw[very thick, lower>] 
(#1em,#2em)++(2em,0) arc (0:-90:1em);

\draw[very thick, doubled] (c) -- +(0,-1em);
}

\newcommand{\nncap}[2]{
\draw[very thick] 
(#1em,#2em) arc (180:90:1em) 
node[coordinate] (c) {}; 

\draw[very thick] 
(#1em,#2em)++(2em,0) arc (0:90:1em);

\draw[very thick, doubled] 
(c) -- +(0,1em);
}

\newcommand{\cupcap}[2]{
\draw[very thick] 
(#1em,#2em) arc (180:90:1em) 
node[coordinate] (c) {}; 

\draw[very thick] 
(#1em,#2em)++(2em,0em) arc (0:90:1em);

\draw[very thick, doubled] 
(c) -- +(0,2em);

\draw[very thick] 
(#1em,#2em) + (0,4em) arc (-180:0:1em);
} 
\newcommand{\widecupcap}[3]{
\draw[very thick] 
(#1em,#2em) arc (180:90:#3em) 
node[coordinate] (c) {}; 

\draw[very thick] 
(#1+#3em+#3em,#2em) arc (0:90:#3em);

\draw[very thick, doubled] 
(c) -- +(0,2em);

\draw[very thick] 
(#1em,#2em) + (0,#3em+#3em+2em) arc (-180:0:#3em);
} 
\newcommand{\lefte}[2]{
\draw[very thick] 
(#1em,#2em) -- ++(0,2em);
\draw[very thick] 
(#1em,#2em) ++(2em,0em) -- ++(0,2em);

\draw[very thick,doubled, mid<] 
(#1em,#2em) ++(0em,1em) -- ++(2em,0em);
} 

\newcommand{\righte}[2]{
\draw[very thick] 
(#1em,#2em) -- ++(0,2em);
\draw[very thick] 
(#1em,#2em) ++(2em,0em) -- ++(0,2em);
\draw[very thick,doubled, mid<] 
(#1em,#2em) ++(2em,1em) -- ++(-2em,0em);
} 

\newcommand{\drawsquare}[3][\phantom{0}]{
\draw  (#2em, #3em)
node[draw,rectangle,inner xsep=1.2ex,inner ysep=1.2ex,fill=white,minimum size=2mm, very thick]  {${#1}$};
}

\newcommand{\drawsquaretext}[3][\phantom{0}]{
	\draw  (#2em, #3em)
	node[draw,rectangle,inner xsep=1.2ex,inner ysep=1.2ex,fill=white,minimum size=2mm, very thick]  {{#1}};
}

\newcommand{\drawP}[3]
{
\draw  (#1em, #2em)
node[draw,circle,inner sep=0.5ex,fill=white,minimum size=.45mm, very thick] {$#3$};
}

 \newcommand{\tikzc}[1]{
 \centering{\xy (0,0)*{\tikz{#1}}; \endxy}}
 
\tikzstyle{doubled}=[thin,double]
\tikzstyle{mid>}=[decoration={markings, mark=at position 0.5 with {\arrow{stealth}}}, postaction={decorate}]
\tikzstyle{mid<}=[decoration={markings, mark=at position 0.5 with {\arrow{stealth reversed}}}, postaction={decorate}]

\tikzstyle{top>}=[decoration={markings, mark=at position 1. with {\arrow{stealth}}}, postaction={decorate}]

\tikzstyle{top<}=[decoration={markings, mark=at position 1. with {\arrow{stealth reversed}}}, postaction={decorate}]

\tikzstyle{bottom<}=[decoration={markings, mark=at position 0.1 with {\arrow{stealth reversed}}}, postaction={decorate}]

\tikzstyle{upper>}=[decoration={markings, mark=at position 0.8 with {\arrow{stealth}}}, postaction={decorate}]
\tikzstyle{upper<}=[decoration={markings, mark=at position 0.8 with {\arrow{stealth reversed}}}, postaction={decorate}]
\tikzstyle{lower>}=[decoration={markings, mark=at position 0.2 with {\arrow{stealth}}}, postaction={decorate}]
\tikzstyle{lower<}=[decoration={markings, mark=at position 0.2 with {\arrow{stealth reversed}}}, postaction={decorate}]

\tikzstyle{dot}=[decoration={markings, mark=at position 0.5 with {\node[fill,circle,inner sep=0pt,minimum size=2mm] {};}}, postaction={decorate}]

\tikzstyle{box}=[decoration={markings, mark=at position 0.5 with {\node[draw,rectangle,inner xsep=1.2ex,inner ysep=.6ex,fill=white,minimum size=2mm, very thick] {};}}, postaction={decorate}]

\tikzstyle{lowerdot}=[decoration={markings, mark=at position 0.25 with {\node[fill,circle,inner sep=0pt,minimum size=2mm] {};}}, postaction={decorate}]

\tikzstyle{upperdot}=[decoration={markings, mark=at position 0.75 with {\node[fill,circle,inner sep=0pt,minimum size=2mm] {};}}, postaction={decorate}]

\tikzstyle{odot}=[decoration={markings, mark=at position 0.5 with {\node[draw,circle,inner sep=0pt,minimum size=2mm,fill=white,thick] {};}}, postaction={decorate}]

\tikzstyle{lowerodot}=[decoration={markings, mark=at position 0.25 with {\node[draw,circle,inner sep=0pt,minimum size=2mm,fill=white,thick] {};}}, postaction={decorate}]

\tikzstyle{upperodot}=[decoration={markings, mark=at position 0.75 with {\node[draw,circle,inner sep=0pt,minimum size=2mm,fill=white,thick] {};}}, postaction={decorate}]

\tikzstyle{stardot}=[decoration={markings, mark=at position 0.5 with {\node[fill,star, star points=6,star point ratio=0.5, inner sep=0pt,minimum size=1.5mm] {};}}, postaction={decorate}]

\tikzstyle{lowerstardot}=[decoration={markings, mark=at position 0.25 with {\node[fill,star, star points=6,star point ratio=0.5, inner sep=0pt,minimum size=1.5mm] {};}}, postaction={decorate}]

\tikzstyle{upperstardot}=[decoration={markings, mark=at position 0.75 with {\node[fill,star, star points=6,star point ratio=0.5, inner sep=0pt,minimum size=1.5mm] {};}}, postaction={decorate}]

\tikzstyle{ostardot}=[decoration={markings, mark=at position 0.5 with {\node[star, star points=6,star point ratio=0.5, inner sep=0pt,minimum size=1.7mm, fill,thick] {};
\node[star, star points=6,star point ratio=0.5, inner sep=0pt,minimum size=1.2mm, fill=white,thick] {};}
}, postaction={decorate}]

\tikzstyle{lowerostardot}=[decoration={markings, mark=at position 0.25 with {\node[fill=white,star, star points=6,star point ratio=0.5, inner sep=0pt,minimum size=1.5mm] {};}}, postaction={decorate}]

\tikzstyle{upperostardot}=[decoration={markings, mark=at position 0.75 with {\node[fill=white,star, star points=6,star point ratio=0.5, inner sep=0pt,minimum size=1.5mm] {};}}, postaction={decorate}]